\documentclass[11pt]{article}
\usepackage[margin=1in]{geometry}

\usepackage{amsmath, amssymb, amsfonts, amsthm}
\usepackage{mathtools}

\usepackage{graphicx}
\usepackage{subcaption}
\usepackage{adjustbox}
\usepackage{tikz-cd}
\usepackage{tikz}
\usetikzlibrary{patterns}

\usepackage{booktabs}
\usepackage{tabularx}
\usepackage{enumitem}

\usepackage{pifont}
\newcommand{\cmark}{\ding{51}}
\newcommand{\xmark}{\ding{55}}

\usepackage{hyperref}
\hypersetup{
  pdftitle={Learning Tangent Bundles and Characteristic Classes with Autoencoder Atlases},
  pdfauthor={E. Paluzo-Hidalgo, Y. Ike},
  colorlinks=true,
  linkcolor=blue,
  citecolor=blue,
  urlcolor=blue
}
\usepackage{cleveref}

\usepackage{appendix}
\newcommand{\R}{\mathbb{R}}
\newcommand{\Z}{\mathbb{Z}}

\newcommand{\Id}{\mathrm{Id}}
\newcommand{\sign}{\mathrm{sign}}

\newcommand{\GL}{\mathrm{GL}}

\theoremstyle{plain}
\newtheorem{theorem}{Theorem}[section]
\newtheorem{proposition}[theorem]{Proposition}
\newtheorem{lemma}[theorem]{Lemma}
\newtheorem{corollary}[theorem]{Corollary}

\theoremstyle{definition}
\newtheorem{definition}[theorem]{Definition}
\newtheorem{example}[theorem]{Example}

\theoremstyle{remark}
\newtheorem{remark}[theorem]{Remark}

\title{Learning Tangent Bundles and Characteristic Classes \\ with Autoencoder Atlases}

\author{
Eduardo Paluzo-Hidalgo\thanks{Corresponding author}\,\,\thanks{Department of Applied Mathematics I, University of Sevilla, Sevilla, Spain; Email: \href{mailto:epaluzo@us.es}{epaluzo@us.es}.}\,\,\thanks{Graduate School of Mathematical Sciences, The University of Tokyo, Tokyo, Japan.}
\and
Yuichi Ike\thanks{Graduate School of Mathematical Sciences, The University of Tokyo, Tokyo, Japan. Email: \href{mailto:ike@ms.u-tokyo.ac.jp}{ike@ms.u-tokyo.ac.jp}.}
}

\date{}

\begin{document}
\fontfamily{pbk}\selectfont
\maketitle

\begin{abstract}
We introduce a theoretical framework that connects multi-chart autoencoders in manifold learning with the classical theory of vector bundles and characteristic classes. Rather than viewing autoencoders as producing a single global Euclidean embedding, we treat a collection of locally trained encoder--decoder pairs as a learned atlas on a manifold. We show that any reconstruction-consistent autoencoder atlas canonically defines transition maps satisfying the cocycle condition, and that linearising these transition maps yields a vector bundle coinciding with the tangent bundle when the latent dimension matches the intrinsic dimension of the manifold. This construction provides direct access to differential-topological invariants of the data. In particular, we show that the first Stiefel--Whitney class can be computed from the signs of the Jacobians of learned transition maps, yielding an algorithmic criterion for detecting orientability. We also show that non-trivial characteristic classes provide obstructions to single-chart representations, and that the minimum number of autoencoder charts is determined by the good cover structure of the manifold. Finally, we apply our methodology to low-dimensional orientable and non-orientable manifolds, as well as to a non-orientable high-dimensional image dataset.
\end{abstract}



\section{Introduction}\label{sec:introduction}

A central goal of manifold learning is to recover, from samples in a high-dimensional ambient space, a low-dimensional representation that faithfully reflects the underlying geometry of the data. The dominant paradigm — embedding the data into Euclidean space via methods such as Isomap, locally linear embedding, diffusion maps, or, more recently, autoencoders — works well when the underlying manifold is topologically trivial, but encounters fundamental obstructions otherwise. The real projective plane $\mathbb{R}P^2$ does not embed in $\R^3$; the Klein bottle does not embed in $\R^3$ without self-intersection; more generally, no closed non-orientable surface embeds in $\R^3$ at all. Even when an embedding exists, Euclidean coordinates carry no direct record of the manifold's tangent bundle, transition functions, or characteristic classes — the very objects that classical differential topology uses to distinguish, for example, an orientable surface from a non-orientable one.

These limitations are not artefacts of any particular algorithm: they are obstructions to the embedding paradigm itself. Recent advances in intrinsic dimension and tangent-space estimation \cite{doi:10.1137/22M1522711} resolve the question of \emph{which} latent dimension to target, but not the deeper question of how to represent manifolds whose topology forbids global Euclidean coordinates.

\subsection*{From embeddings to atlases}

The classical resolution, going back to the foundations of differential topology, is to abandon the single-chart picture. A smooth manifold is not a Euclidean space but a collection of overlapping coordinate charts together with their transition maps, and the topological invariants of the manifold — orientability, characteristic classes, the tangent bundle itself — are encoded in these transitions. We translate this viewpoint into the language of neural networks. In place of a single autoencoder mapping into Euclidean space, we work with an \emph{autoencoder atlas}: a collection of locally trained encoder--decoder pairs $\{(U_i, E_i, D_i)\}$ with overlapping domains, joined by the natural transition maps $T_{ji} = E_j \circ D_i$ that send a latent code in chart~$i$ to the corresponding code in chart~$j$. The atlas itself, not any single global embedding, becomes the learned object.

This shift opens a direct line from learned representations to differential-topological invariants. The transition maps inherit the cocycle condition from reconstruction consistency, so their Jacobians assemble into a vector bundle $\mathcal{T}_\mathcal{A}$ over $M$ — agreeing with $TM$ when the encoders are compatible with the smooth structure. From the cocycle one extracts the first Stiefel--Whitney class $w_1(\mathcal{T}_\mathcal{A})$ as the sign cocycle of the Jacobian determinants, giving an algorithmic test for orientability: the manifold is orientable if and only if this sign cocycle is a \v{C}ech coboundary, a condition that can be checked in linear time on the nerve graph. Non-orientable manifolds that cannot be faithfully embedded in low-dimensional Euclidean space are nonetheless represented faithfully by an atlas, and their non-orientability is detected directly from the trained model.

The main mathematical work of the paper is to make this picture rigorous in the practical setting where autoencoders achieve only approximate reconstruction. We prove that the $\Z/2$-valued sign cocycle is stable under explicit bounds on reconstruction error and Jacobian non-degeneracy, that its cohomology class equals $w_1(TM)$, and that the minimum number of charts in such an atlas is governed by a classical topological invariant — the covering type of $M$ in the sense of Karoubi and Weibel \cite{Karoubi2017} — rather than by the structure of the tangent bundle.

\subsection*{Contributions}

\begin{enumerate}
\item \textbf{Autoencoder atlases as data-driven differential structures} (Section~\ref{sec:atlasbasedformulation}). We formalize the notion of an exact and an approximate \emph{autoencoder atlas} $\mathcal{A} = \{(U_i, E_i, D_i)\}_{i \in I}$, show that reconstruction consistency canonically determines transition maps satisfying the cocycle condition, and construct a vector bundle $\mathcal{T}_\mathcal{A}$ over $M$ from the linearised transitions. This bundle agrees with $TM$ when the encoders form a smooth atlas compatible with the manifold structure.

\item \textbf{Orientability detection from learned transitions} (Section~\ref{sec:atlasbasedformulation}). The Jacobian sign cocycle $\omega_{ji}(x) = \sign(\det g_{ji}(x))$ represents $w_1(\mathcal{T}_\mathcal{A}) \in H^1(M; \Z/2)$. Orientability is decided by checking whether $\{\omega_{ji}\}$ is a \v{C}ech coboundary, a problem reducing to 2-colouring the nerve graph.

\item \textbf{Stability under approximate reconstruction} (Section~\ref{sec:stability}). Neural network training achieves only approximate reconstruction; we prove that the sign cocycle is nonetheless a valid \v{C}ech cocycle under explicit bounds on reconstruction error and Jacobian non-degeneracy (Theorem~\ref{thm:sign-cocycle-stability}). A per-triple refinement (Theorem~\ref{thm:local-stability}) shows that the differential reconstruction condition needs to hold on only two of the three charts of each nerve triangle — a strictly weaker hypothesis that explains why detection succeeds in regimes where the global bounds are violated. A homotopy argument (Theorem~\ref{thm:w1-agreement}) identifies the learned cohomology class with $w_1(TM)$.

\item \textbf{Topological lower bounds on chart count} (Section~\ref{sec:obstruction}). Non-trivial characteristic classes obstruct single-chart representations, and the minimum number of autoencoder charts whose underlying cover is good equals the covering type of $M$ \cite{Karoubi2017}. The bound depends on the homotopy type of $M$, not on the structure of the tangent bundle.
\end{enumerate}

Section~\ref{sec:experiments} validates the framework on four test manifolds — the $2$-sphere, the M\"obius band, the Klein bottle in $\R^4$, and $\mathbb{R}P^2$ represented as line-patch images in $\R^{100}$ — confirming correct orientability detection and the diagnostic role of the per-triple stability bounds.

\subsection*{Related work}

Our approach sits at the intersection of three lines of work.

\emph{Topological coordinates from persistent cohomology.}
A line of work initiated by Perea \cite{10.1007/s00454-017-9927-2,DBLP:conf/compgeom/ScoccolaP23,Scoccola_Perea_2023} uses persistent cohomology to construct \v{C}ech cocycles from overlapping local coordinate charts (typically obtained via multidimensional scaling), and resolves transition functions through orthogonal Procrustes alignment. The resulting multiscale projective coordinates detect non-orientability via sign cocycles of essentially the same form as ours, and produce explicit classifying maps into $\mathbb{R}P^d$. The key methodological difference is the source of the charts: Perea's framework constructs coordinates by linear-algebraic alignment of pre-specified local data, whereas ours obtains them as solutions of a non-linear optimisation problem, the autoencoder training objective. Related work computes persistent Stiefel--Whitney classes from discrete differential geometry on triangulated manifolds \cite{Tinarrage_2021}, but this requires explicit combinatorial manifold structure that is unavailable from raw point cloud data. Methods for identifying representative cycles of homological features \cite{10.1109/SMI.2007.25,yiding_2009} are NP-hard \cite{Chen2011}, and so do not scale to high-dimensional data. Theoretical foundations for bundle structures in discrete and approximate settings appear in \cite{Scoccola_Perea_2023,scoccola2023fiberedfiberwisedimensionalityreduction}.

\emph{Chart autoencoders.}
Schonsheck et al.\ \cite{schonsheck2020chartautoencodersmanifoldstructured,schonsheck2024semisupervisedmanifoldlearningcomplexity} introduced multi-chart autoencoder architectures and proved that multi-chart latent spaces are necessary for topologically non-trivial manifolds. Their work establishes the necessity of charts and develops finite-distortion conditions, but treats the atlas as a vehicle for global reconstruction. We take the atlas itself as the object of interest and extract differential-topological invariants from its transition structure. The topological autoencoder \cite{11301037} incorporates persistent-homology losses to preserve topological features in the latent representation, an approach complementary to but distinct from ours; we obtain topological information post-hoc from the transition cocycle, without homology computation in the training loop. Concurrent work \cite{robinett2026atlasbased,yu2026learninggeometrytopologymultichart} develops Riemannian geometric primitives across multi-chart representations but does not extract bundle-theoretic invariants from the cocycle.

\emph{Cover learning.}
The good-cover hypothesis that underlies our cohomological computations is, in practice, learned from data. Scoccola, Lim, and Harrington \cite{scoccola2025coverlearninglargescaletopology} address this problem directly, building on classical methods such as Mapper \cite{singhMC07,doi:10.1137/24M1641312}. Our experiments adopt simple landmark-based covers as a practical heuristic and we treat principled cover learning as orthogonal to our contribution.

\subsection*{Organisation}

Section~\ref{sec:background} recalls the necessary background on \v{C}ech cohomology, vector bundles, and the first Stiefel--Whitney class. Section~\ref{sec:atlasbasedformulation} introduces autoencoder atlases, proves the cocycle condition from reconstruction consistency, constructs the learned tangent bundle, and states the orientability detection criterion. Section~\ref{sec:stability} develops the stability theory for approximate reconstruction, in both global and per-triple forms. Section~\ref{sec:obstruction} analyses topological lower bounds on chart count. Section~\ref{sec:loss-functions} describes the training losses. Section~\ref{sec:experiments} presents experimental validation.

\section{Background}\label{sec:background}
In this section, we give the mathematical foundations necessary for our main results. 
We begin with essential topological concepts, then turn to characteristic classes of vector bundles, focusing on the first Stiefel--Whitney class, which detects orientability.
 
\subsection{Topological preliminaries}
 
We recall basic notions from algebraic topology that underlie our construction of characteristic classes from autoencoder atlases.
 
We relax the classic definition of a good cover, allowing multiple connected components at intersections. 
\begin{definition}[Good cover]\label{def:good-cover}
An open cover $\mathcal{U} = \{U_\alpha\}_{\alpha\in A}$ of a topological space $X$ with $A$ a finite index set is called a \emph{good cover} if:
\begin{enumerate}
\item Each $U_\alpha$ is contractible (homotopy equivalent to a point).
\item Every connected component of every non-empty finite intersection $U_{\alpha_0} \cap \cdots \cap U_{\alpha_p}$ is contractible, with $p\in\{1,\dots,\vert A\vert-1\}$. 
\end{enumerate}
\end{definition} 
 
Good covers are fundamental in reflecting the topology faithfully
via the \v{C}ech complex of the cover, as stated in the well-known
Nerve theorem~\cite{hatcher2002algebraic} (see Remark~\ref{rem:relaxed-nerve}). The minimum cardinality of a good cover is itself a topological invariant, studied by Karoubi and Weibel~\cite{Karoubi2017} under the name
\emph{covering type}; we return to this quantity in Section~\ref{sec:obstruction}, where it determines the minimum
number of autoencoder charts. When the cover is not given a priori, it must be learned from the data. Recent work~\cite{scoccola2025coverlearninglargescaletopology} identifies and addresses the cover problem, as well as studying
state-of-the-art methods such as Mapper~\cite{singhMC07,doi:10.1137/24M1641312}.
 
\begin{figure}[h!]
    \centering
\includegraphics[width=0.4\linewidth]{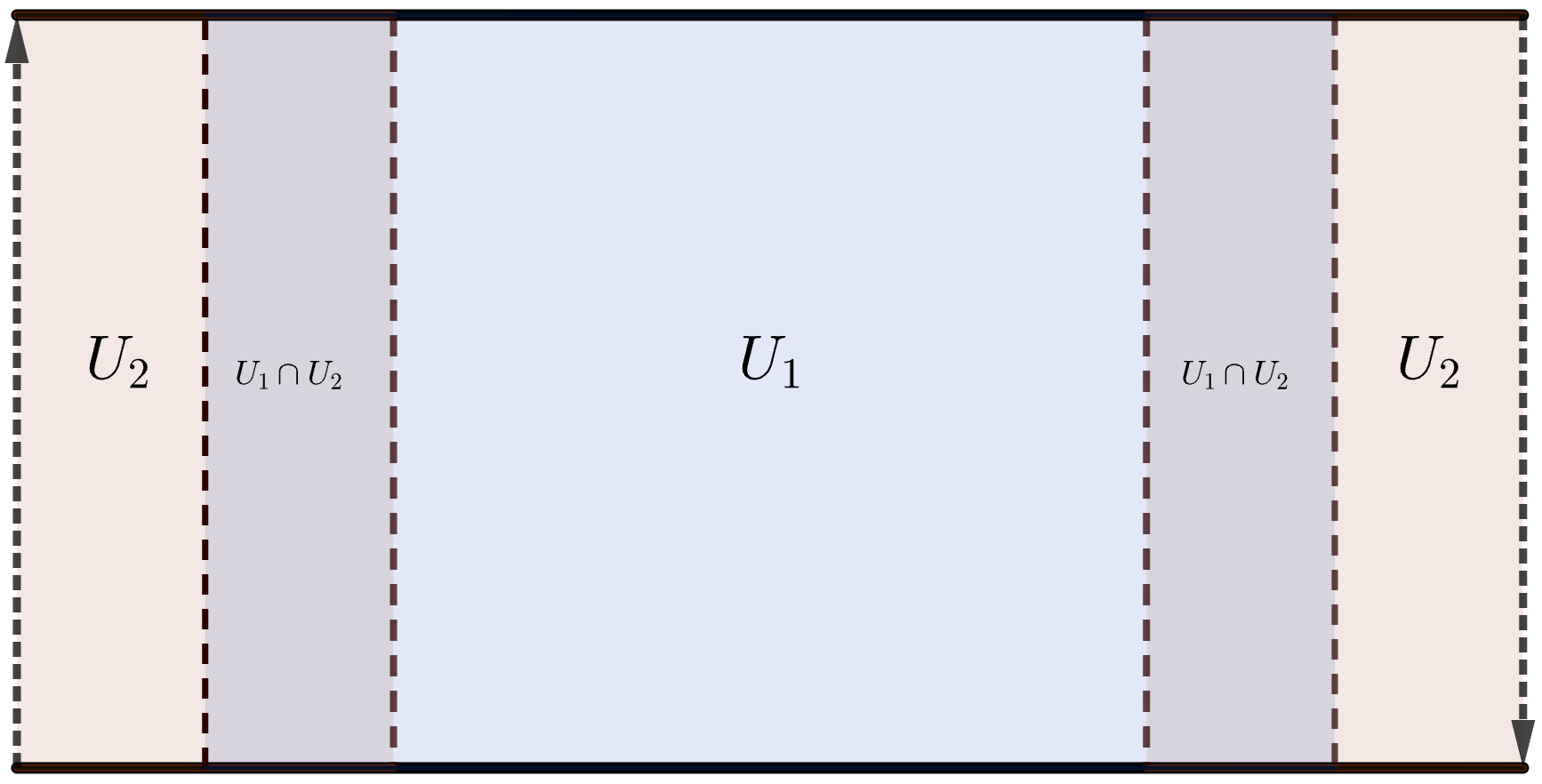}
\includegraphics[width=0.4\linewidth]{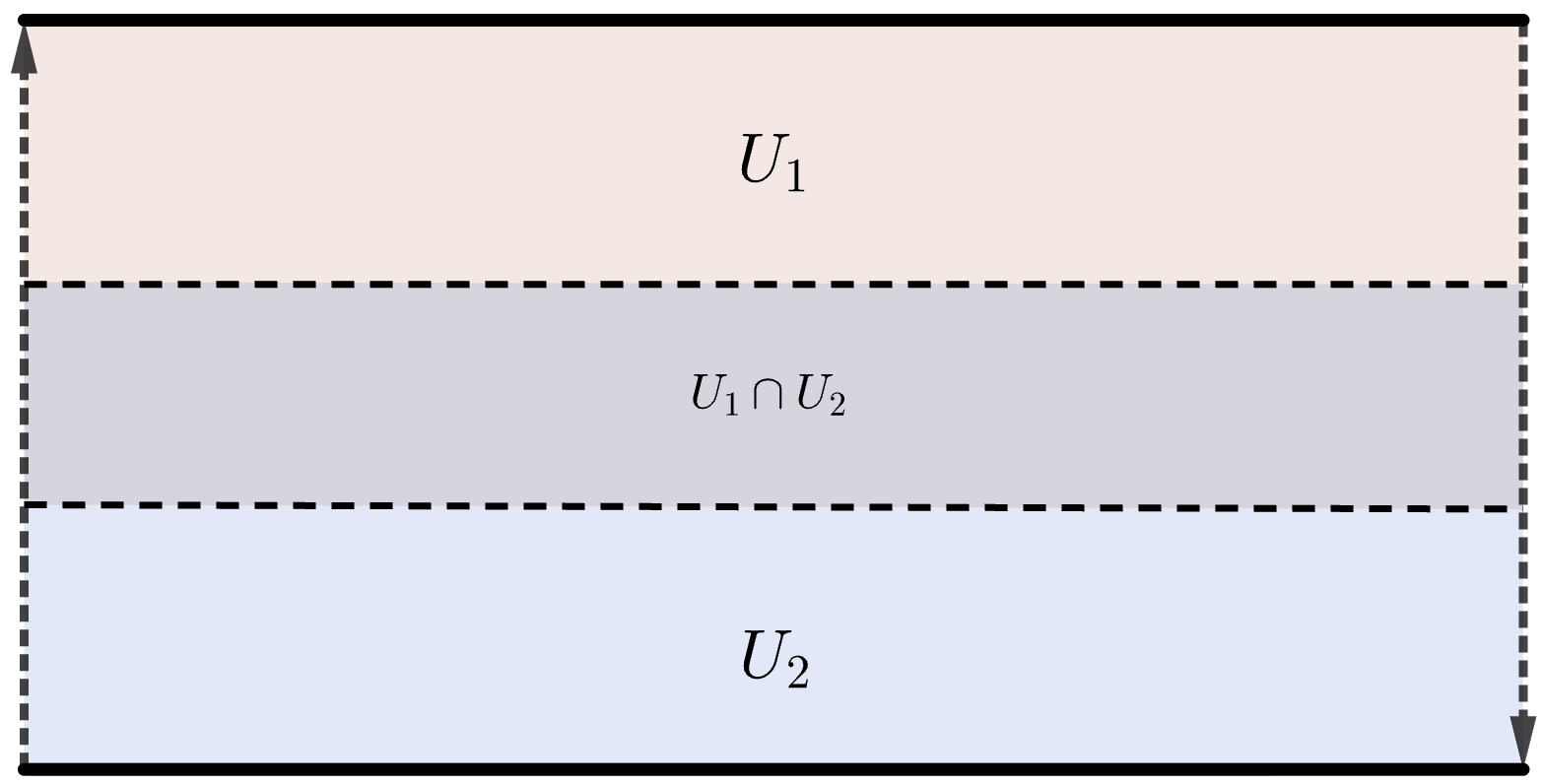}
        \caption{Two different M\"obius covers are shown. On the left, we have a good cover such that the intersections are contractible. However, on the right, we have that the intersection is another M\"obius band and hence, not contractible. }
    \label{fig:coverexample}
\end{figure}
 
\subsubsection*{\v{C}ech cohomology with $\Z/2$-coefficients}
 
For detecting orientability we need the first \v{C}ech cohomology group $\check{H}^1(\mathcal{U};\Z/2)$. We recall the general construction specialised to the case that matters for us.
 
Throughout, we identify $\Z/2$ with the multiplicative group $\{\pm 1\}$ via $0\mapsto +1$, $1\mapsto -1$. This is natural in our setting because the relevant cocycle values will arise as signs of determinants.
 
\begin{definition}[\v{C}ech cochains, cocycles, and cohomology]\label{def:cech-cohomology}
Let $\mathcal{U} = \{U_\alpha\}_{\alpha \in A}$ be an open cover of a topological space $X$, with $A$ a finite ordered index set. For $p\geq 0$, a \emph{\v{C}ech $p$-cochain} with values in $\{\pm 1\}$ assigns to each ordered $(p+1)$-tuple $\alpha_0<\cdots<\alpha_p$ with $U_{\alpha_0}\cap\cdots\cap U_{\alpha_p}\neq\emptyset$ a locally constant function $U_{\alpha_0}\cap\cdots\cap U_{\alpha_p}\to\{\pm 1\}$.
When each connected component of each intersection is contractible (as for a good cover), locally constant means constant on each component, and a $p$-cochain amounts to an assignment of a sign $\pm 1$ to each connected component of each non-empty $(p+1)$-fold intersection.
 
The \emph{coboundary} $\delta^p$ maps $p$-cochains to $(p+1)$-cochains; the key properties are $\delta^{p+1}\circ\delta^p=0$ and the resulting \v{C}ech cohomology groups $\check{H}^p(\mathcal{U};\Z/2)\coloneqq\ker\delta^p/\operatorname{im}\delta^{p-1}$.
\end{definition}
 
\begin{definition}[1-cocycles and 1-coboundaries]\label{def:Z2-cocycles}
A \emph{1-cocycle} is a collection $\{\omega_{\alpha\beta}\}_{\alpha<\beta}$ with $\omega_{\alpha\beta}\in\{\pm 1\}$ (one value per connected component of each non-empty $U_\alpha\cap U_\beta$) satisfying the \emph{cocycle condition}: for every non-empty triple intersection $U_\alpha\cap U_\beta\cap U_\gamma$ with $\alpha<\beta<\gamma$, one has $\omega_{\alpha\gamma}=\omega_{\alpha\beta}\cdot\omega_{\beta\gamma}$.
A 1-cocycle is a \emph{coboundary} if there exist $\nu_\alpha\in\{\pm 1\}$ (constant on each connected $U_\alpha$) such that $\omega_{\alpha\beta}=\nu_\alpha\cdot\nu_\beta$ for all $\alpha<\beta$ with $U_\alpha\cap U_\beta\neq\emptyset$. The first \v{C}ech cohomology group is
\[
\check{H}^1(\mathcal{U};\Z/2)=\frac{\{1\text{-cocycles}\}}{\{1\text{-coboundaries}\}}.
\]
\end{definition}
 
\begin{theorem}[Isomorphism for good covers {\cite[Theorem 8.9]{BottTu}}]\label{thm:good-cover}
If $\mathcal{U}$ is a good cover of a paracompact Hausdorff space $X$, then
$\check{H}^p(\mathcal{U};\Z/2)\cong H^p(X;\Z/2)$,
where $H^p(X;\Z/2)$ denotes the singular cohomology of $X$ with $\Z/2$-coefficients.
\end{theorem}

\begin{remark}[Nerve theorem for relaxed good covers]
\label{rem:relaxed-nerve}
Definition~\ref{def:good-cover} allows intersections to have multiple connected components, each required to be contractible. The classical Nerve Theorem still applies via refinement.
 
Given a cover $\mathcal{U} = \{U_\alpha\}_{\alpha}$ satisfying Definition~\ref{def:good-cover}, construct a refined cover $\mathcal{U}'$ by splitting each intersection $U_{\alpha_0}\cap\cdots\cap U_{\alpha_p}$ into its connected components. Each component of $\mathcal{U}'$ is contractible by assumption, and $\mathcal{U}'$ is therefore a standard good cover in the sense of \cite[Sec.~II.15]{BottTu}.
The refinement map $\mathcal{U}' \to \mathcal{U}$ induces an isomorphism on \v{C}ech cohomology. Hence, characteristic classes computed using the relaxed cover agree with those obtained from a standard good cover. In practice, cocycle computations treat each connected component of each overlap as a separate element.
\end{remark}
 
\subsection{Vector bundles and transition functions}\label{sec:vector-bundles-background}
 
We assume familiarity with real vector bundles (see \cite{milnor1974characteristic,Zakharevich2024} for textbook treatments). Here we recall the notions used directly in our constructions.
 
A rank-$k$ real vector bundle $\pi\colon E\to B$ is locally trivial: around each point of $B$ there is a neighbourhood $U$ and a fiber-preserving homeomorphism $\phi_U\colon\pi^{-1}(U)\xrightarrow{\;\sim\;} U\times\R^k$ that is linear on each fiber. When two trivializations $(U_\alpha,\phi_\alpha)$ and $(U_\beta,\phi_\beta)$ overlap, they differ by a \emph{transition function}
\[
g_{\alpha\beta}\colon U_\alpha\cap U_\beta\to\GL(k,\R), \qquad
\phi_\alpha\circ\phi_\beta^{-1}(b,v)=(b,\,g_{\alpha\beta}(b)\cdot v).
\]
On triple overlaps $U_\alpha\cap U_\beta\cap U_\gamma$, these satisfy the \emph{cocycle condition}
\[
g_{\alpha\gamma}(b)=g_{\alpha\beta}(b)\cdot g_{\beta\gamma}(b).
\]
 
Conversely, transition functions satisfying the cocycle condition completely determine a vector bundle up to isomorphism.
 
\begin{theorem}[Construction from transition functions {\cite[Example 2.19]{Zakharevich2024}}]
\label{thm:bundle-from-transitions}
Given an open cover $\{U_\alpha\}_{\alpha \in A}$ of $B$ and continuous maps 
$g_{\alpha\beta}\colon U_\alpha \cap U_\beta \to \GL(k,\R)$
satisfying $g_{\alpha\alpha}=\Id$, $g_{\alpha\beta}=g_{\beta\alpha}^{-1}$, and the cocycle condition $g_{\alpha\gamma}=g_{\alpha\beta}\cdot g_{\beta\gamma}$, there exists a rank-$k$ vector bundle $\pi \colon E \to B$ with these transition functions. The total space is
\[
E = \left( \bigsqcup_{\alpha \in A} U_\alpha \times \R^k \right)\Big/ \sim
\]
where $(b, v)_\alpha \sim (b, g_{\beta\alpha}(b) \cdot v)_\beta$ for $b \in U_\alpha \cap U_\beta$.
\end{theorem}
 
In our autoencoder framework (Section~\ref{sec:tangent-bundle}), we construct vector bundles by obtaining transition functions from learned coordinate charts and then applying this theorem.

\begin{example}[M\"obius band {\cite[Example 2.19]{Zakharevich2024}}]\label{ex:mobius-bundle}
The M\"obius band is a rank-$1$ vector bundle (line bundle) over $S^1$. Cover $S^1$ by two overlapping arcs $U_1$ and $U_2$ whose intersection consists of two disjoint intervals $I_+$ and $I_-$ (as shown in Figure~\ref{fig:coverexample}). The transition function is
\[
g_{12}(x) = \begin{cases} 
-1 & \text{if } x \in I_- \\
+1 & \text{if } x \in I_+
\end{cases}
\]
This single sign flip distinguishes the M\"obius band from the cylinder $S^1 \times \R$.
\end{example}
 
\subsection{The first Stiefel--Whitney class and orientability}\label{sec:characteristic-classes}
 
Let $\xi$ be a real vector bundle of rank $k$ over a base space $B$. The Stiefel--Whitney classes $w_i(\xi)\in H^i(B;\Z/2)$, $i=0,\dots,k$, are cohomological invariants that depend only on the isomorphism class of $\xi$ and are characterised uniquely by standard axioms~\cite{milnor1974characteristic}. The first class $w_1(\xi)$ is the obstruction to orientability:

\begin{theorem}[Orientability and $w_1$ {\cite[Proposition~4.2]{milnor1974characteristic}}]\label{thm:orientability-w1}
A real vector bundle $\xi$ over $B$ is \emph{orientable} (i.e., its structural group reduces from $\GL_k(\R)$ to $\GL_k^+(\R)$) if and only if $w_1(\xi) = 0$.
\end{theorem}
 
The class $w_1$ admits a concrete description via transition functions. Given transition functions $g_{\alpha\beta}\colon U_\alpha\cap U_\beta\to\GL(k,\R)$ relative to an open cover $\{U_\alpha\}_\alpha$ of $B$, define the \emph{sign cocycle}
\[
\omega_{\alpha\beta}(x)\coloneqq\operatorname{sign}\!\bigl(\det g_{\alpha\beta}(x)\bigr)\in\{\pm 1\}\cong\Z/2.
\]
By multiplicativity of the determinant, $\{\omega_{\alpha\beta}\}_{\alpha,\beta}$ inherits the cocycle condition $\omega_{\alpha\gamma}=\omega_{\alpha\beta}\cdot\omega_{\beta\gamma}$ from $\{g_{\alpha\beta}\}_{\alpha,\beta}$.
 
\begin{theorem}[Computation of $w_1$ {\cite[p.~148]{milnor1974characteristic}}]
\label{thm:w1-sign}
If the cover is good, the cohomology class $[\omega_{\alpha\beta}]\in\check{H}^1(\{U_\alpha\};\Z/2)$ coincides with $w_1(\xi)\in H^1(B;\Z/2)$.
\end{theorem}
 
Combining the orientability criterion with Theorem~\ref{thm:w1-sign} gives an algorithmic test.
 
\begin{corollary}[Orientability via the sign cocycle]
\label{cor:orientability-sign}
Let $\xi$ be a real vector bundle over $B$ with transition functions $\{g_{\alpha\beta}\}_{\alpha,\beta}$ and associated sign cocycle $\{\omega_{\alpha\beta}\}_{\alpha,\beta}$, relative to a good cover. Then $\xi$ is orientable if and only if $\{\omega_{\alpha\beta}\}_{\alpha,\beta}$ is a \v{C}ech coboundary, i.e., there exist locally constant functions $\nu_\alpha\colon U_\alpha\to\{\pm 1\}$ such that
$\omega_{\alpha\beta}(x)=\nu_\alpha(x)\cdot\nu_\beta(x)$ for all $x\in U_\alpha\cap U_\beta$.
\end{corollary}
 
A worked example illustrating this criterion on the M\"obius band is given in Example~\ref{ex:mobius-nonorientable} below.

\begin{example}[M\"obius band is non-orientable]\label{ex:mobius-nonorientable}
Continuing Example~\ref{ex:mobius-bundle}, the sign cocycle is $\omega^1_{12}(x)=+1$ on $I_+$ and $\omega^2_{12}(x)=-1$ on $I_-$. For this to be a coboundary, we would need $\nu_1,\nu_2\in\{+1,-1\}$ (each constant on the connected sets $U_1,U_2$) with $\omega_{12}(x)=\nu_1\cdot\nu_2$ for all $x$. This is impossible: one component requires $\nu_1\cdot\nu_2=-1$ while the other requires $\nu_1\cdot\nu_2=+1$. Therefore $w_1\neq 0$ and the M\"obius band is non-orientable. This illustrates the essential mechanism: non-orientability manifests as an obstruction to finding a consistent global sign assignment, detected by a nontrivial cohomology class.
\end{example}

\section{Atlas-Based Formulation of the Orientability Obstruction}\label{sec:atlasbasedformulation}

\subsection{Autoencoder atlases}\label{sec:autoencoders}

We work throughout with a smooth compact $d$-dimensional manifold $M$ without boundary, embedded in~$\R^N$. This is the natural setting for manifold learning, where data points lie in a high-dimensional ambient space and the manifold structure is to be discovered. The embedding provides a concrete ambient space in which encoders and decoders operate; in particular, tangent spaces and norms are inherited from~$\R^N$.

In their exact form (Definitions~\ref{def:autoencoder-chart} and~\ref{def:smooth-chart} below), autoencoder atlases are equivalent to classical smooth atlases: a smooth autoencoder chart is simply a coordinate chart together with its inverse playing the role of decoder. We introduce this reformulation because it admits a natural approximate generalization (Section~\ref{sec:approximateautoencoders}): replacing the exact reconstruction condition $D_i \circ E_i = \mathrm{Id}_{U_i}$ by an $\varepsilon$-approximate one yields the approximate autoencoder atlas, which is the object that arises from neural network training and for which the stability theory of Section~\ref{sec:stability} applies.

As motivated in~\cite{schonsheck2020chartautoencodersmanifoldstructured}, multi-chart latent spaces are necessary for representing manifolds with non-trivial topology. A single global coordinate system cannot exist for manifolds such as the sphere or projective plane; instead, one must work with a collection of local coordinate charts. This will be extended in Section~\ref{sec:obstruction}.

\begin{definition}[Autoencoder chart]\label{def:autoencoder-chart}
An \emph{autoencoder chart} on a smooth manifold $M \subset \R^N$ is a triple $(U, E, D)$ where:
\begin{enumerate}
\item $U \subset M$ is open;
\item $E \colon U \to Z \subset \R^d$ is a continuous injective map (the \emph{encoder});
\item $D \colon Z \to M$ is continuous (the \emph{decoder});
\item $D \circ E = \mathrm{Id}_{U}$ (the \emph{reconstruction condition}).
\end{enumerate}
\end{definition}

In practice, the reconstruction condition is approximated through optimisation (see Remark~\ref{rem:idealization}).

\begin{definition}[Smooth autoencoder chart]\label{def:smooth-chart}
A \emph{smooth autoencoder chart} is an autoencoder chart $(U, E, D)$ (Definition~\ref{def:autoencoder-chart}) in which $E$ and $D$ are smooth ($C^\infty$) and $E \colon U \to Z$ is a diffeomorphism onto its image. In particular, $D|_Z = E^{-1}$, so a smooth autoencoder chart is precisely a smooth coordinate chart $(U, E)$ in the classical sense, equipped with its inverse as a decoder.
\end{definition}

\begin{remark}[$C^r$ autoencoder charts]\label{rem:Cr-chart}
More generally, for $r \geq 1$, a \emph{$C^r$ autoencoder chart} is an autoencoder chart in which $E$ and $D$ are $C^r$ and $E \colon U \to Z$ is a $C^r$-diffeomorphism onto its image. A \emph{$C^r$ autoencoder atlas} is an autoencoder atlas consisting of $C^r$ autoencoder charts. The stability results of Section~\ref{sec:stability} require only $r = 1$.
\end{remark}

\begin{definition}[Autoencoder atlas]\label{def:autoencoder-atlas}
An \emph{autoencoder atlas} on $M$ is a collection
\[
\mathcal{A} = \{(U_i, E_i, D_i)\}_{i \in I}
\]
of autoencoder charts such that $\mathcal{U} = \{U_i\}_{i \in I}$ is a cover of $M$.
\end{definition}

\begin{figure}[h!]
    \centering
    \begin{tikzpicture}
    % Functions i
    \path[->] (0.8, 0) edge [bend right] node[left, xshift=-2mm] {$E_i$} (-1, -2.9);
    \draw[white,fill=white] (0.15,-0.4) circle (.1cm);
    \path[->] (-0.7, -3.05) edge [bend right] node [right, yshift=-3mm] {$D_i$} (1.093, -0.11);
    \draw[white, fill=white] (1.1,-0.7) circle (.1cm);
    % Functions j
    \path[->] (5.8, -2.8) edge [bend left] node[midway, xshift=-5mm, yshift=-3mm] {$D_j$} (3.8, -0.35);
    \draw[white, fill=white] (3.8,-0.6) circle (.1cm);
    \path[->] (4.2, 0) edge [bend left] node[right, xshift=2mm] {$E_j$} (6.2, -2.8);
    \draw[white, fill=white] (4.4,-0.03) circle (.1cm);
    % Manifold
    \draw[smooth cycle, tension=0.5, fill=white, pattern color=brown, pattern=north west lines, opacity=0.7] plot coordinates{(2,2) (-0.5,0) (3,-1) (5,1)} node at (3,2.3) {$M$};
    % Subsets
    \draw[smooth cycle, pattern color=orange, pattern=crosshatch dots]
        plot coordinates {(1,0) (1.5, 1.2) (2.5,1.3) (2.6, 0.4)}
        node [label={[label distance=-0.3cm, xshift=-2cm, fill=white]:$U_i$}] {};
    \draw[smooth cycle, pattern color=blue, pattern=crosshatch dots]
        plot coordinates {(4, 0) (3.7, 0.8) (3.0, 1.2) (2.5, 1.2) (2.2, 0.8) (2.3, 0.5) (2.6, 0.3) (3.5, 0.0)}
        node [label={[label distance=-0.8cm, xshift=.75cm, yshift=1cm, fill=white]:$U_j$}] {};
    \draw[thick, ->] (-3,-5) -- (0.2, -5) node [label=above:{$Z_i \coloneqq E_i(U_i)$}] {};
    \draw[thick, ->] (-3,-5) -- (-3, -2) node [label=right:$\R^d$] {};
    \draw[->] (0, -3.85) -- node[midway, above]{$T_{ji} \coloneqq E_j\circ D_i$} (4.5, -3.85);
    \draw[thick, ->] (5, -5) -- (8.1, -5) node [label=above:{$Z_j \coloneqq E_j(U_j)$}] {};
    \draw[thick, ->] (5, -5) -- (5, -2) node [label=right:$\R^d$] {};
    \draw[white, pattern color=orange, pattern=crosshatch dots] (-0.67, -3.06) -- +(180:0.8) arc (180:270:0.8);
    \fill[even odd rule, white] [smooth cycle] plot coordinates{(-2, -4.5) (-2, -3.2) (-0.8, -3.2) (-0.8, -4.5)} (-0.67, -3.06) -- +(180:0.8) arc (180:270:0.8);
    \draw[smooth cycle] plot coordinates{(-2, -4.5) (-2, -3.2) (-0.8, -3.2) (-0.8, -4.5)};
    \draw (-1.45, -3.06) arc (180:270:0.8);
    \draw[white, pattern color=blue, pattern=crosshatch dots] (5.7, -3.06) -- +(-90:0.8) arc (-90:0:0.8);
    \fill[even odd rule, white] [smooth cycle] plot coordinates{(7, -4.5) (7, -3.2) (5.8, -3.2) (5.8, -4.5)} (5.7, -3.06) -- +(-90:0.8) arc (-90:0:0.8);
    \draw[smooth cycle] plot coordinates{(7, -4.5) (7, -3.2) (5.8, -3.2) (5.8, -4.5)};
    \draw (5.69, -3.85) arc (-90:0:0.8);
\end{tikzpicture}
    \caption{Diagram of an atlas autoencoder. Given charts $(U_i, E_i, D_i)$ and $(U_j, E_j, D_j)$ with $U_i \cap U_j \neq \emptyset$, the transition map $T_{ji} = E_j \circ D_i$ converts latent representations between charts.}
    \label{fig:atlasautoencoder}
\end{figure}

\begin{remark}\label{rem:idealization}
The definitions above describe the \emph{ideal} mathematical structure: exact reconstruction $D_i \circ E_i = \mathrm{Id}_{U_i}$, smooth maps, and precise chart domains. In practice, these conditions are approximated through optimisation. The reconstruction condition becomes a \emph{loss function} to minimise, and the theoretical guarantees hold in the limit of perfect training.
\end{remark}

\subsubsection{Transition maps and cocycle condition}

The fundamental insight is that autoencoder atlases naturally give rise to transition maps defined from the encoding and decoding functions \cite{schonsheck2020chartautoencodersmanifoldstructured}.

\begin{definition}[Transition maps]
For an autoencoder atlas $\mathcal{A} = \{(U_i, E_i, D_i)\}$ with $U_i \cap U_j \neq \varnothing$, the \emph{transition map} from chart $i$ to chart $j$ is
\[
T_{ji} = E_j \circ D_i \colon E_i(U_i \cap U_j) \to E_j(U_i \cap U_j).
\]
\end{definition}

The transition map $T_{ji}$ converts a latent representation in chart $i$ to the corresponding representation in chart $j$: it decodes from chart $i$'s latent space back to the manifold, then encodes into chart $j$'s latent space (see Figure~\ref{fig:atlasautoencoder}).

\begin{remark}[Well-definedness of transition maps]\label{rem:transition-well-defined}
The composition $E_j\circ D_i$ is well defined on $E_i(U_i\cap U_j)$. If $z=E_i(x)$ with $x\in U_i\cap U_j$, then $D_i(z)=D_i(E_i(x))=x\in U_j$, so $E_j(D_i(z))$ is defined.
\end{remark}

\begin{proposition}\label{prop:transition-properties}
The transition maps satisfy:
\begin{enumerate}
\item $T_{ij} \circ T_{ji} = \mathrm{Id}$ on $E_i(U_i \cap U_j)$;
\item if $\mathcal{A}$ consists of smooth autoencoder charts, then each $T_{ji}$ is a diffeomorphism.
\end{enumerate}
\end{proposition}

\begin{proof}
Let $x \in U_i \cap U_j$. We compute
\begin{align*}
T_{ij}(T_{ji}(E_i(x))) &= E_i(D_j(E_j(D_i(E_i(x))))) \\
&= E_i(D_j(E_j(x))) \\
&= E_i(x).
\end{align*}
Hence $T_{ij} \circ T_{ji} = \mathrm{Id}$ on $E_i(U_i \cap U_j)$. Similarly $T_{ji} \circ T_{ij} = \mathrm{Id}$ on $E_j(U_i \cap U_j)$, so $T_{ji}$ is a bijection with inverse $T_{ij}$.

For smoothness: if $E_i, D_i, E_j, D_j$ are all smooth, then $T_{ji} = E_j \circ D_i$ is smooth as a composition of smooth maps. Since $T_{ji}$ is a smooth bijection with smooth inverse $T_{ij}$, it is a diffeomorphism.
\end{proof}

\begin{lemma}[Cocycle condition from reconstruction]\label{lemma:cocycle-from-reconstruction}
The transition maps of an autoencoder atlas satisfy the cocycle condition: for all $i, j, k$ with $U_i \cap U_j \cap U_k \neq \emptyset$,
\[
T_{ki} = T_{kj} \circ T_{ji} \quad \text{on } E_i(U_i \cap U_j \cap U_k).
\]
\end{lemma}

\begin{proof}
Let $x \in U_i \cap U_j \cap U_k$. Both sides agree when evaluated at $E_i(x)$:
\begin{align*}
T_{ki}(E_i(x)) &= E_k(D_i(E_i(x))) = E_k(x), \\
T_{kj}(T_{ji}(E_i(x))) &= T_{kj}(E_j(D_i(E_i(x)))) = T_{kj}(E_j(x)) \\
&= E_k(D_j(E_j(x))) = E_k(x). \qedhere
\end{align*}
\end{proof}

\subsection{The tangent bundle of an autoencoder atlas}\label{sec:tangent-bundle}

To access the theory of characteristic classes, we construct a vector bundle from our autoencoder atlas. We do this by linearising the non-linear transition maps.

\subsubsection{Linearisation}

\begin{definition}[Linearised transition functions]
Let $\mathcal{A} = \{(U_i, E_i, D_i)\}$ be a smooth autoencoder atlas. For $x \in U_i \cap U_j$, define the \emph{linearised transition function}
\[
g_{ji}(x) \coloneqq d(T_{ji})_{E_i(x)} = \frac{\partial(E_j \circ D_i)}{\partial z}\Big|_{z=E_i(x)} \in \GL_d(\R),
\]
the Jacobian matrix of the transition map $T_{ji}$ evaluated at the latent code $E_i(x)$.
\end{definition}

The linearised transition function describes how infinitesimal perturbations in chart $i$'s latent space transform to chart $j$'s latent space.

\begin{lemma}\label{lemma:linearised-cocycle}
The linearised transition functions $\{g_{ji}\}$ satisfy:
\begin{enumerate}
\item $g_{ii}(x) = \mathrm{Id}$ for all $x \in U_i$;
\item $g_{ij}(x) = g_{ji}(x)^{-1}$ for all $x \in U_i \cap U_j$;
\item the cocycle condition $g_{ki}(x) = g_{kj}(x) \cdot g_{ji}(x)$ for all $x \in U_i \cap U_j \cap U_k$.
\end{enumerate}
\end{lemma}

\begin{proof}
Apply the chain rule to the corresponding identities for the non-linear transition maps from Proposition~\ref{prop:transition-properties} and Lemma~\ref{lemma:cocycle-from-reconstruction}.

(1) $g_{ii}(x) = d(T_{ii})_{E_i(x)} = d(\mathrm{Id})_{E_i(x)} = \mathrm{Id}$.

(2) Since $T_{ij} \circ T_{ji} = \mathrm{Id}$ on $E_i(U_i \cap U_j)$, by the chain rule
\[
\mathrm{Id} = d(T_{ij} \circ T_{ji})_{E_i(x)} = d(T_{ij})_{T_{ji}(E_i(x))} \cdot d(T_{ji})_{E_i(x)} = g_{ij}(x) \cdot g_{ji}(x),
\]
where the last equality uses $T_{ji}(E_i(x)) = E_j(x)$.

(3) Since $T_{ki} = T_{kj} \circ T_{ji}$, by the chain rule
\[
g_{ki}(x) = d(T_{kj})_{T_{ji}(E_i(x))} \cdot d(T_{ji})_{E_i(x)} = d(T_{kj})_{E_j(x)} \cdot d(T_{ji})_{E_i(x)} = g_{kj}(x) \cdot g_{ji}(x). \qedhere
\]
\end{proof}

\subsubsection{Vector bundle construction}

The linearised transition functions satisfy the hypotheses of Theorem~\ref{thm:bundle-from-transitions}, allowing us to construct a vector bundle.

\begin{definition}[Tangent bundle of an autoencoder atlas]\label{def:tangent-bundle}
Let $\mathcal{A} = \{(U_i, E_i, D_i)\}$ be a smooth autoencoder atlas. The \emph{tangent bundle} $\mathcal{T}_{\mathcal{A}}$ is the vector bundle constructed from the linearised transition functions:
\[
\mathcal{T}_{\mathcal{A}} \coloneqq \left( \bigsqcup_{i \in I} U_i \times \R^d \right)\Big/ \sim
\]
where $(x, v)_i \sim (x, g_{ji}(x) v)_j$ for $x \in U_i \cap U_j$ and $v \in \R^d$.
\end{definition}

\begin{corollary}\label{cor:tangent-bundle-is-vector-bundle}
The tangent bundle $\mathcal{T}_{\mathcal{A}}$ is a rank-$d$ vector bundle over $M$ with transition functions $\{g_{ji}\}$.
\end{corollary}

\begin{proof}
Immediate from Theorem~\ref{thm:bundle-from-transitions} and Lemma~\ref{lemma:linearised-cocycle}, which shows that the linearised transition functions take values in $\GL_d(\R)$ and satisfy the cocycle condition.
\end{proof}

\begin{proposition}[Relationship to the tangent bundle]\label{prop:tangent-bundle-isomorphism}
Let $M$ be a smooth $d$-manifold and $\mathcal{A} = \{(U_i, E_i, D_i)\}$ a smooth autoencoder atlas with latent dimension $d$. If $\{(U_i, E_i)\}$ forms a smooth atlas compatible with the smooth structure of $M$, then $\mathcal{T}_{\mathcal{A}} \cong TM$.
\end{proposition}
\begin{proof}
The transition functions of $TM$ with respect to the smooth atlas $\{(U_i, E_i)\}$ are given by the Jacobians of the coordinate changes $E_j \circ E_i^{-1}$. By the chain rule,
 \[
 d(E_j \circ E_i^{-1})_{E_i(x)} = d(E_j \circ D_i)_{E_i(x)} = g_{ji}(x),
 \]
where we used $D_i = E_i^{-1}$ on $Z_i$ (by the reconstruction condition $D_i \circ E_i = \mathrm{Id}_{U_i}$ and the injectivity of $E_i$ in a smooth autoencoder chart, $D_i$ restricted to $Z_i = E_i(U_i)$ is the inverse of $E_i$). Thus the transition functions of $\mathcal{T}_{\mathcal{A}}$ coincide with those of $TM$, giving an isomorphism of vector bundles.
\end{proof}

The learned bundle $\mathcal{T}_{\mathcal{A}}$ approximates $TM$ when reconstruction quality is high. The robustness of the $\mathbb{Z}/2$ sign cocycle is formalised in Section~\ref{sec:stability}. Theorem~\ref{thm:sign-cocycle-stability} establishes a global cocycle validity result under quantitative reconstruction bounds, Theorem~\ref{thm:local-stability} refines this to a per-triple local statement, and Theorem~\ref{thm:w1-agreement} shows agreement with $w_1(TM)$ when the approximate atlas is sufficiently close to a compatible exact atlas.

\subsubsection{Detecting orientability}

We now apply the theory of Stiefel--Whitney classes to detect orientability of $\mathcal{T}_{\mathcal{A}}$.

\begin{definition}[Jacobian sign cocycle]
Let $\mathcal{A} = \{(U_i, E_i, D_i)\}$ be a smooth autoencoder atlas with linearised transition functions $g_{ji}(x)$. The \emph{Jacobian sign cocycle} is
\[
\omega_{ji} \colon U_i \cap U_j \to \{\pm 1\}, \qquad \omega_{ji}(x) = \mathrm{sign}(\det g_{ji}(x)).
\]
\end{definition}

By Lemma~\ref{lemma:linearised-cocycle} and the multiplicativity of the determinant, $\{\omega_{ji}\}$ satisfies the cocycle condition and thus represents a class in $\check{H}^1(\{U_i\}; \Z/2)$.

\begin{proposition}[Orientability detection]\label{prop:main-orientability}
Let $M$ be a connected smooth manifold and let $\mathcal{A} = \{(U_i, E_i, D_i)\}_{i \in I}$ be a smooth autoencoder atlas. Assume $\{U_i\}$ forms a good cover. Then the following are equivalent:
\begin{enumerate}
\item The tangent bundle $\mathcal{T}_{\mathcal{A}}$ is orientable.
\item The first Stiefel--Whitney class vanishes: $w_1(\mathcal{T}_{\mathcal{A}}) = 0 \in H^1(M; \mathbb{Z}/2)$.
\item The Jacobian sign cocycle $\{\omega_{ji}\}$ is a \v{C}ech coboundary.
\item There exist signs $\nu_i \in \{\pm 1\}$ for each chart such that
\[
\mathrm{sign}(\det g_{ji}(x)) = \nu_j \cdot \nu_i \qquad \text{for all } x \in U_i \cap U_j.
\]
\end{enumerate}
Moreover, if $\mathcal{T}_{\mathcal{A}} \cong TM$, then these conditions are equivalent to orientability of $M$.
\end{proposition}

\begin{proof}
We consider a good cover in the sense of Definition~\ref{def:good-cover}. When intersections have multiple connected components, apply the refinement argument of Remark~\ref{rem:relaxed-nerve} and treat each connected component of each overlap as a separate element in the \v{C}ech complex; since $g_{ji}(x) \in \GL_d(\R)$ the map $x \mapsto \det g_{ji}(x)$ is continuous and never zero, so $\omega_{ji}$ is constant on each connected component by the intermediate value theorem.

$(1) \Leftrightarrow (2)$: Theorem~\ref{thm:orientability-w1}.

$(2) \Leftrightarrow (3)$: Since $\{U_i\}$ is a good cover, \v{C}ech cohomology computes singular cohomology. By Theorem~\ref{thm:w1-sign}, the sign cocycle $\{\omega_{ji}\}$ represents $w_1(\mathcal{T}_{\mathcal{A}})$. Thus $w_1(\mathcal{T}_{\mathcal{A}}) = 0$ if and only if $\{\omega_{ji}\}$ is a coboundary.

$(3) \Leftrightarrow (4)$: By definition, a \v{C}ech $1$-cocycle with values in $\Z/2 \cong \{\pm 1\}$ is a coboundary if and only if there exist $\nu_i \in \{\pm 1\}$ with $\omega_{ji} = \nu_j \cdot \nu_i$.

The final statement follows since bundle isomorphism preserves orientability.
\end{proof}

\begin{remark}[Index convention compatibility]\label{rem:index-compatibility}
In the multiplicative group $\{\pm 1\}$ we have $\omega_{ji} = \omega_{ij}$, since $\det(g_{ji}) = (\det g_{ij})^{-1}$ and $(\pm 1)^{-1} = \pm 1$. Thus the target-first indexing $\omega_{ji}(x) = \mathrm{sign}(\det g_{ji}(x))$ is compatible with the ordered coboundary convention $\omega_{\alpha\beta} = \nu_\alpha \nu_\beta$ used in Corollary~\ref{cor:orientability-sign}.
 
This identity is an algebraic consequence of the exact reconstruction condition $T_{ij} \circ T_{ji} = \mathrm{Id}$ (Proposition~\ref{prop:transition-properties}). In the approximate setting of Section~\ref{sec:stability}, the composition $T_{ij} \circ T_{ji}$ deviates from the identity by $O(\varepsilon)$, and the symmetry $\omega_{ji} = \omega_{ij}$ does not follow automatically; see Remark~\ref{rem:reindexing} for the treatment in the local stability theorem.
\end{remark}

\section{Stability of the sign cocycle under approximate reconstruction}\label{sec:stability}

The theoretical framework developed in Sections~\ref{sec:autoencoders}--\ref{sec:tangent-bundle} assumes exact reconstruction: $D_i \circ E_i = \Id_{U_i}$. In practice, autoencoder training only achieves approximate reconstruction via loss minimisation. In this section, we bridge this gap by proving that the $\Z/2$-valued sign cocycle is stable under sufficiently small reconstruction error, provided the transition map Jacobians remain non-degenerate. We give two stability results: a \emph{global} theorem with uniform constants over the whole atlas (Theorem~\ref{thm:sign-cocycle-stability}), and a \emph{local} per-triple refinement (Theorem~\ref{thm:local-stability}) in which the differential reconstruction condition is required only on two of the three charts of each triangle of the nerve. The local result explains the practical detection regime observed in Section~\ref{sec:exp-higher-dim}, where global hypotheses fail but detection succeeds.

\paragraph{Notation.}
Since the manifold $M$ is embedded in $\R^N$, all norms $\|\cdot\|$ on vectors in $\R^N$ and $\R^d$ denote the Euclidean norm. For a linear map $A \colon \R^m \to \R^n$ (or its matrix representation), $\|A\|_{\mathrm{op}} \coloneqq \sup_{\|v\| = 1} \|Av\|$ denotes the \emph{operator norm} (equivalently, the largest singular value $\sigma_{\max}(A)$), and $\sigma_{\min}(A)$ denotes the smallest singular value. We write $I_N$ for the $N \times N$ identity matrix and $\Id_{T_xM}$ for the identity map on the tangent space $T_xM$. The orthogonal projection from $\R^N$ onto the normal space $(T_xM)^\perp$ is denoted $\Pi_{T_xM}^\perp$. For a $C^1$ map $f : U \subset \R^m \to \R^n$ defined on an open set $U$, the \emph{$C^1$ norm} is
\[
\|f\|_{C^1(U)} \coloneqq \sup_{x \in U} \|f(x)\| + \sup_{x \in U} \|df_x\|_{\mathrm{op}}.
\]
We write $\tau(M) > 0$ for the \emph{reach} of $M$ in $\R^N$ in the sense of \cite{Niyogi2008}: the supremum of $r > 0$ such that every point in the open $r$-tube $\{p \in \R^N : \mathrm{dist}(p, M) < r\}$ has a unique nearest point on $M$. Since $M$ is compact and smoothly embedded, $\tau(M) > 0$.

\subsection{Approximate autoencoder atlases}\label{sec:approximateautoencoders}

\begin{definition}[Approximate autoencoder atlas]\label{def:approximate-atlas}
Let $M$ be a smooth $d$-manifold embedded in $\R^N$. An \emph{$(\varepsilon, \eta)$-approximate autoencoder atlas} on $M$ is a collection $\mathcal{A} = \{(U_i, E_i, D_i)\}_{i \in I}$ where:
\begin{enumerate}
    \item $\{U_i\}$ is an open cover of $M$;
    \item each encoder $E_i$ is a $C^1$ map defined on an open neighborhood $O_i \supset U_i$ in $\R^N$, with values in $\R^d$, and each decoder $D_i \colon \tilde{Z}_i \to \R^N$ is a $C^1$ map on an open set $\tilde{Z}_i \supset E_i(O_i) \subset \R^d$;
    \item the \emph{pointwise reconstruction error} satisfies
    \[
    \sup_{x \in U_i} \| D_i(E_i(x)) - x \| \leq \varepsilon;
    \]
    \item the \emph{differential reconstruction error} satisfies
    \[
    \sup_{x \in U_i} \| d(D_i \circ E_i)_x - \Id_{T_x M} \|_{\mathrm{op}} \leq \eta,
    \]
    where $d(D_i \circ E_i)_x \colon T_x M \to T_{D_i(E_i(x))} \R^N$ is the differential of the reconstruction map restricted to $T_x M$;
    \item there exists a constant $R \geq 1$ such that the open neighborhoods
    $O_i$ contain the closed enlarged neighborhood
    \[
    \overline{B_{\R^N}(U_i,\, R\varepsilon)} \coloneqq \{p \in \R^N : \mathrm{dist}(p, U_i) \leq R\varepsilon\}.
    \]
    The role of $R$ is to ensure that all off-manifold points encountered
    in the proof of the stability theorem — in particular, points of the
    form $D_i(E_i(x))$ and $\Phi_j(D_i(E_i(x)))$ for $x \in U_i$ — lie in
    the encoder domain $O_k$ of every chart $k$ involved.  The specific
    value $R = L_E L_D + 3$ used in Theorem~\ref{thm:sign-cocycle-stability}
    is derived in Lemma~\ref{lem:off-manifold-reconstruction} and
    Remark~\ref{rem:domain-verification}.
\end{enumerate}
We call $\Phi_i \coloneqq D_i \circ E_i \colon O_i \to \R^N$ the
\emph{reconstruction map} of chart $i$.
\end{definition}

\begin{figure}[h!]
    \centering
    \begin{tikzpicture}
    \draw[smooth cycle]
    plot coordinates{
    (-7, 0.5)
    (-7, 1.8)
    (-5.8, 1.8)
    (-5.8, 0.5)
    };
    \path[->] (1.05,1.3) edge [bend right] node[left, yshift=2mm] {$E_i$} (-6.4,1.1);
    \draw[white,fill=white] (0.3,1.7) circle (.1cm);
    \path[->] (-6.4,0.9) edge [bend right] node [right, yshift=-3mm] {$D_i$} (0.8,0.8);
    \draw[white,fill=white] (0.3,0.6) circle (.1cm);
    \draw[white, fill=white] (-5.5,0.5) circle (.21cm);
\draw[smooth cycle, tension=0.5]
plot coordinates{(2,2.5) (-0.5,1) (3,0) (4,2)}
node at (3.3,2.6) {$M$};
    \draw[smooth cycle,dashed]
    plot coordinates {(0.8,0.8) (1.3, 2.4) (2.7,2.5) (2.8, 1.2)}
    node [label={[label distance=-0.8cm, xshift=0.1cm,fill=white]:$O_i$}] {};
    \draw[smooth cycle]
        plot coordinates {(1.1,1.1) (1.5, 2.2) (2.5,2.3) (2.6, 1.4)}
        node [label={[label distance=0cm, xshift=-0.5cm]:$U_i$}] {};
    \draw[thick, ->] (-8,0) -- (-4.8, 0) node [label=above:{$Z_i \coloneqq E_i(U_i)$}] {};
    \draw[thick, ->] (-8,0) -- (-8, 3) node [label=right:$\R^d$] {};
    \foreach \Point in {(1.2,1.2),(0.9,0.9),(-6.5,1)}{
    \node at \Point {\textbullet};
       };
    \node at (1.4,1.2)   {$x$};
    \node at (1.1,0.9)  {$y$};
\end{tikzpicture}
    \caption{Diagram of an approximate atlas autoencoder. Given a point $x \in U_i$, we need to extend the domain to $O_i$ because $y = D_i(E_i(x))$ may not lie on $U_i$. Also, $O_i$ may not be fully in $M$, accommodating off-manifold points.}
    \label{fig:aprox-atlasautoencoder}
\end{figure}

Condition~(2) requires each encoder $E_i$ to be defined on an open set $O_i \supset U_i$ in $\R^N$, rather than only on $U_i \subset M$. This extension is essential in the approximate setting because the decoder output $y = D_i(E_i(x))$ satisfies $\|y - x\| \leq \varepsilon$ but generically $y \notin M$; subsequent encoders must be evaluable at such off-manifold points (see Figure~\ref{fig:aprox-atlasautoencoder}). In practice, this condition is automatically satisfied: neural network encoders are defined on all of $\R^N$ (or a large open subset thereof). Condition~(5) ensures more strongly that all points within distance $R\varepsilon$ of the chart domains remain in the encoder's domain; the factor $R = L_E L_D + 3$ is the bound that arises when the reconstruction map of a neighboring chart is composed with the off-manifold image of another chart's reconstruction (see Lemma~\ref{lem:off-manifold-reconstruction} and Step~0 of the proof of Theorem~\ref{thm:sign-cocycle-stability}). In practice, condition~(3) is controlled by the reconstruction loss $\mathcal{L}_{\mathrm{recon}}$. Condition~(4) is an additional regularity requirement on the derivatives; it is encouraged by the Jacobian regularity loss $\mathcal{L}_{\mathrm{jac}}$ and by the use of smooth activations (e.g., $\tanh$). When $\varepsilon$ and $\eta$ are both zero, we recover an exact autoencoder atlas (Definition~\ref{def:autoencoder-atlas}). See Section~\ref{sec:loss-functions} for the definitions of the different loss functions. In condition~(4), both $d(D_i\circ E_i)_x$ and $\mathrm{Id}_{T_xM}$ are viewed as maps $T_xM \to \mathbb{R}^N$ via the embedding $M \subset \mathbb{R}^N$.

We also require a uniform non-degeneracy condition on the transition map Jacobians.

\begin{definition}[Non-degeneracy gap]\label{def:nondegeneracy}
Let $\mathcal{A} = \{(U_i, E_i, D_i)\}$ be an approximate autoencoder atlas with transition maps $T_{ji} = E_j \circ D_i$. The \emph{non-degeneracy gap} of $\mathcal{A}$ is
\[
\delta(\mathcal{A}) \coloneqq \min_{\substack{(i,j) \\ U_i \cap U_j \neq \emptyset}}
\;\inf_{x \in U_i \cap U_j} |\det g_{ji}(x)|.
\]
We say $\mathcal{A}$ has \emph{positive non-degeneracy gap} if $\delta(\mathcal{A}) > 0$.
\end{definition}

The non-degeneracy gap $\delta > 0$ ensures that the sign function $\sign(\det g_{ji}(x))$ is well-defined (the determinant never passes through zero) and locally constant.

\begin{remark}\label{rem:sign-constancy}
Whenever the non-degeneracy gap is positive, the continuous function $x \mapsto \det g_{ji}(x)$ is bounded away from zero on each non-empty overlap, hence the sign $\omega_{ji}(x) = \mathrm{sign}(\det g_{ji}(x))$ is constant on each connected component of $U_i \cap U_j$ by the intermediate value theorem.
\end{remark}

\subsection{The global stability theorem}

The main subtlety in the approximate setting is that the non-linear cocycle condition $T_{ki} = T_{kj} \circ T_{ji}$ fails: the error depends on the reconstruction quality of the middle chart $j$, as we will make precise in Theorem~\ref{thm:cocycle-error-approx}. Consequently, the matrix-valued cocycle condition $g_{ki}(x) = g_{kj}(x) \cdot g_{ji}(x)$ also fails. However, we now show that the \emph{sign} cocycle condition holds exactly, provided the reconstruction error is small relative to the non-degeneracy gap.

The key observation is a factorisation of the Jacobians through the reconstruction map.

\begin{lemma}[Jacobian factorisation]\label{lemma:jacobian-factorisation}
Let $\mathcal{A} = \{(U_i, E_i, D_i)\}$ be a $C^1$ approximate autoencoder atlas with reconstruction maps $\Phi_j = D_j \circ E_j$. For $x \in U_i \cap U_j \cap U_k$, set $y \coloneqq D_i(E_i(x))$. Assume $y \in O_j \cap O_k$ (verified in Remark~\ref{rem:domain-verification}). Then:
\begin{enumerate}
    \item The direct transition Jacobian satisfies
    \[
    g_{ki}(x) = d(E_k)_y \cdot d(D_i)_{E_i(x)}.
    \]
    \item The composed transition Jacobian satisfies
    \[
    g_{kj}(y) \cdot g_{ji}(x) = d(E_k)_{\Phi_j(y)} \cdot d(\Phi_j)_y \cdot d(D_i)_{E_i(x)},
    \]
    where $g_{kj}(y) \coloneqq d(T_{kj})_{E_j(y)}$ denotes the Jacobian of $T_{kj}$ evaluated at the point $y$ as represented in chart $j$.
\end{enumerate}
\end{lemma}

\begin{proof}
(1) By definition, $T_{ki} = E_k \circ D_i$, so by the chain rule
\[
g_{ki}(x) = d(E_k \circ D_i)_{E_i(x)} = d(E_k)_{D_i(E_i(x))} \cdot d(D_i)_{E_i(x)} = d(E_k)_y \cdot d(D_i)_{E_i(x)}.
\]
Here $E_k$ is evaluated at $y = D_i(E_i(x)) \in O_k$, where it is $C^1$ by Definition~\ref{def:approximate-atlas}(2).

(2) The composition $T_{kj} \circ T_{ji} = (E_k \circ D_j) \circ (E_j \circ D_i) = E_k \circ \Phi_j \circ D_i$. By the chain rule,
\begin{align*}
d(T_{kj} \circ T_{ji})_{E_i(x)}
&= d(E_k)_{\Phi_j(D_i(E_i(x)))} \cdot d(\Phi_j)_{D_i(E_i(x))} \cdot d(D_i)_{E_i(x)} \\
&= d(E_k)_{\Phi_j(y)} \cdot d(\Phi_j)_y \cdot d(D_i)_{E_i(x)}.
\end{align*}
The left-hand side equals $g_{kj}(y) \cdot g_{ji}(x)$ by the chain rule applied to $T_{kj} \circ T_{ji}$ at $E_i(x)$.
\end{proof}

The factorisation in Lemma~\ref{lemma:jacobian-factorisation} reveals that the direct Jacobian $g_{ki}(x)$ and the composed Jacobian $g_{kj}(y) \cdot g_{ji}(x)$ differ by the insertion of the reconstruction map $\Phi_j = D_j \circ E_j$ in place of the identity. The discrepancy depends on how far $\Phi_j$ is from the identity, both in value (controlled by $\varepsilon$) and in derivative (controlled by $\eta$).

\begin{remark}[Domain verification for Lemma~\ref{lemma:jacobian-factorisation}]\label{rem:domain-verification}
In Lemma~\ref{lemma:jacobian-factorisation}, the evaluation points must lie in the extended encoder domains. We verify the two non-trivial cases.
\begin{enumerate}
    \item $y = D_i(E_i(x)) \in \R^N$ satisfies $\|y - x\| \leq \varepsilon$ by Definition~\ref{def:approximate-atlas}(3). Since $x \in U_j \cap U_k$, $\mathrm{dist}(y, U_j) \leq \varepsilon$ and $\mathrm{dist}(y, U_k) \leq \varepsilon$, so $y \in O_j \cap O_k$ by Definition~\ref{def:approximate-atlas}(5) (which provides a margin of $R\varepsilon$, $R \geq 1$).
    \item $\Phi_j(y) = D_j(E_j(y))$ satisfies $\|\Phi_j(y) - y\| \leq (L_E L_D + 2)\varepsilon$ by Lemma~\ref{lem:off-manifold-reconstruction}. Combined with $\|y - x\| \leq \varepsilon$, the triangle inequality gives $\|\Phi_j(y) - x\| \leq (L_E L_D + 3)\varepsilon = R\varepsilon$. Since $x \in U_k$, we have $\mathrm{dist}(\Phi_j(y), U_k) \leq R\varepsilon$, so $\Phi_j(y) \in O_k$ by Definition~\ref{def:approximate-atlas}(5). This is precisely the role of the constant $R$ in the definition of an approximate atlas.
\end{enumerate}
\end{remark}

\begin{lemma}[Off-manifold reconstruction bound]\label{lem:off-manifold-reconstruction}
Let $\mathcal{A}$ be an $(\varepsilon,\eta)$-approximate autoencoder atlas. If $x \in U_j$ and $y \in O_j$ satisfy $\|y - x\| \leq \varepsilon$, then
\[
\|\Phi_j(y) - y\| \leq (L_{E_j} L_{D_j} + 2)\,\varepsilon,
\]
where $L_{E_j}, L_{D_j}$ are Lipschitz constants for $E_j$ and $D_j$.
\end{lemma}
\begin{proof}
Insert the intermediate points $\Phi_j(x)$ and $x$:
\[
\|\Phi_j(y) - y\| \leq \|\Phi_j(y) - \Phi_j(x)\| + \|\Phi_j(x) - x\| + \|x - y\|.
\]
The first term is at most $L_{E_j} L_{D_j} \varepsilon$, the second at most $\varepsilon$ by Definition~\ref{def:approximate-atlas}(3), the third at most $\varepsilon$ by hypothesis.
\end{proof}

We now state the main global stability result.

\begin{theorem}[Stability of the sign cocycle]\label{thm:sign-cocycle-stability}
Let $M$ be a smooth compact $d$-manifold embedded in $\R^N$ and $\mathcal{A} = \{(U_i, E_i, D_i)\}_{i \in I}$ an $(\varepsilon, \eta)$-approximate autoencoder atlas with positive non-degeneracy gap $\delta(\mathcal{A}) > 0$. Writing $\delta \coloneqq \delta(\mathcal{A})$ for brevity, assume that the following bounds hold uniformly over all charts and on the extended encoder domains $O_i$:
\begin{enumerate}[label=(\roman*)]
    \item \textbf{Encoder Lipschitz regularity:} $\| d(E_k)_p \|_{\mathrm{op}} \leq L_E$ for all $k$ and all $p \in O_k$.
    \item \textbf{Encoder derivative Lipschitz continuity:} $\| d(E_k)_p - d(E_k)_q \|_{\mathrm{op}} \leq L_E' \|p - q\|$ for all $k$ and all $p, q \in O_k$.
    \item \textbf{Decoder regularity:} $\| d(D_i)_z \|_{\mathrm{op}} \leq L_D$ for all $i$ and $z \in E_i(O_i)$.
    \item \textbf{Decoder derivative Lipschitz continuity:} $\| d(D_i)_z - d(D_i)_{z'} \|_{\mathrm{op}} \leq L_D' \|z - z'\|$ for all $i$ and $z, z' \in E_i(O_i)$.
    \item \textbf{Tubular geometry:} $\varepsilon < \tau(M)$.
\end{enumerate}
Define the sign cocycle $\omega_{ji}(x) \coloneqq \sign(\det g_{ji}(x))$ on each connected component of each non-empty overlap. Define the effective differential error
\begin{equation}\label{eq:eta-eff-def}
\eta_{\mathrm{eff}} \coloneqq \frac{(L_E L_D + 2)\,\eta}{1 - \eta} + L_\Phi'\,\varepsilon,
\qquad L_\Phi' \coloneqq L_D' L_E^2 + L_D L_E',
\end{equation}
the perturbation magnitude
\begin{equation}\label{eq:Gamma-def}
\Gamma \coloneqq L_E\,\eta_{\mathrm{eff}}\, L_D + L_E'\,\widetilde{\varepsilon}\,(1 + \eta_{\mathrm{eff}})\, L_D,
\qquad \widetilde{\varepsilon} \coloneqq (L_E L_D + 2)\varepsilon,
\end{equation}
and the determinant-stability quantity
\begin{equation}\label{eq:Ldet-def}
K_{\det} \coloneqq d\,C_g\,\varepsilon\,\bigl(L_E L_D + C_g\,\varepsilon\bigr)^{d-1},
\qquad C_g \coloneqq L_E\,(L_E L_D' + L_E' L_D^2).
\end{equation}
If
\begin{equation}\label{eq:stability-condition}
\eta < 1 \qquad \text{and} \qquad \max\!\Bigl(d\,\Gamma\,(L_E L_D + \Gamma)^{d-1},\; K_{\det}\Bigr) < \delta,
\end{equation}
then $\{\omega_{ji}\}$ satisfies the \v{C}ech cocycle condition: for all $x \in U_i \cap U_j \cap U_k$,
\[
\omega_{ki}(x) = \omega_{kj}(x) \cdot \omega_{ji}(x).
\]
In particular, $\{\omega_{ji}\}$ defines a class $[\omega] \in \check{H}^1(\mathcal{U}; \Z/2)$.
\end{theorem}

\begin{proof}[Proof sketch (full proof in Appendix~\ref{proof:sign-cocycle-stability})]
Fix $x \in U_i \cap U_j \cap U_k$ and set $y \coloneqq D_i(E_i(x))$. By Lemma~\ref{lemma:jacobian-factorisation}, the direct and composed transition Jacobians factor as $g_{ki}(x) = QP$ and $g_{kj}(y) \cdot g_{ji}(x) = Q'RP$, where $R = d(\Phi_j)_y$ is the differential of the reconstruction map. This gives $g_{kj}(y) \cdot g_{ji}(x) = g_{ki}(x) + \Delta$ for an explicit perturbation $\Delta$.

The main difficulty is that $\|R - I_N\|_{\mathrm{op}} \geq 1$ since $\Phi_j$ factors through $\R^d$ and annihilates normal directions. However, only the restricted action of $R - I_N$ on the nearly tangential subspace $\mathrm{range}(P)$ enters the bound on $\|\Delta\|_{\mathrm{op}}$. A tangent--normal decomposition exploiting the approximate reconstruction condition yields $\|\Delta\|_{\mathrm{op}} \leq \Gamma$, where $\Gamma = O(\eta + \varepsilon)$ with explicit constants.

A determinant perturbation bound then gives $|\det(g_{ki}(x) + \Delta) - \det g_{ki}(x)| < \delta \leq |\det g_{ki}(x)|$, forcing sign agreement. A Lipschitz continuity argument along the segment $[x, y]$, using condition~(v) and Definition~\ref{def:approximate-atlas}(5) to keep the segment in $O_j \cap O_k$, corrects the evaluation point from $y$ back to $x$.
\end{proof}

\begin{remark}[Role of condition~(v)]\label{rem:reach-condition}
Condition~(v) ($\varepsilon < \tau(M)$) appears only in Step~4 of the proof, where we follow the determinant along the segment from $x$ to $y = D_i(E_i(x))$. Since $\|y - x\| \leq \varepsilon$, the segment lies within the open $\varepsilon$-tube around $M$; condition~(v) guarantees that this tube is regular (every point has a unique nearest point on $M$), and combined with Definition~\ref{def:approximate-atlas}(5) ensures that the segment lies in $O_j \cap O_k$ where the relevant Jacobians are $C^1$. The reach $\tau(M)$ depends only on the geometry of the embedding and is uniformly positive for compact smoothly embedded manifolds.
\end{remark}

\begin{remark}[Simplified sufficient condition]\label{rem:simplified-stability}
In the regime of small perturbations ($\varepsilon, \eta \ll 1$), we have $\eta_{\mathrm{eff}} \approx (L_E L_D + 2)\eta + L_\Phi'\varepsilon$, $\widetilde{\varepsilon} \approx \varepsilon$, and the leading-order terms in~\eqref{eq:stability-condition} are
\[
d \cdot L_E^{d-1} L_D^d \cdot \bigl[(L_E L_D + 2)\, L_E\, \eta \;+\; (L_\Phi' L_E + L_E')\,\varepsilon\bigr] < \delta.
\]
\end{remark}

\begin{remark}[Lipschitz constant for the reconstruction differential]\label{rem:Lphi-constant}
Conditions~(i)--(iv) yield a Lipschitz constant for $p \mapsto d(\Phi_j)_p = d(D_j)_{E_j(p)} \cdot d(E_j)_p$ via the product rule:
\[
L_\Phi' \coloneqq L_D' L_E^2 + L_D L_E'.
\]
Indeed, for $p, q \in O_j$:
\begin{align*}
\|d(\Phi_j)_p - d(\Phi_j)_q\|_{\mathrm{op}}
&\leq \|d(D_j)_{E_j(p)} - d(D_j)_{E_j(q)}\|_{\mathrm{op}}\,\|d(E_j)_p\|_{\mathrm{op}}
+ \|d(D_j)_{E_j(q)}\|_{\mathrm{op}}\,\|d(E_j)_p - d(E_j)_q\|_{\mathrm{op}} \\
&\leq L_D'\,\|E_j(p) - E_j(q)\|\,L_E + L_D\,L_E'\,\|p - q\|
\leq (L_D' L_E^2 + L_D L_E')\,\|p - q\|.
\end{align*}
\end{remark}

Neural networks with smooth activations (e.g., $\tanh$) are $C^\infty$ on $\R^N$ and hence automatically satisfy conditions~(i)--(iv) on any bounded domain.

\subsection{Local stability and per-triple sufficient conditions}\label{sec:local-stability}

Theorem~\ref{thm:sign-cocycle-stability} requires uniform bounds across the whole atlas: a single chart violating $\eta < 1$ invalidates the conclusion globally. The experiments of Section~\ref{sec:exp-higher-dim} show that this is far stronger than necessary in practice: detection succeeds when only the charts \emph{participating in a given triangle} of the nerve are well behaved. We now prove a strictly stronger theorem that captures this phenomenon.

Recall that the cocycle condition for $\{\omega_{ji}\}$ requires checking the identity $\omega_{ki}(x) = \omega_{kj}(x) \cdot \omega_{ji}(x)$ on each non-empty triple intersection $U_i \cap U_j \cap U_k$, independently for each triangle of the nerve. The proof of Theorem~\ref{thm:sign-cocycle-stability} also proceeds triangle by triangle. This suggests localising the hypotheses.

\begin{definition}[Per-triple bounds]\label{def:per-triple}
Let $\mathcal{A}$ be a $C^1$ approximate autoencoder atlas and fix an ordered triple $(i, j, k)$ with $U_i \cap U_j \cap U_k \neq \emptyset$. The \emph{per-triple bounds} for $(i, j, k)$ are the following local versions of conditions~(i)--(v) of Theorem~\ref{thm:sign-cocycle-stability}:
\begin{enumerate}[label=(\roman*)$_{ijk}$]
    \item $\| d(E_\ell)_p \|_{\mathrm{op}} \leq L_E^{(ijk)}$ for $\ell \in \{j, k\}$ and all $p$ in the relevant local domains;
    \item $\| d(E_\ell)_p - d(E_\ell)_q \|_{\mathrm{op}} \leq (L_E')^{(ijk)} \|p - q\|$ for $\ell \in \{j, k\}$;
    \item $\| d(D_\ell)_z \|_{\mathrm{op}} \leq L_D^{(ijk)}$ for $\ell \in \{i, j\}$;
    \item $\| d(D_\ell)_z - d(D_\ell)_{z'} \|_{\mathrm{op}} \leq (L_D')^{(ijk)} \|z - z'\|$ for $\ell \in \{i, j\}$.
    \item[(v)$_{ijk}$] $\varepsilon < \tau(M)$ (the same global condition).
\end{enumerate}
The \emph{per-triple non-degeneracy gap} is
\[
\delta^{(ijk)} \coloneqq \min\!\bigl(\inf_{x \in U_i \cap U_j} |\det g_{ji}(x)|,\;\inf_{x \in U_i \cap U_k} |\det g_{ki}(x)|,\;\inf_{x \in U_j \cap U_k} |\det g_{kj}(x)|\bigr).
\]
The \emph{per-triple differential reconstruction error} is
\[
\eta^{(ijk)} \coloneqq \max\!\bigl(\eta_i,\, \eta_j\bigr),
\qquad \text{where } \eta_\ell \coloneqq \sup_{x \in U_\ell} \|d(\Phi_\ell)_x|_{T_xM} - \Id_{T_xM}\|_{\mathrm{op}}.
\]
\end{definition}

Note that the chart $k$ enters the per-triple bounds only through its encoder $E_k$: its decoder $D_k$ does \emph{not} appear in the proof of Theorem~\ref{thm:sign-cocycle-stability}, and its reconstruction error $\eta_k$ does not enter $\eta^{(ijk)}$.

\begin{theorem}[Local stability of the sign cocycle]\label{thm:local-stability}
Let $M$ be a smooth compact $d$-manifold embedded in $\R^N$, and let $\mathcal{A}$ be a $C^1$ approximate autoencoder atlas. Fix an ordered triple $(i, j, k)$ with $U_i \cap U_j \cap U_k \neq \emptyset$. Suppose the per-triple bounds (i)$_{ijk}$--(v)$_{ijk}$ hold, that $\delta^{(ijk)} > 0$, and that
\begin{equation}\label{eq:local-stability-condition}
\eta^{(ijk)} < 1 \qquad \text{and} \qquad \max\!\Bigl(d\,\Gamma^{(ijk)}\,(L_E^{(ijk)} L_D^{(ijk)} + \Gamma^{(ijk)})^{d-1},\; K_{\det}^{(ijk)}\Bigr) < \delta^{(ijk)},
\end{equation}
where $\Gamma^{(ijk)}$ and $K_{\det}^{(ijk)}$ are defined as in \eqref{eq:Gamma-def} and \eqref{eq:Ldet-def} with each constant replaced by its per-triple version, and $\eta_{\mathrm{eff}}^{(ijk)}$ uses $\eta^{(ijk)}$ in place of $\eta$. Then the sign cocycle identity
\[
\omega_{ki}(x) = \omega_{kj}(x) \cdot \omega_{ji}(x)
\]
holds for all $x \in U_i \cap U_j \cap U_k$.

In particular, fix a total order on $I$ and define the \v{C}ech 
$1$-cochain by $\omega_{\alpha\beta}^{\check{C}} 
\coloneqq \omega_{\beta\alpha}$ for $\alpha < \beta$ 
(cf.\ Remark~\ref{rem:index-compatibility}). If 
\eqref{eq:local-stability-condition} holds for every ordered triple 
$(\alpha, \beta, \gamma)$ with $\alpha < \beta < \gamma$ and 
$U_\alpha \cap U_\beta \cap U_\gamma \neq \emptyset$, then 
$\{\omega_{\alpha\beta}^{\check{C}}\}$ is a \v{C}ech $1$-cocycle 
and defines a class 
$[\omega] \in \check{H}^1(\mathcal{U}; \Z/2)$.
\end{theorem}

\begin{proof}
The proof of Theorem~\ref{thm:sign-cocycle-stability} (Appendix~\ref{proof:sign-cocycle-stability}) is carried out at a single fixed $x \in U_i \cap U_j \cap U_k$: the uniform constants $L_E, L_E', L_D, L_D', \eta, \delta$ enter only as bounds on
\begin{itemize}
    \item the encoder $E_k$ and its derivative on $O_k$ (Steps~0, 1c, 4) -- this uses (i)$_{ijk}$, (ii)$_{ijk}$ for $\ell = k$;
    \item the encoder $E_j$ and its derivative on $O_j$ (Steps~1a, 4) -- this uses (i)$_{ijk}$, (ii)$_{ijk}$ for $\ell = j$;
    \item the decoder $D_i$ and its derivative (Steps~1, 1c) -- this uses (iii)$_{ijk}$, (iv)$_{ijk}$ for $\ell = i$;
    \item the decoder $D_j$ and its derivative, via the reconstruction map $\Phi_j$ (Steps~1a, 1b, 4) -- this uses (iii)$_{ijk}$, (iv)$_{ijk}$ for $\ell = j$;
    \item the differential reconstruction conditions for charts $i$ and $j$ (Steps~1a, 1b) -- this uses $\eta^{(ijk)} = \max(\eta_i, \eta_j)$;
    \item the non-degeneracy at the three vertices of the triangle (Steps~3, 4) -- this uses $\delta^{(ijk)}$;
    \item the reach $\tau(M)$ (Step~4) -- this uses (v)$_{ijk}$.
\end{itemize}
The decoder $D_k$ and the differential reconstruction error $\eta_k$ never appear: chart $k$'s reconstruction map $\Phi_k$ is not involved in the factorisation of Lemma~\ref{lemma:jacobian-factorisation}, and $D_k$ does not appear in any Jacobian computation in the proof. Substituting the per-triple constants for the uniform constants throughout the proof yields the conclusion.
\end{proof}

\begin{remark}[Hiding a bad chart via re-indexing]\label{rem:reindexing}
The \v{C}ech cocycle condition is stated with respect to a fixed total order on the index set $I$. For $\alpha < \beta < \gamma$, the identity verified by Theorem~\ref{thm:local-stability} is
\[
\omega_{\gamma\alpha}(x) = \omega_{\beta\alpha}(x) \cdot \omega_{\gamma\beta}(x),
\]
where the chart with the largest index $\gamma$ occupies the encoder-only slot: its decoder $D_\gamma$ and reconstruction error $\eta_\gamma$ do not enter the per-triple bound.
 
Since \v{C}ech cohomology is independent of the labelling of the index set, we may re-index the charts so that any single chart $c$ with anomalously large $\eta_c$ receives the \emph{largest} index. With this choice, for every ordered triple $(\alpha, \beta, c)$ with $\alpha < \beta < c$, the per-triple differential reconstruction error is $\eta^{(\alpha,\beta,c)} = \max(\eta_\alpha, \eta_\beta)$, which involves only the regular charts. For triples not involving~$c$, all three charts are regular and the per-triple condition holds regardless of slot assignment.
 
Call a chart $\ell$ \emph{regular} if all the per-chart bounds (i)$_{ijk}$--(iv)$_{ijk}$ and the condition $\eta_\ell < 1$ hold. A single anomalous chart can be hidden in the largest-index slot of every triangle it participates in, provided the remaining two charts in each such triangle are regular. However, if two or more charts per triangle have $\eta_\ell \gg 1$, at most one can occupy the encoder-only slot, and the other must sit in the $\alpha$- or $\beta$-position where its reconstruction error enters the bound.
 
The re-indexing trick determines the sign cocycle as a \v{C}ech cochain for a particular ordering. Under the additional hypotheses of Theorem~\ref{thm:w1-agreement}, the interpolation to an exact atlas gives $\omega_{ji}^{\mathcal{A}}(x) = \omega_{ji}^{\mathcal{A}_0}(x)$ pointwise for every pair $(i,j)$, and in the exact atlas the symmetry $\omega_{ji}^{\mathcal{A}_0} = \omega_{ij}^{\mathcal{A}_0}$ holds (Remark~\ref{rem:index-compatibility}). The resulting cohomology class is therefore independent of the ordering and equals $w_1(TM)$.
\end{remark}

\begin{remark}[Asymmetric role of the three charts]\label{rem:asymmetric}
The proof of Theorem~\ref{thm:local-stability} reveals an asymmetry: only the reconstruction maps $\Phi_i$ and $\Phi_j$ enter the bound, whereas chart~$k$ contributes only through its encoder. In the global Theorem~\ref{thm:sign-cocycle-stability}, the uniform bounds allow the proof to be run for every ordering of every triple, so this asymmetry is absorbed into the global constants. In the local theorem, it is essential and can be exploited through re-indexing (Remark~\ref{rem:reindexing}): a single chart with bad reconstruction quality can be placed in the largest-index slot, where only its encoder is used, leaving the sign cocycle valid on all triangles containing it. This explains why, in our experiments (Section~\ref{sec:exp-higher-dim}), a single chart with $\eta_{\mathrm{lat},k} \gg 1$ can either preserve or destroy detection depending on how many other charts in its incident triangles are also irregular.
\end{remark}

\subsection{Agreement with the true first Stiefel--Whitney class}

Having established that the sign cocycle is a valid \v{C}ech cocycle under approximate reconstruction (globally or locally), we now show that its cohomology class agrees with $w_1(TM)$.

\begin{theorem}[Agreement with $w_1(TM)$]\label{thm:w1-agreement}
Let $M$ be a compact smooth $d$-manifold and $\{U_i\}$ a good cover of $M$. Suppose there exists an exact smooth autoencoder atlas $\mathcal{A}_0 = \{(U_i, E_i^0, D_i^0)\}$ compatible with the smooth structure of $M$, with non-degeneracy gap $\delta_0 \coloneqq \delta(\mathcal{A}_0) > 0$. Let $\mathcal{A} = \{(U_i, E_i, D_i)\}$ be a $C^1$ approximate atlas over the same cover satisfying:
\begin{enumerate}[label=(\roman*)]
    \item for each $i$ there exists an open set $\widetilde{Z}_i \subset \R^d$ containing the $\mu$-neighborhood of $Z_i^0$, on which both $D_i^0$ and $D_i$ are defined and $C^1$;
    \item $\|E_i - E_i^0\|_{C^1(U_i)} \leq \mu$ and $\|D_i - D_i^0\|_{C^1(\widetilde{Z}_i)} \leq \mu$ for all $i$;
    \item $\mathcal{A}$ has positive non-degeneracy gap.
\end{enumerate}
If
\begin{equation}\label{eq:mu-condition}
\mu < \frac{\delta_0}{2C_0},
\end{equation}
where $C_0$ depends on the number of charts $|I|$, the $C^2$-norms of the exact atlas $\mathcal{A}_0$, and the geometry of $M$, then the sign cocycles of $\mathcal{A}$ and $\mathcal{A}_0$ define the same cohomology class:
\[
[\omega^{\mathcal{A}}] = [\omega^{\mathcal{A}_0}] = w_1(TM) \quad \in \check{H}^1(\mathcal{U}; \Z/2).
\]
\end{theorem}
\begin{proof}[Proof sketch (full proof in Appendix~\ref{proof:agreement})]
Define a one-parameter family of atlases by linear interpolation, $\mathcal{A}_t = \{(U_i, E_i^t, D_i^t)\}$ with $E_i^t \coloneqq (1-t)E_i^0 + tE_i$ and $D_i^t \coloneqq (1-t)D_i^0 + tD_i$ for $t \in [0,1]$. The transition Jacobians $g_{ji}^t(x)$ vary continuously with $t$. By compactness of $M$ and uniform continuity, condition \eqref{eq:mu-condition} ensures $|\det g_{ji}^t(x)| \geq \delta_0/2 > 0$ for all $t \in [0,1]$. The intermediate value theorem then forces $\sign(\det g_{ji}^t(x))$ to be constant in $t$, giving $\omega_{ji}^{\mathcal{A}} = \omega_{ji}^{\mathcal{A}_0}$ pointwise. Since $\mathcal{A}_0$ is compatible with the smooth structure, $[\omega^{\mathcal{A}_0}] = w_1(TM)$ by Theorem~\ref{thm:w1-sign}.
\end{proof}

\begin{remark}[On the constant $C_0$]\label{rem:C0-constant}
$C_0$ can be made explicit as $C_0 \leq |I|^2 \cdot \sup_{i,j,x} \|d(\det g_{ji}^t)/d\mu\|$, depending on the $C^2$ norms of the encoders and decoders. In practice, condition~\eqref{eq:stability-condition} of Theorem~\ref{thm:sign-cocycle-stability} (which is fully explicit) is typically binding.
\end{remark}

\begin{remark}[Role of compactness]
Compactness of $M$ is used to obtain uniform bounds. For non-compact manifolds, the same result holds if the regularity bounds and non-degeneracy gap hold uniformly on a compact subset containing all overlaps.
\end{remark}

\begin{remark}[Existence of a compatible exact atlas]
The existence of $\mathcal{A}_0$ over a given good cover is guaranteed by Theorem~\ref{thm:charts-from-good-covers}, since good covers automatically trivialise $TM$ (Lemma~\ref{lem:good-cover-trivializes}).
\end{remark}

\begin{corollary}[Orientability detection from learned approximate atlases]\label{cor:practical-orientability}
Let $M$ be a compact smooth $d$-manifold embedded in $\R^N$ and $\{U_i\}_{i \in I}$ a good cover of $M$. Let $\mathcal{A} = \{(U_i, E_i, D_i)\}_{i \in I}$ be an $(\varepsilon, \eta)$-approximate autoencoder atlas satisfying:
\begin{enumerate}[label=(\roman*)]
    \item \textbf{Cocycle validity:} the global hypotheses of Theorem~\ref{thm:sign-cocycle-stability}, or the per-triple hypotheses of Theorem~\ref{thm:local-stability} on every triangle, hold, so that the sign cocycle $\omega_{ji}(x) = \sign(\det g_{ji}(x))$ is a valid \v{C}ech $1$-cocycle;
    \item \textbf{Proximity to an exact atlas:} there exists an exact smooth autoencoder atlas $\mathcal{A}_0$ compatible with the smooth structure of $M$ over the same cover, with $\|E_i - E_i^0\|_{C^1} \leq \mu$, $\|D_i - D_i^0\|_{C^1} \leq \mu$, and $\mu$ satisfying \eqref{eq:mu-condition}.
\end{enumerate}
Then $M$ is orientable if and only if there exist signs $\nu_i \in \{\pm 1\}$ such that $\omega_{ji}(x) = \nu_j \cdot \nu_i$ for all $x \in U_i \cap U_j$.
\end{corollary}

\begin{proof}
By (i) and Theorem~\ref{thm:sign-cocycle-stability} or Theorem~\ref{thm:local-stability}, the sign cocycle defines a class $[\omega^{\mathcal{A}}] \in \check{H}^1(\mathcal{U}; \Z/2)$. By (ii) and Theorem~\ref{thm:w1-agreement}, $[\omega^{\mathcal{A}}] = w_1(TM)$. By Theorem~\ref{thm:orientability-w1}, $M$ is orientable iff $w_1(TM) = 0$, iff $\{\omega_{ji}\}$ is a coboundary.
\end{proof}

\subsection{The non-degeneracy condition in practice}

The non-degeneracy gap $\delta > 0$ is the crucial assumption underlying stability. We now discuss its practical significance and verifiability.

\begin{remark}[Monitoring non-degeneracy]\label{rem:monitoring}
In practice, $\delta(\mathcal{A}) > 0$ can be verified post-training by computing $|\det g_{ji}(x)|$ at each sample point $x$ in each overlap and checking that the minimum value is bounded away from zero. A small minimum indicates that the sign cocycle may be unreliable at those points, suggesting the need for finer charts or additional training.
\end{remark}

\begin{remark}[Relationship to the Jacobian regularity loss]
The Jacobian regularity loss $\mathcal{L}_{\mathrm{jac}}$ encourages $\sigma_{\min}(d E_i) \geq \epsilon > 0$. Since $g_{ji}(x) = d(E_j)_{D_i(E_i(x))} \cdot d(D_i)_{E_i(x)}$, the Jacobian regularity loss indirectly promotes non-degeneracy of the transition map Jacobians. The following proposition makes this connection precise for exact atlases.
\end{remark}

\begin{proposition}[Non-degeneracy from encoder--decoder regularity]\label{prop:nondeg-bound}
Let $\mathcal{A}$ be a $C^1$ exact autoencoder atlas on a smooth $d$-manifold $M \subset \R^N$. Suppose there exist $s_E, s_D > 0$ such that for all $i$ and $x \in U_i$,
\[
\sigma_{\min}\bigl(d(E_i)_x|_{T_xM}\bigr) \geq s_E, \qquad \sigma_{\min}\bigl(d(D_i)_{E_i(x)}\bigr) \geq s_D.
\]
Then for all overlapping pairs $(i,j)$ and all $x \in U_i \cap U_j$,
\[
|\det g_{ji}(x)| \geq s_E^d \cdot s_D^d,
\]
so $\delta(\mathcal{A}) \geq (s_E \cdot s_D)^d > 0$.
\end{proposition}

\begin{proof}
Set $y \coloneqq D_i(E_i(x))$ and let $A \coloneqq d(E_j)_y$, $B \coloneqq d(D_i)_{E_i(x)}$, so that $g_{ji}(x) = AB$.

Since $\mathcal{A}$ is exact, $D_i(Z_i) \subset M$, hence $\operatorname{range}(B) \subseteq T_y M$. The restriction $d(E_j)_y|_{T_yM} : T_yM \to \R^d$ satisfies $\|Aw\| \geq s_E \|w\|$ for all $w \in T_yM$. Since $Bv \in T_yM$ for every $v \in \R^d$,
\[
\|g_{ji}(x)v\| = \|A(Bv)\| \geq s_E\,\|Bv\| \geq s_E \cdot s_D\,\|v\|,
\]
giving $\sigma_{\min}(g_{ji}(x)) \geq s_E \cdot s_D$ and $|\det g_{ji}(x)| \geq (s_E s_D)^d$.
\end{proof}

\begin{remark}[Extension to approximate atlases]
For approximate atlases with reconstruction error $\varepsilon > 0$, the decoder image $D_i(Z_i)$ deviates from $M$ by $O(\varepsilon)$, so the range of $B$ deviates from $T_yM$ by $O(\varepsilon)$. A perturbation argument yields $|\det g_{ji}(x)| \geq (s_E s_D)^d - O(\varepsilon)$.
\end{remark}

\begin{remark}[Summary of the stability framework]
The practical pipeline for orientability detection is justified as follows:
\begin{enumerate}
    \item \textbf{Train} an autoencoder atlas with reconstruction loss, achieving small $\varepsilon$.
    \item \textbf{Verify non-degeneracy:} check that $|\det g_{ji}(x)| \geq \delta > 0$ on all overlaps.
    \item \textbf{Compute the sign cocycle} $\omega_{ji}(x) = \sign(\det g_{ji}(x))$.
    \item \textbf{Test the coboundary condition:} search for $\{\nu_i\}$ with $\omega_{ji} = \nu_j \cdot \nu_i$.
\end{enumerate}
Theorem~\ref{thm:sign-cocycle-stability} (or its localised refinement Theorem~\ref{thm:local-stability}) guarantees that the sign cocycle is a valid \v{C}ech cocycle, and Theorem~\ref{thm:w1-agreement} guarantees that this class equals $w_1(TM)$. Corollary~\ref{cor:practical-orientability} then ensures that the coboundary test correctly detects orientability.
\end{remark}

\section{Topological Constraints on the Number of Charts}\label{sec:obstruction}

We analyse how many autoencoder charts are required to represent a manifold. The main conclusion is that the minimum number of charts is determined by the topology of $M$ as a space, specifically by the minimum cardinality of a good cover, rather than by the structure of the tangent bundle $TM$. Non-trivial characteristic classes provide useful \emph{obstructions} to single-chart representations, but the converse does not hold: a trivial tangent bundle does not guarantee that a single chart suffices.

\subsection{Single-chart autoencoders imply trivial tangent bundle}

\begin{definition}[Single-chart autoencoder]
A \emph{single-chart autoencoder} on a smooth $d$-manifold $M$ consists of maps $E \colon M \to Z \subset \R^d$ and $D \colon Z \to M$ such that $E$ is a diffeomorphism onto its image and $D \circ E = \Id_M$.
\end{definition}

\begin{proposition}[Single chart implies trivial bundle]\label{prop:single-chart-trivial}
If $M$ admits a single-chart autoencoder with latent dimension $d = \dim M$, then the tangent bundle $TM$ is trivialisable, and consequently all Stiefel--Whitney classes vanish: $w_i(TM) = 0$ for all $i \geq 1$.
\end{proposition}

\begin{proof}
The encoder $E \colon M \to Z \subset \R^d$ is a diffeomorphism onto an open subset $Z$. The standard frame $\{\partial/\partial x_1, \ldots, \partial/\partial x_d\}$ on $\R^d$ restricts to a global frame on $Z$, which pulls back via $E^{-1}$ to a global frame on $M$. Hence $TM \cong M \times \R^d$ is trivial, and all characteristic classes of a trivial bundle vanish~\cite[Lemma~6.8]{Zakharevich2024}.
\end{proof}

\begin{corollary}[Non-trivial bundle requires multiple charts]\label{cor:nontrivial-multiple}
If $w_i(TM) \neq 0$ for some $i \geq 1$, then any autoencoder atlas on $M$ with latent dimension $d = \dim M$ requires at least two charts.
\end{corollary}

\begin{remark}[The converse fails]
A trivial tangent bundle does not imply that a single chart suffices. Every compact manifold requires at least two charts, since a compact $d$-manifold cannot be diffeomorphic to an open subset of $\R^d$. This includes compact parallelizable manifolds such as $S^1$, $T^2$, $S^3$, and $S^7$, all of which have $TM \cong M \times \R^d$ yet require multiple charts for purely topological reasons.
\end{remark}

\subsection{Chart count is determined by good cover structure}

We now show that the minimum number of autoencoder charts equals the minimum cardinality of a good cover of $M$, independently of the bundle structure.

The key observation is that every good cover automatically trivialises \emph{any} vector bundle.

\begin{lemma}[Good covers trivialize all bundles]\label{lem:good-cover-trivializes}
Let $\mathcal{U} = \{U_i\}$ be a good cover of a smooth manifold $M$, and let $\xi$ be any vector bundle over $M$. Then $\xi|_{U_i}$ is trivial for every $i$.
\end{lemma}

\begin{proof}
Each $U_i$ is contractible. By homotopy invariance of vector bundles~\cite[\S 3]{milnor1974characteristic}, any bundle over a contractible paracompact space is trivial: if $h \colon U_i \times [0,1] \to U_i$ contracts $U_i$ to a point $p$, then $\xi|_{U_i} \cong h_1^*\xi = U_i \times \xi_p$.
\end{proof}

\begin{proposition}[Autoencoder charts are local trivializations]\label{prop:chart-trivialization}
Let $(U, E, D)$ be a smooth autoencoder chart with latent dimension $d = \dim M$. Then $E$ provides a local trivialization of $TM|_U$: the frame $\{(dE)^{-1}(\partial/\partial z_1), \ldots, (dE)^{-1}(\partial/\partial z_d)\}$ trivialises $TM$ over $U$.
\end{proposition}

\begin{proof}
The encoder $E \colon U \to Z \subset \R^d$ is a diffeomorphism onto its image (Definition~\ref{def:autoencoder-chart}). Its differential $dE_x \colon T_xM \to T_{E(x)}\R^d \cong \R^d$ is an isomorphism at each $x \in U$. The inverse $(dE_x)^{-1}$ applied to the standard basis of $\R^d$ gives a frame on $U$.
\end{proof}

\begin{theorem}[Charts from good covers]\label{thm:charts-from-good-covers}
Let $\mathcal{U} = \{U_i\}_{i=1}^n$ be a good cover of $M$ in which each $U_i$ is diffeomorphic to an open subset of $\R^d$. Then there exists an autoencoder atlas $\mathcal{A} = \{(U_i, E_i, D_i)\}_{i=1}^n$ with $n$ charts such that $\mathcal{T}_{\mathcal{A}} \cong TM$.

Conversely, any autoencoder atlas $\mathcal{A} = \{(U_i, E_i, D_i)\}_{i=1}^n$ with $\mathcal{T}_{\mathcal{A}} \cong TM$ provides a trivializing cover of $TM$ with $n$ open sets, each diffeomorphic to an open subset of $\R^d$.
\end{theorem}

\begin{proof}
$(\Rightarrow)$ By hypothesis, for each $i$ there is a smooth diffeomorphism $\phi_i \colon U_i \to V_i$ onto an open subset $V_i \subset \R^d$. Set $E_i \coloneqq \phi_i$ and $D_i \coloneqq \phi_i^{-1}$. Then $(U_i, E_i, D_i)$ is a smooth autoencoder chart in the sense of Definition~\ref{def:smooth-chart}, and $\{(U_i, E_i)\}$ is a smooth atlas compatible with the smooth structure of $M$. By Proposition~\ref{prop:tangent-bundle-isomorphism}, $\mathcal{T}_{\mathcal{A}} \cong TM$.

$(\Leftarrow)$ By Proposition~\ref{prop:chart-trivialization}, each chart $(U_i, E_i, D_i)$ trivialises $\mathcal{T}_{\mathcal{A}}$ over $U_i$. If $\mathcal{T}_{\mathcal{A}} \cong TM$, this is also a trivialisation of $TM$.
\end{proof}
\begin{remark}[Realization of the hypothesis]\label{rem:good-cover-realization}
The hypothesis is mild: every Riemannian good cover (good cover by geodesically convex balls under any auxiliary Riemannian metric on~$M$) satisfies it, since $\exp_{p_i}$ restricts to a diffeomorphism from a Euclidean ball onto each $U_i$. For a compact smooth manifold $M$, Bott--Tu~\cite[Theorem~5.1]{BottTu} shows that any open cover admits a refinement by such balls; this refinement may increase cardinality, but in all examples we consider (and in the bounds of~\cite{Karoubi2017}), the strict covering type is realized by such a cover.
\end{remark}
\begin{corollary}[Minimum number of charts]\label{cor:min-charts}
The minimum number of charts in an autoencoder atlas on $M$ equals the minimum cardinality of a good cover of $M$ whose elements are each diffeomorphic to an open subset of $\R^d$.
\end{corollary}

\begin{proof}
The forward direction follows from Theorem~\ref{thm:charts-from-good-covers}: such a cover of size $n$ yields an autoencoder atlas with $n$ charts. Conversely, by Proposition~\ref{prop:chart-trivialization}, the chart domains of an autoencoder atlas are each diffeomorphic to an open subset of $\R^d$; if the atlas has $\mathcal{T}_{\mathcal{A}} \cong TM$, this is also a trivialising cover of $TM$. By Lemma~\ref{lem:good-cover-trivializes}, every good cover trivialises $TM$, so the converse cover-count bound applies.
\end{proof}

\begin{center}
\renewcommand{\arraystretch}{1.3}
\begin{tabular}{|l|c|l|}
\hline
\textbf{Condition} & \textbf{Min.\ charts} & \textbf{Reason} \\
\hline
Non-compact, contractible & $1$ & $M \cong \R^d$ \\
Non-compact, not contractible & $\geq 2$ & Good cover requires $\geq 2$ sets \\
Compact & $\geq 2$ & No open subset of $\R^d$ is compact \\
$w_i(TM) \neq 0$ for some $i$ & $\geq 2$ & Proposition~\ref{prop:single-chart-trivial} \\
\hline
\end{tabular}
\end{center}

The minimum cardinality of a good cover of a space has been studied systematically by Karoubi and Weibel~\cite{Karoubi2017}, who define the \emph{covering type} $\operatorname{ct}(X)$ as the minimum size of a good cover of any space homotopy equivalent to~$X$. For a fixed smooth manifold~$M$, Corollary~\ref{cor:min-charts} identifies the minimum number of autoencoder charts with the minimum cardinality of a good cover whose elements admit smooth charts to $\R^d$. This is bounded below by the strict covering type of~$M$, the minimum cardinality of any good cover of $M$, and coincides with it whenever the strict covering type is realized by a cover of the form above --- which is the case for every example we consider. The covering type is bounded below by the homological dimension via $\operatorname{ct}(M) \geq \operatorname{hd}(M) + 2$ \cite[Proposition~3.1]{Karoubi2017}, and by the non-vanishing of the cohomology cup product \cite[Proposition~5.2]{Karoubi2017}. The latter yields $\operatorname{ct}(M) \geq 6$ whenever the cup product on $H^1(M; \Z/2)$ is non-trivial, which applies to all non-orientable surfaces of genus $q \geq 2$ and all oriented surfaces of genus $g \geq 1$.

\section{Loss functions}\label{sec:loss-functions}

We now describe the loss functions used to train autoencoder atlases in practice. As discussed in Remark~\ref{rem:idealization}, the mathematical definitions of Section~\ref{sec:autoencoders} represent idealised conditions approximated through optimisation.

Throughout this section, we assume the cover $\{U_i\}_{i \in I}$ is given, and we write $\rho_i(x) = \mathbf{1}_{U_i}(x)$ for the indicator function of $U_i$.

\subsection*{Reconstruction loss}

The reconstruction loss enforces the condition $D_i \circ E_i \approx \mathrm{Id}_{U_i}$ from Definition~\ref{def:autoencoder-chart}.

\begin{definition}[Reconstruction loss]\label{def:recon-loss}
\[
\mathcal{L}_{\mathrm{recon}} = \mathbb{E}_{x \sim M}\left[ \sum_{i=1}^{n} \rho_i(x)\, \| x - D_i(E_i(x)) \|^2 \right].
\]
\end{definition}

For an $(\varepsilon, \eta)$-approximate atlas, the value of $\mathcal{L}_{\mathrm{recon}}$ controls the pointwise reconstruction bound $\sup_x \|D_i(E_i(x)) - x\|$ via standard concentration, giving direct access to the quantity $\varepsilon$ of Definition~\ref{def:approximate-atlas}.

\subsection*{Cocycle loss and its relation to reconstruction loss}\label{subsec:cocycle-loss}

By Lemma~\ref{lemma:cocycle-from-reconstruction}, exact reconstruction implies the cocycle condition. We now show that in the approximate setting, the cocycle error in a triple $i \to j \to k$ depends only on the reconstruction error of the \emph{middle} chart $j$. This justifies the absence of an explicit cocycle term in the loss.

\begin{definition}[Cocycle loss]\label{def:cocycle-loss}
For an autoencoder atlas $\mathcal{A} = \{(U_i, E_i, D_i)\}_{i=1}^n$, the \emph{cocycle loss} is
\[
\mathcal{L}_{\mathrm{cocycle}} = \mathbb{E}_{x \sim M}\left[ \sum_{\substack{i,j,k \\ U_i \cap U_j \cap U_k \neq \emptyset}} \rho_{ijk}(x)\, \left\| T_{ki}(E_i(x)) - T_{kj}(T_{ji}(E_i(x))) \right\|^2 \right],
\]
where $\rho_{ijk}(x) = \mathbf{1}_{U_i \cap U_j \cap U_k}(x)$.
\end{definition}

We give two versions of the key identity: an exact version (clean pointwise equality on the manifold) and an approximate version (quantitative bound through extended encoders).

\begin{theorem}[Cocycle error depends only on chart $j$, exact case]\label{thm:cocycle-error-exact}
Let $\mathcal{A} = \{(U_i, E_i, D_i)\}$ be a smooth (exact) autoencoder atlas. For all $x \in U_i \cap U_j \cap U_k$,
\begin{equation}\label{eq:cocycle-error-exact}
T_{ki}(E_i(x)) - T_{kj}(T_{ji}(E_i(x))) = 0.
\end{equation}
That is, the integrand of $\mathcal{L}_{\mathrm{cocycle}}$ vanishes identically.
\end{theorem}

\begin{proof}
By the reconstruction condition $D_i \circ E_i = \mathrm{Id}_{U_i}$, $D_i(E_i(x)) = x$, so
\[
T_{ki}(E_i(x)) = E_k(D_i(E_i(x))) = E_k(x).
\]
Similarly $D_j \circ E_j = \mathrm{Id}_{U_j}$ gives $D_j(E_j(x)) = x$, and
\[
T_{kj}(T_{ji}(E_i(x))) = T_{kj}(E_j(D_i(E_i(x)))) = T_{kj}(E_j(x)) = E_k(D_j(E_j(x))) = E_k(x).
\]
Both expressions equal $E_k(x)$.
\end{proof}

In the approximate setting, $y \coloneqq D_i(E_i(x))$ no longer equals $x$ and generically lies off the manifold. The encoders must then be evaluated at $y$ in their extended domains $O_j, O_k$ (Definition~\ref{def:approximate-atlas}(2),(5)). The cocycle error is then controlled by chart $j$'s reconstruction quality alone.

\begin{theorem}[Cocycle error depends only on chart $j$, approximate case]\label{thm:cocycle-error-approx}
Let $\mathcal{A}$ be an $(\varepsilon, \eta)$-approximate autoencoder atlas. Fix $x \in U_i \cap U_j \cap U_k$ and set $y \coloneqq D_i(E_i(x))$. Then $y \in O_j \cap O_k$, and
\begin{equation}\label{eq:cocycle-error-approx}
T_{ki}(E_i(x)) - T_{kj}(T_{ji}(E_i(x))) = E_k(y) - E_k\bigl(\Phi_j(y)\bigr).
\end{equation}
In particular:
\begin{enumerate}
    \item The cocycle error vanishes if and only if $\Phi_j(y) = y$, i.e.\ chart $j$ reconstructs the point $y$ exactly.
    \item Under the regularity bounds (i)--(iv) of Theorem~\ref{thm:sign-cocycle-stability},
    \begin{equation}\label{eq:cocycle-error-bound}
    \bigl\|T_{ki}(E_i(x)) - T_{kj}(T_{ji}(E_i(x)))\bigr\| \leq L_E \cdot (L_E L_D + 2)\,\varepsilon.
    \end{equation}
\end{enumerate}
\end{theorem}

\begin{proof}
Since $\|y - x\| \leq \varepsilon$ and $x \in U_j \cap U_k$, Definition~\ref{def:approximate-atlas}(5) yields $y \in O_j \cap O_k$, so $E_j(y)$ and $E_k(y)$ are defined.

For the left-hand side of~\eqref{eq:cocycle-error-approx}:
\[
T_{ki}(E_i(x)) = E_k(D_i(E_i(x))) = E_k(y).
\]
For the right-hand side:
\[
T_{kj}(T_{ji}(E_i(x))) = T_{kj}(E_j(D_i(E_i(x)))) = T_{kj}(E_j(y)) = E_k(D_j(E_j(y))) = E_k(\Phi_j(y)).
\]
Subtracting gives~\eqref{eq:cocycle-error-approx}.

For (1): the right-hand side $E_k(y) - E_k(\Phi_j(y))$ vanishes when $\Phi_j(y) = y$, since $E_k$ is well defined at both points. Conversely, if it vanishes for all such $x$ in some open subset, and $E_k$ is locally injective on $O_k$ (which follows from (i) and (ii) of Theorem~\ref{thm:sign-cocycle-stability} combined with $\sigma_{\min}(d E_k) > 0$, ensured for instance by the Jacobian regularity loss), then $\Phi_j(y) = y$.

For (2): by the Lipschitz bound (i) on $E_k$ on $O_k$,
\[
\|E_k(y) - E_k(\Phi_j(y))\| \leq L_E \cdot \|\Phi_j(y) - y\|.
\]
Lemma~\ref{lem:off-manifold-reconstruction} bounds $\|\Phi_j(y) - y\| \leq (L_E L_D + 2)\varepsilon$, giving~\eqref{eq:cocycle-error-bound}.
\end{proof}

\begin{remark}[Interpretation]\label{rem:cocycle-interpretation}
Theorems~\ref{thm:cocycle-error-exact} and~\ref{thm:cocycle-error-approx} together explain why an explicit cocycle term is unnecessary in the loss: the cocycle error is bounded by $L_E (L_E L_D + 2)\,\varepsilon$, so $\mathcal{L}_{\mathrm{recon}} \to 0$ forces $\mathcal{L}_{\mathrm{cocycle}} \to 0$ at the same rate, with no additional optimisation pressure required.

Moreover, the reconstruction error of chart $i$ enters only by determining \emph{where} chart $j$'s reconstruction is evaluated (at $y$ instead of $x$); the cocycle error itself is controlled entirely by chart $j$'s failure to satisfy $\Phi_j = \mathrm{Id}$ at $y$. Chart $k$ contributes only through the encoder $E_k$, which transports the ambient-space discrepancy $\Phi_j(y) - y$ into a latent-space discrepancy via $d E_k$.
\end{remark}
\begin{corollary}[Implicit cocycle control via reconstruction]\label{cor:cocycle-from-recon}
For an $(\varepsilon, \eta)$-approximate autoencoder atlas satisfying the regularity bounds of Theorem~\ref{thm:sign-cocycle-stability}, the integrand of $\mathcal{L}_{\mathrm{cocycle}}$ in Definition~\ref{def:cocycle-loss} is bounded pointwise, on every triple $(i,j,k)$ with $x \in U_i \cap U_j \cap U_k$, by
\[
\bigl\|T_{ki}(E_i(x)) - T_{kj}(T_{ji}(E_i(x)))\bigr\|^2 \leq L_E^2 (L_E L_D + 2)^2\, \varepsilon^2.
\]
Consequently
\[
\mathcal{L}_{\mathrm{cocycle}} \leq T_{\mathcal{U}}\, L_E^2 (L_E L_D + 2)^2\, \varepsilon^2,
\]
where $T_{\mathcal{U}} \coloneqq \sup_{x \in M}\#\{(i,j,k) : x \in U_i \cap U_j \cap U_k\}$ is the maximum number of triple overlaps meeting any single point.
\end{corollary}

\begin{proof}
The pointwise bound is the squared form of~\eqref{eq:cocycle-error-bound} in Theorem~\ref{thm:cocycle-error-approx}. Substituting into Definition~\ref{def:cocycle-loss} and using $\sum_{i,j,k} \rho_{ijk}(x) \leq T_{\mathcal{U}}$ for every $x$ gives the second bound.
\end{proof}

This is the precise mathematical content of the empirical observation in Section~\ref{sec:experiments} that minimising $\mathcal{L}_{\mathrm{recon}}$ alone already drives $\mathcal{L}_{\mathrm{cocycle}}$ to small values without an explicit cocycle term.

\subsection*{Jacobian regularity loss}

For the linearised transition maps $g_{ji}(x) = d(T_{ji})_{E_i(x)}$ to be well-defined elements of $\GL_d(\R)$, the encoders must be local diffeomorphisms, requiring the encoder Jacobian to have full rank.

\begin{definition}[Jacobian regularity loss]\label{def:jac-loss}
Let $J_{E_i}(x) = \frac{\partial E_i}{\partial x}(x) \in \R^{d \times N}$ be the encoder Jacobian, and let $\sigma_{\min}(A)$ denote the smallest singular value of $A$. For a threshold $\epsilon > 0$,
\[
\mathcal{L}_{\mathrm{jac}} = \mathbb{E}_{x \sim M}\left[ \sum_{i=1}^{n} \rho_i(x) \cdot \max\bigl( 0,\, \epsilon - \sigma_{\min}(J_{E_i}(x)) \bigr) \right].
\]
\end{definition}

This loss vanishes when $\sigma_{\min}(J_{E_i}(x)) \geq \epsilon$ for all $x \in U_i$, ensuring the encoder Jacobian has full row rank.

\begin{remark}[Effect on transition map non-degeneracy]
Combining the Jacobian regularity loss with smooth activations (e.g.\ $\tanh$) ensures that the encoder $E_i$ is a local diffeomorphism on $U_i$. Combined with a symmetric decoder regularity property, Proposition~\ref{prop:nondeg-bound} then gives a quantitative lower bound on the non-degeneracy gap $\delta(\mathcal{A})$, which is the operative condition for the stability theorems of Section~\ref{sec:stability}. The Jacobian regularity loss is therefore the principal optimisation lever for enforcing the hypotheses of Theorem~\ref{thm:sign-cocycle-stability}.
\end{remark}

\subsection*{Total loss}

The total loss is
\[
\mathcal{L}_{\mathrm{total}} = \mathcal{L}_{\mathrm{recon}} + \lambda_{\mathrm{jac}}\, \mathcal{L}_{\mathrm{jac}},
\]
where $\lambda_{\mathrm{jac}} \geq 0$ is a hyperparameter. By Corollary~\ref{cor:cocycle-from-recon}, no separate cocycle term is needed: $\mathcal{L}_{\mathrm{recon}} \to 0$ automatically drives $\mathcal{L}_{\mathrm{cocycle}} \to 0$ at the same rate.

\section{Experiments}\label{sec:experiments}

We validate our theoretical framework on manifolds with known orientability: the 2-sphere $S^2$ (orientable), the M\"obius band (non-orientable), the Klein bottle in $\R^4$ (non-orientable), and $\mathbb{R}P^2$ represented as line-patch images in $\R^{100}$ (non-orientable). The experiments demonstrate that (i) reconstruction loss alone enforces cocycle consistency without explicit regularisation (Lemma~\ref{lemma:cocycle-from-reconstruction} and Corollary~\ref{cor:cocycle-from-recon}); (ii) the sign cocycle correctly classifies orientability via the coboundary test (Corollary~\ref{cor:practical-orientability}); (iii) the diagnostic quantities $\eta$ and $\delta$ reliably distinguish successful from failed training; and (iv) the local stability theorem (Theorem~\ref{thm:local-stability}) explains the practical detection regime more accurately than the global Theorem~\ref{thm:sign-cocycle-stability}.

\subsection{Experimental setup}\label{sec:exp-setup}

\paragraph{Architecture.}
Each chart autoencoder consists of an encoder $E_i \colon \R^N \to \R^d$ and decoder $D_i \colon \R^d \to \R^N$, each with two hidden layers (32 and 16 units for the encoder, 16 and 32 for the decoder) and $\tanh$ activations. The latent dimension $d = 2$ matches the intrinsic dimension of all test manifolds.

\paragraph{Training.}
We optimise with Adam, learning rate $10^{-3}$, batch size $64$, for $1000$--$5000$ epochs. Each chart autoencoder is trained only on points assigned to its domain $U_i$. We use reconstruction loss alone, with no explicit cocycle regularisation, testing the prediction of Corollary~\ref{cor:cocycle-from-recon}. All experiments are repeated over $5$ random seeds; we report mean $\pm$ standard deviation.

\paragraph{Metrics.}
For each experiment we report quantities corresponding to the theoretical objects of Sections~\ref{sec:stability}--\ref{sec:loss-functions}:
\begin{itemize}
    \item \textbf{Reconstruction error} $\varepsilon = \sup_x \|D_i(E_i(x)) - x\|$: the quantity in Definition~\ref{def:approximate-atlas}(3).
    \item \textbf{Differential error} $\eta_{\mathrm{lat}} = \sup_x \|d(E_i \circ D_i)_{E_i(x)} - I_d\|_{\mathrm{op}}$: a computable diagnostic that proxies the tangent-restricted error $\eta$ of Definition~\ref{def:approximate-atlas}(4). The two quantities measure related but distinct objects. The quantity $\eta$ is the operator norm of $d(D_i \circ E_i)_x|_{T_xM} - \Id_{T_xM}$, controlling the ambient reconstruction map along tangent directions; $\eta_{\mathrm{lat}}$ is the operator norm of $d(E_i \circ D_i)_{E_i(x)} - I_d$ on $\R^d$. Both vanish on an exact autoencoder atlas, and on a well-trained smooth approximate atlas in low codimension they track each other closely. In high codimension $N - d$, training imperfections allow the decoder image $D_i(\R^d)$ to acquire components in directions normal to $M$; the subsequent encoder projection then inflates $\eta_{\mathrm{lat}}$ even when the cocycle-relevant tangent behaviour is comparatively well controlled. We therefore report $\eta_{\mathrm{lat}}$ as a diagnostic and interpret it with codimension in mind.
    \item \textbf{Non-degeneracy gap} $\delta = \min_{i,j,x} |\det g_{ji}(x)|$: the operative quantity in Theorems~\ref{thm:sign-cocycle-stability} and~\ref{thm:local-stability}.

    \item \textbf{Pairwise compatibility error} $\|T_{ji}(E_i(x))-E_j(x)\|$: measures the pariwise reconstruction-induced discrepancy. In the approximate setting it is bounded by $L_E\varepsilon$. This is the quantity we report in our experiments. It is related to the cocycle error of Definition~\ref{def:cocycle-loss} by Theorem~\ref{thm:cocycle-error-approx}.
\end{itemize}

\subsection{The 2-sphere: orientable manifold with good cover}\label{sec:exp-sphere}

\paragraph{Data and cover.}
We sample $n = 1000$ points uniformly from $S^2 \subset \R^3$ by drawing $x \sim \mathcal{N}(0, I_3)$ and normalising. We construct a four-chart good cover from the vertices of an inscribed regular tetrahedron
\[
v_0 = \tfrac{1}{\sqrt 3}(1,1,1),\quad v_1 = \tfrac{1}{\sqrt 3}(1,-1,-1),\quad v_2 = \tfrac{1}{\sqrt 3}(-1,1,-1),\quad v_3 = \tfrac{1}{\sqrt 3}(-1,-1,1),
\]
with $U_i = \{x \in S^2 : \langle x, v_i\rangle > -0.3\}$. The nerve of this cover is $\partial \Delta^3 \simeq S^2$, so by the Nerve Theorem, $\check{H}^*(\mathcal{U}; \Z/2)$ computes $H^*(S^2; \Z/2)$.

\begin{table}[ht]
\centering
\caption{Theoretical metrics for the $S^2$ experiment, computed from
a $4$-chart autoencoder atlas whose chart domains
$U_i=\{x\in S^2:\langle x,v_i\rangle>-0.3\}$ are induced by the vertices
of an inscribed regular tetrahedron in $\R^3$ (good cover; nerve
$\partial\Delta^3\simeq S^2$). All quantities are reported as
mean $\pm$ standard deviation over $5$ independent training runs.
Symbols: $\varepsilon$, the maximum per-chart sup reconstruction error
$\sup_{x\in U_i}\|D_i(E_i(x))-x\|$; $\bar\varepsilon$, the empirical mean
reconstruction error $\mathbb{E}_{x\sim M}\|D_i(E_i(x))-x\|$;
$\eta_{\mathrm{lat}}$, the latent-side differential proxy
$\sup_x\|d(E_i\circ D_i)_{E_i(x)}-I_d\|_{\mathrm{op}}$ used as a
practical estimator of $\eta$ in
Definition~\ref{def:approximate-atlas}(4); $\delta$, the global
non-degeneracy gap $\min_{i,j,x}|\det g_{ji}(x)|$ over all overlaps;
``Compatibility err'', the pairwise transition residual
$\|T_{ji}(E_i(x))-E_j(x)\|$; and $\sigma_{\min}(dE)$, the minimum
singular value of the encoder Jacobian, $\min_i\sigma_{\min}(dE_i)$.
The bottom row reports the success rate of orientability detection
via the coboundary test of Corollary~\ref{cor:practical-orientability}.}
\label{tab:s2_metrics}
\begin{tabular}{lcc}
\toprule
Metric & Value & Theoretical role \\
\midrule
$\varepsilon$            & $0.032 \pm 0.008$ & $\sup_x \|D_i(E_i(x)) - x\|$ \\
$\bar\varepsilon$        & $0.007 \pm 0.001$ & $\mathbb{E}_x \|D_i(E_i(x)) - x\|$ \\
$\eta_{\mathrm{lat}}$    & $0.54  \pm 0.19$  & $\sup_x \|d(E_i \circ D_i)_{E_i(x)} - I_d\|_{\mathrm{op}}$ \\
$\delta$                 & $0.101 \pm 0.016$ & $\min_{i,j,x} |\det g_{ji}(x)|$ \\
Compatibility err        & $0.008 \pm 0.001$ & $\|T_{ji}(E_i(x)) - E_j(x)\|$ \\
$\sigma_{\min}(dE)$      & $0.66  \pm 0.05$  & $\min_i \sigma_{\min}(dE_i)$ \\
\midrule
\multicolumn{3}{c}{\textbf{Orientability detection: 100\% (5/5 trials)}} \\
\bottomrule
\end{tabular}
\end{table}

\paragraph{Results.}
As reported in Table~\ref{tab:s2_metrics}, the hypotheses of Theorem~\ref{thm:sign-cocycle-stability} are satisfied: $\eta_{\mathrm{lat}} < 1$, $\delta \approx 0.10 > 0$, and the pairwise compatibility error is $O(10^{-3})$, consistent with the encoder-Lipschitz bound $\|T_{ji}(E_i(x)) - E_j(x)\| \leq L_E\,\varepsilon$. The coboundary test finds a consistent orientation assignment $(\nu_0, \nu_1, \nu_2, \nu_3) = (+1, -1, +1, -1)$ satisfying $\omega_{ji} = \nu_j \cdot \nu_i$ on all six pairwise overlaps, confirming $[\omega] = 0 \in \check{H}^1(S^2; \Z/2)$. By Corollary~\ref{cor:practical-orientability}, $S^2$ is correctly detected as orientable in all five trials.

\subsection{The M\"obius band: non-orientable manifold with disconnected overlap}\label{sec:exp-mobius}

\paragraph{Data and cover.}
We sample $n = 1500$ points from the standard immersion of the M\"obius band in $\R^3$:
\[
x(u, v) = \bigl(1 + \tfrac{v}{2}\cos\tfrac{u}{2}\bigr)\cos u,\quad
y(u, v) = \bigl(1 + \tfrac{v}{2}\cos\tfrac{u}{2}\bigr)\sin u,\quad
z(u, v) = \tfrac{v}{2}\sin\tfrac{u}{2},
\]
with $u \in [0, 2\pi)$, $v \in [-1, 1]$. We use a two-chart cover by partitioning along the $y$-coordinate:
\[
U_0 = \{x \in M : y(x) > -0.3\},\qquad U_1 = \{x \in M : y(x) < 0.3\}.
\]
The overlap $U_0 \cap U_1$ consists of \emph{two disconnected components} related by the M\"obius twist. By Remark~\ref{rem:relaxed-nerve}, each connected component is treated separately in the cocycle computation.

\begin{table}[ht]
\centering
\caption{Theoretical metrics for the M\"obius-band experiment,
computed from a $2$-chart autoencoder atlas with chart domains
$U_0=\{y(x)>-0.3\}$ and $U_1=\{y(x)<0.3\}$ on a sample of $n=1500$
points from the standard immersion of the M\"obius band in $\R^3$.
The overlap $U_0\cap U_1$ consists of two disconnected components
related by the M\"obius twist; each is treated as a separate
$1$-simplex of the nerve as per Remark~\ref{rem:relaxed-nerve}.
All quantities are reported as mean $\pm$ standard deviation over
$5$ independent training runs.
Symbols: $\varepsilon$, the maximum per-chart sup reconstruction error
$\sup_{x\in U_i}\|D_i(E_i(x))-x\|$; $\bar\varepsilon$, the empirical mean
reconstruction error; $\eta_{\mathrm{lat}}$, the latent-side
differential proxy $\sup_x\|d(E_i\circ D_i)_{E_i(x)}-I_d\|_{\mathrm{op}}$;
$\delta$, the global non-degeneracy gap
$\min_{i,j,x}|\det g_{ji}(x)|$; ``Compatibility err'', the pairwise
transition residual $\|T_{ji}(E_i(x))-E_j(x)\|$; and $\sigma_{\min}(dE)$,
the minimum singular value of the encoder Jacobian. The bottom row
reports the success rate of non-orientability detection by the
coboundary test of Corollary~\ref{cor:practical-orientability}:
no assignment of $\nu_0,\nu_1\in\{\pm 1\}$ can satisfy
$\omega_{10}=\nu_1\cdot\nu_0$ on both overlap components
simultaneously, certifying $w_1\neq 0$.}
\label{tab:mobius_metrics}
\begin{tabular}{lcc}
\toprule
Metric & Value & Theoretical role \\
\midrule
$\varepsilon$            & $0.098 \pm 0.018$ & $\sup_x \|D_i(E_i(x)) - x\|$ \\
$\bar\varepsilon$        & $0.021 \pm 0.002$ & $\mathbb{E}_x \|D_i(E_i(x)) - x\|$ \\
$\eta_{\mathrm{lat}}$    & $0.47  \pm 0.12$  & $\sup_x \|d(E_i \circ D_i)_{E_i(x)} - I_d\|_{\mathrm{op}}$ \\
$\delta$                 & $0.36  \pm 0.18$  & $\min_{i,j,x} |\det g_{ji}(x)|$ \\
Compatibility err        & $0.027 \pm 0.004$ & $\|T_{ji}(E_i(x)) - E_j(x)\|$ \\
$\sigma_{\min}(dE)$      & $0.55  \pm 0.09$  & $\min_i \sigma_{\min}(dE_i)$ \\
\midrule
\multicolumn{3}{c}{\textbf{Non-orientability detection: 100\% (5/5 trials)}} \\
\bottomrule
\end{tabular}
\end{table}

\paragraph{Detection mechanism.}
On the two components of $U_0 \cap U_1$, the sign cocycle takes opposite values:
\[
\omega_{10}|_{\text{Component 0}} = -1,\qquad \omega_{10}|_{\text{Component 1}} = +1.
\]
For $\omega_{10}$ to be a coboundary, we would need $\nu_0, \nu_1 \in \{\pm 1\}$, each constant on the connected charts $U_0, U_1$, with $\omega_{10}(x) = \nu_1 \cdot \nu_0$ on both components simultaneously. This is impossible: one component requires $\nu_1 \cdot \nu_0 = -1$, the other $+1$. By Proposition~\ref{prop:main-orientability} and Corollary~\ref{cor:practical-orientability}, $w_1(\mathcal{T}_{\mathcal{A}}) \neq 0$, correctly detecting non-orientability in all five trials. The metrics summarised in Table~\ref{tab:mobius_metrics} confirm detection on all five trials.

\begin{figure}[htbp]
    \centering
    \begin{subfigure}[t]{0.48\textwidth}
        \centering
        \includegraphics[width=\linewidth]{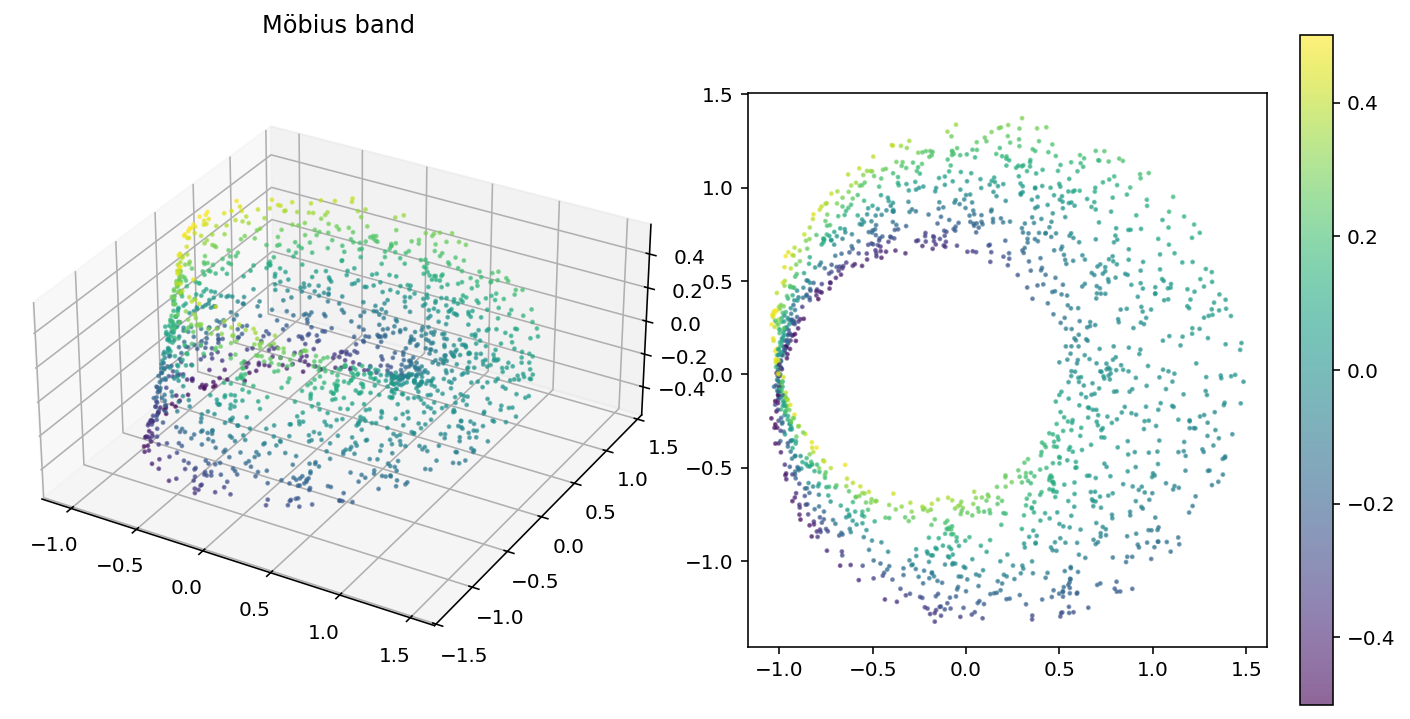}
        \caption{Sampled M\"obius band in $\R^3$ with two-chart cover.}
        \label{fig:mobius-data}
    \end{subfigure}
    \hfill
    \begin{subfigure}[t]{0.48\textwidth}
        \centering
        \includegraphics[width=\linewidth]{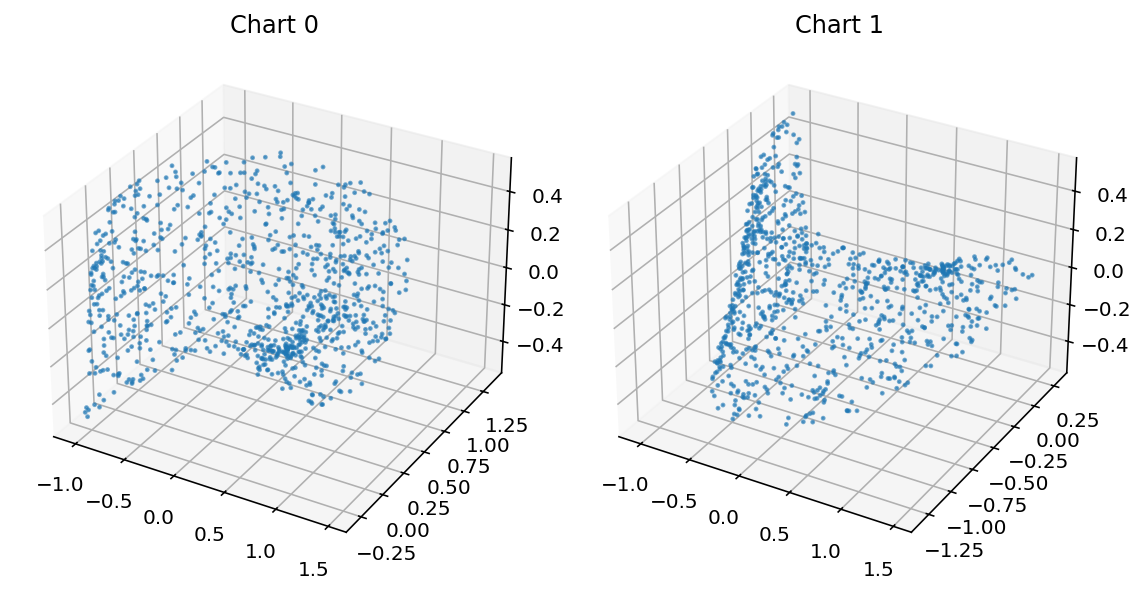}
        \caption{Chart domains and decomposition of $U_0 \cap U_1$ into two components.}
        \label{fig:mobius-charts}
    \end{subfigure}

    \vspace{0.4cm}

    \begin{subfigure}[t]{0.35\textwidth}
        \centering
        \includegraphics[width=\linewidth]{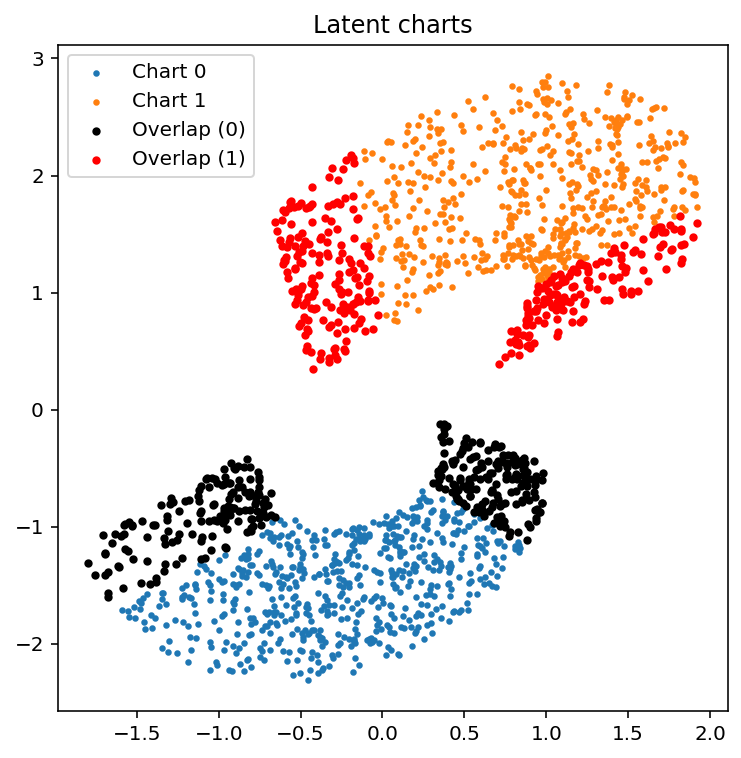}
        \caption{Latent representations $E_0(U_0)$ and $E_1(U_1)$ in $\R^2$.}
        \label{fig:mobius-latent}
    \end{subfigure}
    \hfill
    \begin{subfigure}[t]{0.35\textwidth}
        \centering
        \includegraphics[width=\linewidth]{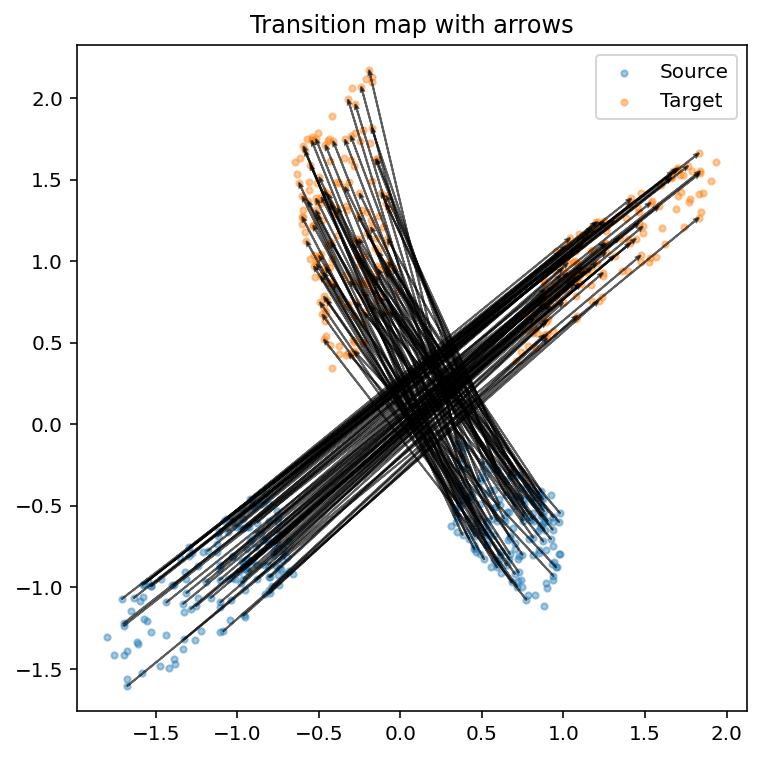}
        \caption{Transition map $T_{10}$ on the two overlap components.}
        \label{fig:mobius-transitions}
    \end{subfigure}

    \caption{Autoencoder atlas for the M\"obius band. The two-chart cover produces an overlap with two disconnected components on which the sign cocycle takes opposite values, detecting non-orientability.}
    \label{fig:mobius-experiment}
\end{figure}

\subsection{Higher-dimensional manifolds: Klein bottle and $\mathbb{R}P^2$}\label{sec:exp-higher-dim}

These experiments test the framework on manifolds where covers are learned from data and training convergence is not guaranteed. 

\paragraph{Cover construction versus cover certification.}
Cover construction from data is a well-studied problem~\cite{singhMC07,scoccola2025coverlearninglargescaletopology}, but cover \emph{certification} — verifying contractibility of each chart and each connected component of each finite intersection from finite samples — is a separate and presently open problem. Our pipeline uses heuristic constructions (landmark-based geodesic balls; DBSCAN-based connected-component decomposition) and does not certify the good-cover property. Theorem~\ref{thm:good-cover} therefore enters our experiments as a hypothesis we cannot fully verify; the diagnostic role of $\delta$ and $\eta$ provides post-training detection of regimes where the framework is unreliable. We regard this as acceptable for two reasons: the theoretical framework degrades gracefully when the good-cover condition is only approximately satisfied (Remark~\ref{rem:relaxed-nerve}), and principled cover learning lies outside the scope of this paper. 

A key finding is that not all training runs produce valid 
atlases. We identify three diagnostic criteria that reliably 
distinguish successful from failed runs \emph{before} 
examining orientability results:
\begin{enumerate}
    \item \textbf{Per-chart differential error compatible 
    with the local stability theorem.}  The global 
    Theorem~\ref{thm:sign-cocycle-stability} requires $\eta<1$ which, in our case we approximate as 
    $\eta_{\mathrm{lat},i} < 1$ on \emph{every} chart (which in low codimension closely tracks the cocycle-relevant $\eta_i$).  
    The local Theorem~\ref{thm:local-stability} is 
    strictly weaker: via the re-indexing of 
    Remark~\ref{rem:reindexing}, a single chart with 
    $\eta_{\mathrm{lat},i} \gg 1$ can be placed in the 
    encoder-only slot of every triangle, provided all 
    other charts satisfy $\eta_{\mathrm{lat},j} < 1$.  
    Multiple charts with $\eta_{\mathrm{lat}} \geq 1$ 
    cannot be simultaneously hidden and indicate that 
    neither stability theorem applies.
    
    \item \textbf{Positive non-degeneracy gap.} $\delta > 0$ 
    should be verifiable post-training; small $\delta$ 
    signals that the sign cocycle is unreliable.
    
    \item \textbf{No mixed signs within overlap components.}  
    Both stability theorems require $\delta > 0$, i.e.\ that 
    $\det g_{ji}(x)$ never vanishes on any overlap.  Since 
    $\delta$ is computed at finitely many sample points, a 
    positive sampled value does not guarantee the true 
    $\delta$ is positive: the determinant may pass through 
    zero between samples.  Mixed signs within a single 
    connected overlap component provide direct evidence of 
    this: by the intermediate value theorem, if 
    $\mathrm{sign}(\det g_{ji})$ takes both values $\pm 1$ 
    on a connected set, then $\det g_{ji}$ vanishes 
    somewhere in that set, and the true $\delta = 0$ on 
    that overlap regardless of the sampled minimum.  We 
    therefore classify a trial as non-converged whenever 
    any overlap component exhibits mixed signs.
\end{enumerate}
The mixed-signs diagnostic is intertwined with the 
    good-cover hypothesis.  Sign constancy on connected 
    overlap components follows from the intermediate value 
    theorem applied to the continuous function 
    $x \mapsto \det g_{ji}(x)$, but ``connected'' here 
    refers to the true manifold topology, which our 
    pipeline approximates via DBSCAN clustering on the 
    point cloud.  Two distinct failure modes can produce 
    mixed signs within a single DBSCAN component:
    \begin{enumerate}[label=(\alph*)]
        \item \emph{Genuine zero crossing 
        ($\delta = 0$).}  The transition map Jacobian 
        determinant passes through zero within a truly 
        connected overlap region, violating the 
        non-degeneracy hypothesis of both stability 
        theorems.  This indicates a poorly trained chart.
        \item \emph{Incorrectly merged components 
        (cover artefact).}  Two genuinely disconnected 
        overlap components --- on which opposite signs are 
        the expected topological signal for a non-orientable 
        manifold --- are merged into one by an overly 
        coarse DBSCAN threshold.  This is a failure of the 
        connected-component decomposition, not of the 
        atlas, and is related to the uncertified good-cover 
        hypothesis: if the decomposition does not correctly 
        resolve the connected components of each 
        intersection, the \v{C}ech complex is 
        misrepresented.
    \end{enumerate}
    In our experiments, case~(a) and case~(b) can be 
    distinguished by the sampled $\delta$: genuine zero 
    crossings produce small $\delta$ on the affected 
    overlaps, whereas decomposition artefacts typically 
    have healthy $\delta$ since the determinant is 
    well-behaved on each true component.  In the 
    non-converged Klein bottle trials, the mixed-sign 
    overlaps consistently involve the outlier chart and 
    have small sampled $\delta$ ($\approx 0.03$), 
    supporting interpretation~(a).  We classify a trial 
    as non-converged whenever any overlap component 
    exhibits mixed signs accompanied by small $\delta$; 
    mixed signs with large $\delta$ would instead 
    indicate a decomposition artefact requiring 
    refinement of the DBSCAN threshold.

\subsubsection{Klein bottle}\label{sec:exp-klein}

\paragraph{Data, cover, and training.}
We sample $n = 1000$ points from the standard immersion of the Klein bottle in $\R^4$,
\[
\iota(u, v) = \bigl((m + \cos v)\cos u,\, (m + \cos v)\sin u,\, \sin v \cos(u/2),\, \sin v \sin(u/2)\bigr),
\]
with $m = 4$ and the identification $(u, v) \sim (u + 2\pi, 2\pi - v)$. We construct an $8$-chart cover by farthest-point sampling on the $100$-nearest-neighbour geodesic graph, with chart $U_i$ consisting of all sample points whose geodesic distance to landmark $i$ is below the 20th percentile of the landmark distance matrix. The retry mechanism is best-effort: when a run fails to reach $\varepsilon < \varepsilon_{\mathrm{thresh}}$ within 3 retries (as in trial~0), we record the final state and rely on the post-training diagnostics (criteria~(1)--(3) above) to flag the run as non-converged.

\begin{table}[ht]
\centering
\caption{Klein bottle: per-trial results over $5$ independent runs of an
$8$-chart autoencoder atlas. Charts are obtained by farthest-point
sampling on the $100$-nearest-neighbour geodesic graph of $n=1000$ points
drawn from the standard immersion of the Klein bottle in $\R^4$
(see Section~\ref{sec:exp-klein}); chart $U_i$ collects sample points
whose geodesic distance to landmark~$i$ falls below the $20$th
percentile of the landmark distance matrix. Columns:
``Seed'' is the random seed of the training run; $\varepsilon$ is the
maximum over charts of the sup reconstruction error
$\sup_{x\in U_i}\|D_i(E_i(x))-x\|$; $\bar\varepsilon$ is the empirical
mean reconstruction error; $\eta_{\mathrm{lat}}$ is the
maximum over charts of the latent-side differential proxy
$\sup_x\|d(E_i\circ D_i)_{E_i(x)}-I_d\|_{\mathrm{op}}$;
$\delta$ is the global non-degeneracy gap
$\min_{i,j,x}|\det g_{ji}(x)|$; $\sigma_{\min}(dE)$ is the minimum
singular value of the encoder Jacobian across charts;
``Detected'' is the orientability verdict returned by the coboundary
test of Corollary~\ref{cor:practical-orientability}; and
``Correct'' (\cmark/\xmark) indicates agreement with the ground
truth (non-orientable). Trials~$0$ and~$2$ are classified as
\emph{non-converged} by the post-training diagnostics
(criteria~(1)--(3) of Section~\ref{sec:exp-klein}): they exhibit
the highest reconstruction error and overlap components with mixed
signs in $\mathrm{sign}(\det g_{ji})$, indicating that the transition
map Jacobian determinant passes through zero within the affected
overlap.}
\label{tab:klein_per_trial}
\resizebox{\textwidth}{!}{%
\begin{tabular}{ccccccccc}
\toprule
Trial & Seed & $\varepsilon$ & $\bar\varepsilon$ & $\eta_{\mathrm{lat}}$ & $\delta$ & $\sigma_{\min}(dE)$ & Detected & Correct \\
\midrule
0 & 42 & 0.182 & 0.018 & 31.11 & 0.008 & 0.236 & Orientable     & \xmark \\
1 & 43 & 0.109 & 0.025 & 1.51  & 0.083 & 0.207 & Non-orientable & \cmark \\
2 & 44 & 0.141 & 0.018 & 6.04  & 0.016 & 0.240 & Orientable     & \xmark \\
3 & 45 & 0.131 & 0.025 & 0.62  & 0.090 & 0.212 & Non-orientable & \cmark \\
4 & 46 & 0.104 & 0.022 & 1.23  & 0.053 & 0.258 & Non-orientable & \cmark \\
\bottomrule
\end{tabular}}
\end{table}

\paragraph{Diagnostic analysis.}
\emph{Among the converged Klein bottle trials, the local 
theorem extends coverage beyond the global theorem.} Table~\ref{tab:klein_per_trial} reports the per-trial outcome. Only one of 
the three converged trials (seed 45) satisfy 
$\eta_{\mathrm{lat}} < 1$ across all charts and are covered by 
the global Theorem~\ref{thm:sign-cocycle-stability}.  Trial 4 (seed~46) has a single outlier chart with 
$\eta_{\mathrm{lat}} = 1.23$, while all remaining charts 
satisfy $\eta_{\mathrm{lat},i} \leq 0.91 < 1$, the same happen for trial 1 (seed 43).  The global 
theorem is violated:
\begin{center}
\begin{tabular}{ccc}
\toprule
Trial & $\eta_{\mathrm{lat}}$ (chart 4) & $\max_{i \neq 4}\, \eta_i$ \\
\midrule
0 (seed 42, \xmark) & 31.11 & 0.62 \\
2 (seed 44, \xmark) & 6.04  & 0.67 \\
\midrule
1 (seed 43, \cmark) & 0.71  & 1.51 \\
3 (seed 45, \cmark) & 0.60  & 0.62 \\
4 (seed 46, \cmark) & 1.23  & 0.91 \\
\bottomrule
\end{tabular}
\end{center}

Under the re-indexing of 
Remark~\ref{rem:reindexing}, placing the outlier chart in 
the encoder-only ($\gamma$-)slot, all 11~triangles of the 
nerve satisfy the per-triple condition 
$\eta^{(\alpha,\beta,\gamma)} < 1$, no overlap component 
exhibits mixed signs, and the per-triple non-degeneracy gaps 
remain positive.  
Theorem~\ref{thm:local-stability} therefore applies, and 
non-orientability is correctly detected.  The two 
non-converged trials (seeds 42, 44) exhibit the highest reconstruction error values and a qualitatively 
different failure mode: overlap components involving the 
outlier chart contain mixed signs in 
$\mathrm{sign}(\det g_{ji})$, indicating that the 
transition map Jacobian determinant passes through zero 
within those overlaps.  The per-triple non-degeneracy gaps 
on the affected triangles are correspondingly small 
($\delta^{(ijk)} \approx 0.008$ and $0.016$), and neither 
the global Theorem~\ref{thm:sign-cocycle-stability} nor the 
local Theorem~\ref{thm:local-stability} applies --- both 
require $\delta > 0$.  These trials are correctly excluded 
by the convergence diagnostics, confirming that the 
non-degeneracy gap, rather than the differential 
reconstruction error, is the binding constraint in this 
experiment. The two non-converged trials (seeds 42, 44) are identified 
by criterion~(3): overlap components involving the outlier 
chart exhibit mixed signs in $\mathrm{sign}(\det g_{ji})$, 
with small minorities (e.g.\ 1 out of 49 points, or 1 out 
of 92) taking the opposite sign from the majority.  While 
the sampled $\delta$ remains nominally positive 
($\delta \approx 0.008$ and $0.016$), the mixed signs 
establish via the intermediate value theorem that the true 
$\delta = 0$ on those overlaps, and neither stability 
theorem applies. The transition maps for trial 1 (seed 43) can be seen in Figure~\ref{fig:klein-transitions_supplement}.

\paragraph{Converged trials.}
Restricting to the three correct trials (seeds 43, 45, 46), the coboundary test correctly fails: no consistent $\{\nu_i\}$ assignment exists. Aggregate metrics on these converged trials are shown in Table~\ref{tab:klein_metrics}.

\begin{table}[ht]
\centering
\caption{Theoretical metrics for the Klein-bottle experiment,
restricted to the $3$ of $5$ trials (seeds $43$, $45$, $46$) that
satisfy the post-training convergence diagnostics. The atlas is the
$8$-chart cover described in Section~\ref{sec:exp-klein}; values are
mean $\pm$ standard deviation over the $3$ converged trials.
Symbols: $\varepsilon$, the maximum per-chart sup reconstruction error
$\sup_{x\in U_i}\|D_i(E_i(x))-x\|$; $\bar\varepsilon$, the empirical
mean reconstruction error; $\eta_{\mathrm{lat}}$, the latent-side
differential proxy
$\sup_x\|d(E_i\circ D_i)_{E_i(x)}-I_d\|_{\mathrm{op}}$ for the
differential reconstruction error $\eta$ of
Definition~\ref{def:approximate-atlas}(4); $\delta$, the global
non-degeneracy gap $\min_{i,j,x}|\det g_{ji}(x)|$; ``Compatibility err'',
the pairwise transition residual $\|T_{ji}(E_i(x))-E_j(x)\|$; and
$\sigma_{\min}(dE)$, the minimum singular value of the encoder
Jacobian across charts. The coboundary test of
Corollary~\ref{cor:practical-orientability} fails on all $3$
converged trials, correctly certifying non-orientability.}
\label{tab:klein_metrics}
\begin{tabular}{lcc}
\toprule
Metric & Value & Theoretical role \\
\midrule
$\varepsilon$            & $0.1147 \pm 0.0116$ & $\sup_x \|D_i(E_i(x)) - x\|$ \\
$\bar\varepsilon$        & $0.0240 \pm 0.0013$ & $\mathbb{E}_x \|D_i(E_i(x)) - x\|$ \\
$\eta_{\mathrm{lat}}$    & $1.1201 \pm 0.3729$ & $\sup_x \|d(E_i \circ D_i)_{E_i(x)} - I_d\|_{\mathrm{op}}$ \\
$\delta$                 & $0.0756 \pm 0.0160$ & $\min_{i,j,x} |\det g_{ji}(x)|$ \\
Compatibility err        & $0.0134 \pm 0.0010$ & $\|T_{ji}(E_i(x)) - E_j(x)\|$ \\
$\sigma_{\min}(dE)$      & $0.2256 \pm 0.0227$ & $\min_i \sigma_{\min}(dE_i)$ \\
\bottomrule
\end{tabular}
\end{table}

\subsubsection{$\mathbb{R}P^2$ via line patches}\label{sec:exp-rp2}

\paragraph{Data, cover, and training.}
We generate $5625$ line-patch images ($10 \times 10$ grayscale, ambient dimension $100$, see Fig.~\ref{fig:rp2-experiment}) following the construction in~\cite{10.1007/s00454-017-9927-2}: each image is a blurred line segment, and since a line at angle $\theta$ is identical to a line at angle $\theta + \pi$, the persistent cohomology of this space is consistent with $\mathbb{R}P^2$. We use a $10$-chart cover from projective landmarks, no Jacobian regularisation ($\lambda_{\mathrm{jac}} = 0$), $\tanh$ activations, and $1000$ training epochs.

\begin{figure}[htbp]
    \centering
    \begin{subfigure}[t]{0.4\textwidth}
        \centering
        \includegraphics[width=\linewidth]{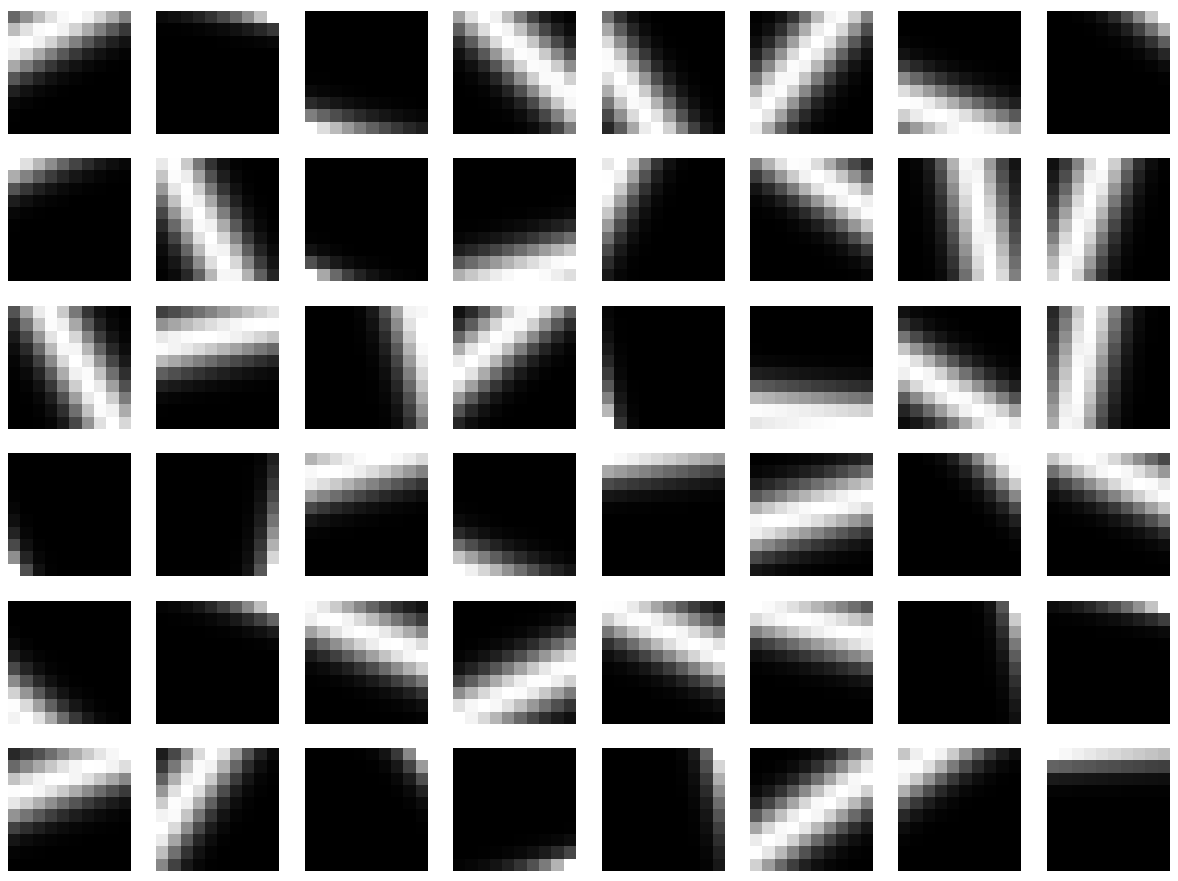}
        \caption{Sample line patch images.}
        \label{fig:rp2-patches}
    \end{subfigure}
    \quad
    \begin{subfigure}[t]{0.45\textwidth}
        \centering
        \includegraphics[width=\linewidth]{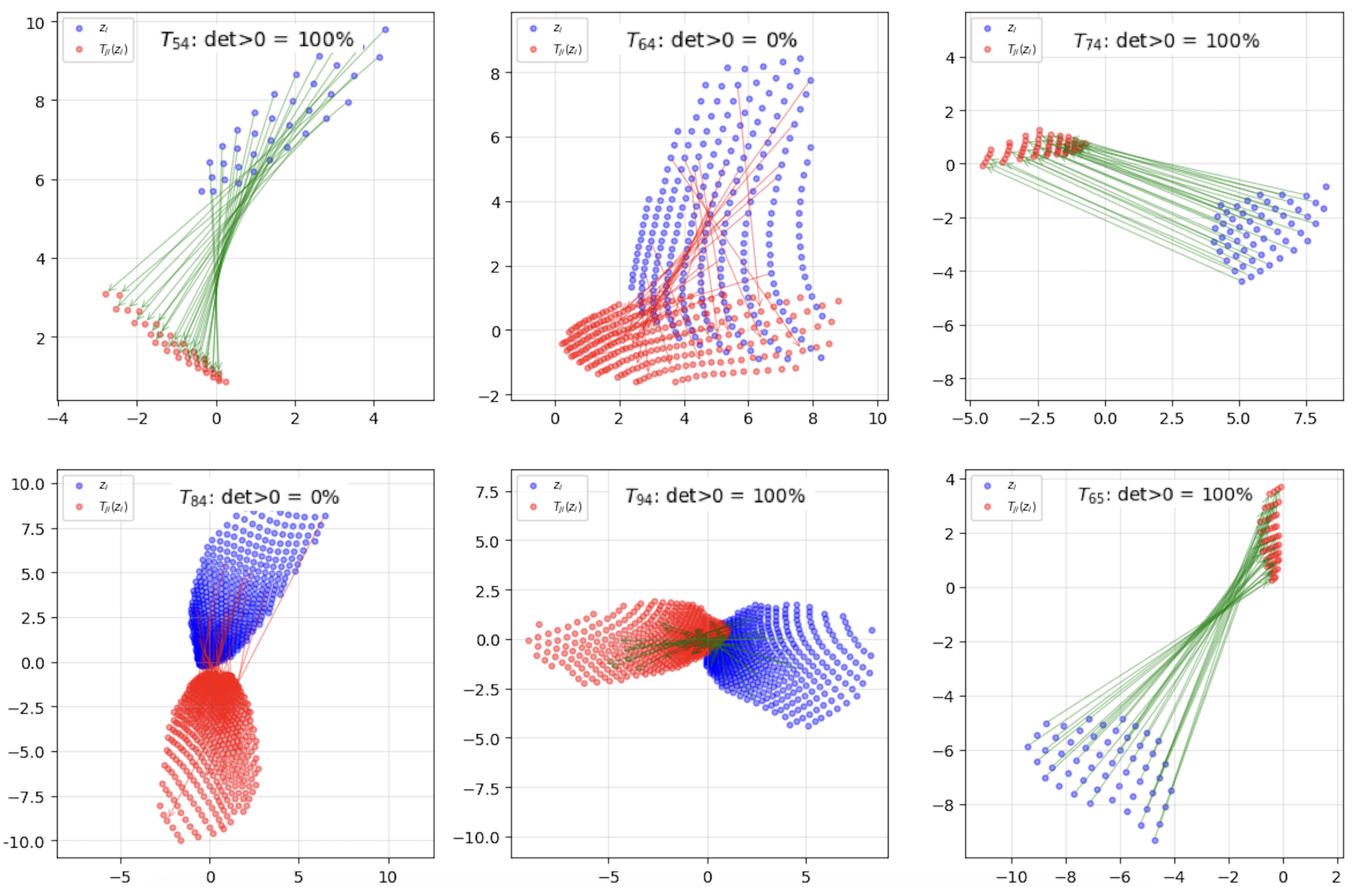}
        \caption{Representative transition maps.}
        \label{fig:rp2-sample-transitions}
    \end{subfigure}
    \caption{The $\mathbb{RP}^2$ line patches experiment. (a)~Sample $10\times10$ grayscale patches. (b)~Selected transition maps $T_{ji}=E_j\circ D_i$: source points (blue) and images (red) in $\mathbb{R}^2$. The complete set of all pairwise transitions is in Figure~\ref{fig:rp2-transitions_supplement}.}
    \label{fig:rp2-experiment}
\end{figure}

\paragraph{Results.}
Of five training runs, four satisfied the convergence criteria. In all four, the coboundary test correctly failed, detecting non-orientability. The sign distributions across the 25 overlap components are balanced but asymmetric (11:14, 16:9, 12:13, 15:10 across the four converged trials), confirming that the nerve graph admits no consistent 2-colouring. Aggregate metrics across the four converged trials are reported in Table~\ref{tab:rp2_converged}.

\begin{table}[ht]
\centering
\caption{$\mathbb{R}P^2$ line patches: summary metrics for converged
trials. Data are $5625$ blurred $10\times 10$ line-patch images
(ambient dimension $N=100$) following~\cite{10.1007/s00454-017-9927-2};
because a line at angle $\theta$ is identical to a line at angle
$\theta+\pi$, the resulting point cloud has the persistent cohomology
of $\mathbb{R}P^2$ (intrinsic dimension $d=2$). The atlas uses a
$10$-chart cover from projective landmarks; training uses $\tanh$
activations, no Jacobian regularisation ($\lambda_{\mathrm{jac}}=0$),
and $1000$ epochs. Of $5$ training runs, $4$ satisfy the post-training
convergence diagnostics; each entry below is the per-trial value
(e.g.\ for $\eta_{\mathrm{lat}}$, the maximum over charts), summarised
as mean $\pm$ standard deviation over the $4$ converged trials.
Symbols: $\varepsilon$, sup reconstruction error
$\sup_{x\in U_i}\|D_i(E_i(x))-x\|$; $\bar\varepsilon$, empirical mean
reconstruction error; $\eta_{\mathrm{lat}}$, latent-side differential
proxy $\sup_x\|d(E_i\circ D_i)_{E_i(x)}-I_d\|_{\mathrm{op}}$;
$\delta$, global non-degeneracy gap $\min_{i,j,x}|\det g_{ji}(x)|$;
``Compatibility error'', pairwise transition residual
$\|T_{ji}(E_i(x))-E_j(x)\|$; $\sigma_{\min}(dE)$, minimum singular value
of the encoder Jacobian across charts. The global hypotheses of
Theorem~\ref{thm:sign-cocycle-stability} ($\eta_{\mathrm{lat}}<1$) are
violated; nonetheless $\delta>0$ holds and the coboundary test
correctly fails on all $4$ converged trials, certifying
non-orientability. The discrepancy is explained by the
codimension-driven gap between $\eta_{\mathrm{lat}}$ and~$\eta$
discussed in Section~\ref{sec:exp-rp2}.}
\label{tab:rp2_converged}
\begin{tabular}{lc}
\toprule
Metric & Value \\
\midrule
$\varepsilon$ (sup reconstruction error)        & $0.309 \pm 0.081$ \\
$\bar\varepsilon$ (mean reconstruction error)   & $0.060 \pm 0.003$ \\
$\eta_{\mathrm{lat}}$ (differential error)      & $12.90 \pm 6.91$  \\
$\delta$ (non-degeneracy gap)                   & $0.027 \pm 0.010$ \\
Compatibility error                             & $0.042 \pm 0.008$ \\
$\sigma_{\min}(dE)$ (encoder regularity)        & $0.541 \pm 0.064$ \\
\bottomrule
\end{tabular}
\end{table}

These trials have $\eta_{\mathrm{lat}} \gg 1$ across the board, so the global hypotheses of Theorem~\ref{thm:sign-cocycle-stability} are violated. Nevertheless, $\delta > 0$ and sign consistency hold within each overlap, and correct detection is achieved. We analyse this regime through the lens of Theorem~\ref{thm:local-stability} in the next subsection.

\paragraph{Diagnosis analysis.} \emph{The $\mathbb{R}P^2$ results are not explained by 
either stability theorem, but are consistent with a 
codimension-driven gap between $\eta_{\mathrm{lat}}$ 
and~$\eta$.}  In the low-codimension experiments ($S^2$, 
M\"obius band, Klein bottle), the ambient dimension $N$ 
exceeds the intrinsic dimension $d$ by at most~$2$, and 
$\eta_{\mathrm{lat}} \approx \eta$: the latent round-trip 
diagnostic faithfully tracks the tangent-restricted 
reconstruction quality that enters the stability theorems.  
In the $\mathbb{R}P^2$ experiment, the codimension is 
$N - d = 98$.  Here $\eta_{\mathrm{lat}} \gg 1$ uniformly 
across all charts, so the re-indexing mechanism of 
Remark~\ref{rem:reindexing} does not apply and neither 
stability theorem can be invoked.  However, 
$\eta_{\mathrm{lat}}$ and $\eta$ measure different objects: 
$\eta_{\mathrm{lat}}$ is the operator norm of 
$d(E_i \circ D_i)_{E_i(x)} - I_d$ on all of~$\R^d$, whereas 
$\eta$ is the operator norm of 
$d(D_i \circ E_i)_x|_{T_xM} - \Id_{T_xM}$ restricted to 
tangent directions.  When the codimension is large, 
imperfections in the decoder allow its image to acquire 
components in normal directions; the subsequent encoder 
projection annihilates these components, inflating 
$\eta_{\mathrm{lat}}$ even when the tangent-restricted 
behaviour is well controlled.  The successful detection in 
the $\mathbb{R}P^2$ trials ($\delta > 0$, no mixed signs, 
correct non-orientability on all four converged runs) is 
therefore consistent with the tangent-restricted $\eta$ 
satisfying the stability bounds even though 
$\eta_{\mathrm{lat}}$ does not.  Measuring $\eta$ directly 
--- rather than through the latent proxy 
$\eta_{\mathrm{lat}}$ --- is a concrete instrumentation task 
for future work that would determine whether the stability 
theorems apply in this high-codimension regime.

\paragraph{Role of Jacobian regularisation.}
Ablation on the Klein bottle without Jacobian regularisation gives 0\% accuracy (no run satisfies the convergence criteria). On $\mathbb{R}P^2$, convergence is achieved on 4 of 5 trials without regularisation. The role of $\mathcal{L}_{\mathrm{jac}}$ thus depends on the cover structure and data geometry, consistent with its role of enforcing the non-degeneracy hypothesis of Theorem~\ref{thm:sign-cocycle-stability} when sample geometry is otherwise insufficient.

Figure~\ref{fig:rp2-transitions_supplement} in Appendix~\ref{app:proofs} shows the full panel of pairwise transition maps for the line-patch experiment; the mix of orientation-preserving and orientation-reversing transitions confirms the non-trivial sign cocycle.

\subsection{Summary and discussion}\label{sec:exp-summary}

\begin{table}[ht]
\centering
\caption{Summary of all four orientability-detection experiments
reported in Section~\ref{sec:experiments}, restricted to runs that
satisfy the post-training convergence diagnostics. Columns:
``Manifold'' lists the test manifold; ``Dim'' is its intrinsic
dimension~$d$; ``Ambient'' is the ambient space $\R^N$ used for
sampling (or, for $\mathbb{R}P^2$, the line-patch image space);
``Ground truth'' is the known orientability class of the manifold;
$\bar\varepsilon=\mathbb{E}_x\|D_i(E_i(x))-x\|$ is the empirical mean
reconstruction error; $\delta=\min_{i,j,x}|\det g_{ji}(x)|$ is the
global non-degeneracy gap; ``Converged'' is the number of converged
runs out of $5$; and ``Accuracy'' is the proportion of converged
runs on which the coboundary test of
Corollary~\ref{cor:practical-orientability} returns the correct
orientability verdict. Across the four manifolds, $17$ of $20$ runs
converge and the coboundary test is correct on all $17$ of them.}
\label{tab:summary_all}
\begin{tabular}{lccccccc}
\toprule
Manifold & Dim & Ambient & Ground truth & $\bar\varepsilon$ & $\delta$ & Converged & Accuracy \\
\midrule
$S^2$                       & 2 & $\R^3$    & Orientable     & 0.007 & 0.101 & 5/5 & 100\% \\
M\"obius band               & 2 & $\R^3$    & Non-orientable & 0.021 & 0.36  & 5/5 & 100\% \\
Klein bottle                & 2 & $\R^4$    & Non-orientable & 0.024 & 0.076 & 3/5 & 100\% \\
$\mathbb{R}P^2$ (patches)   & 2 & $\R^{100}$& Non-orientable & 0.060 & 0.027 & 4/5 & 100\% \\
\bottomrule
\end{tabular}
\end{table}

Aggregating across all four manifolds (Table~\ref{tab:summary_all}), the experiments support four conclusions:

\begin{enumerate}
    \item \textbf{Theoretical correctness.} When atlas diagnostics are satisfied, orientability detection achieves 100\% accuracy across all tested manifolds, validating Corollary~\ref{cor:practical-orientability}.
    \item \textbf{Cocycle consistency from reconstruction.} All converged trials achieved cocycle errors below 0.08 using reconstruction loss alone.
    \item \textbf{Practical diagnostics.} Moderate per-chart $\eta_{\mathrm{lat},i}$ and positive $\delta$ provide a reliable method for assessing atlas validity \emph{before} the coboundary test, without requiring knowledge of the ground truth.
    \item \textbf{Scalability.} The $\mathbb{R}P^2$ experiment demonstrates correct detection in ambient dimension $\R^{100}$, validating the framework for high-dimensional image and signal data.
\end{enumerate}

\paragraph{Computational considerations.}
The Jacobian computation $g_{ji}(x) = d(T_{ji})_{E_i(x)}$ uses automatic differentiation through the encoder--decoder composition; for $\tanh$ activations this is well-conditioned in any ambient dimension. The coboundary test reduces to checking whether the nerve graph admits a consistent 2-colouring, solvable in linear time. The overall cost is dominated by autoencoder training.

\section{Conclusions and future work}\label{sec:conclusions}

We have introduced a framework that treats collections of locally trained autoencoders as learned atlases on data manifolds, connecting neural network representations to classical vector bundle theory. The key insight is that reconstruction consistency alone enforces the cocycle condition on transition maps, so that linearising these transitions yields a vector bundle whose first Stiefel--Whitney class can be read off from signs of Jacobian determinants. The stability theory developed in Section~\ref{sec:stability} establishes this on two levels: Theorem~\ref{thm:sign-cocycle-stability} gives a global stability result under uniform regularity bounds, and Theorem~\ref{thm:local-stability} refines this to a per-triple statement in which the differential reconstruction condition is required on only two of the three charts of each triangle of the nerve. Theorem~\ref{thm:w1-agreement} then identifies the cohomology class of the learned sign cocycle with $w_1(TM)$. Corollary~\ref{cor:cocycle-from-recon} shows that the pointwise cocycle error is bounded by the reconstruction error at the same rate, explaining why no explicit cocycle term is needed in the loss; and Corollary~\ref{cor:min-charts} identifies the minimum chart count with the covering type. Experiments on $S^2$, the M\"obius band, the Klein bottle, and $\mathbb{R}P^2$ line patches validate the framework, achieving 100\% orientability detection on converged trials without explicit cocycle regularisation. The re-indexing mechanism of Remark~\ref{rem:reindexing}, exploiting the asymmetric role 
of the three charts in the local theorem 
(Remark~\ref{rem:asymmetric}), extends the theoretical 
coverage of the stability framework beyond the global 
theorem in the Klein bottle experiment, where the global 
hypothesis $\eta < 1$ is violated but detection succeeds 
under per-triple analysis.  The $\mathbb{R}P^2$ experiment 
demonstrates correct detection in a regime where neither 
stability theorem can be invoked, suggesting that the 
tangent-restricted $\eta$ --- which is not directly 
measured --- may satisfy the stability bounds even when the 
latent proxy $\eta_{\mathrm{lat}}$ does not.  The full 
ordering independence of the resulting cohomology class is 
recovered through Theorem~\ref{thm:w1-agreement}.

Several natural extensions remain open. The most direct is extracting higher Stiefel--Whitney classes (beginning with $w_2$, the obstruction to spin structure) and Chern classes (in the complex setting, via $\det g_{ji} / |\det g_{ji}|$) from the full linearised transition cocycle. A second is the integration of principled cover learning methods~\cite{scoccola2025coverlearninglargescaletopology} to bridge the gap between good covers and data-driven covers, which would close one of the practical assumptions of our framework. Finally, scaling to higher intrinsic dimensions in applied domains such as robotics, materials science, or molecular configuration spaces would test the framework against data where the underlying manifold structure is not known a priori.

\section*{Acknowledgments}
Eduardo Paluzo-Hidalgo acknowledges funding from the European Union's Horizon 2020 research and innovation programme under the Marie Sk\l{}odowska-Curie grant agreement No.\ 101153039 (CHALKS), and by the IMUS--Mar\'ia de Maeztu grant CEX2024-001517-M (Apoyo a Unidades de Excelencia Mar\'ia de Maeztu), funded by MICIU/AEI/10.13039/501100011033.
Yuichi Ike is supported by JSPS Grant-in-Aid for Transformative Research Areas (A) Grant Number JP22H05107 and JST, CREST Grant Number JPMJCR24Q1, Japan.

Eduardo Paluzo-Hidalgo thanks Prof.\ Gunnar Carlsson for introducing the initial problem that inspired this paper during his research stay at Stanford University, and Dr.\ Chunyin Siu for insightful conversations during the early stages of this research. The authors acknowledge the use of AI tools to polish the written text.

\section*{Code availability}
\url{https://github.com/EduPH/learningtangentbundle}.

\bibliographystyle{plain}
\bibliography{references}

@misc{schonsheck2020chartautoencodersmanifoldstructured,
      title={Chart Auto-Encoders for Manifold Structured Data}, 
      author={Stefan Schonsheck and Jie Chen and Rongjie Lai},
      year={2020},
      eprint={1912.10094},
      archivePrefix={arXiv},
      primaryClass={cs.LG},
      url={https://arxiv.org/abs/1912.10094}, 
}

@misc{schonsheck2024semisupervisedmanifoldlearningcomplexity,
      title={Semi-Supervised Manifold Learning with Complexity Decoupled Chart Autoencoders}, 
      author={Stefan C. Schonsheck and Scott Mahan and Timo Klock and Alexander Cloninger and Rongjie Lai},
      year={2024},
      eprint={2208.10570},
      archivePrefix={arXiv},
      primaryClass={cs.LG},
      url={https://arxiv.org/abs/2208.10570}, 
}

@book{milnor1974characteristic,
   title={Characteristic Classes},
   author={Milnor, John W. and Stasheff, James D.},
   year={1974},
   publisher={Princeton University Press},
   series={Annals of Mathematics Studies},
   number={76}
 }

@unpublished{Zakharevich2024,
    author = {Zakharevich, Inna},
    title = {Topological K-theory and Characteristic Classes: A Homotopical Perspective},
    note = {\url{https://pi.math.cornell.edu/~zakh}}
}

@inproceedings{DBLP:conf/compgeom/ScoccolaP23,
  author       = {Luis Scoccola and
                  Jose A. Perea},
  editor       = {Erin W. Chambers and
                  Joachim Gudmundsson},
  title        = {FibeRed: Fiberwise Dimensionality Reduction of Topologically Complex
                  Data with Vector Bundles},
  booktitle    = {39th International Symposium on Computational Geometry, SoCG 2023,
                  Dallas, Texas, USA, June 12-15, 2023},
  series       = {LIPIcs},
  volume       = {258},
  pages        = {56:1--56:18},
  publisher    = {Schloss Dagstuhl - Leibniz-Zentrum f{\"{u}}r Informatik},
  year         = {2023},
  url          = {https://doi.org/10.4230/LIPIcs.SoCG.2023.56},
  doi          = {10.4230/LIPICS.SOCG.2023.56},
  bibsource    = {dblp computer science bibliography, https://dblp.org}
}

@article{10.1007/s00454-017-9927-2,
author = {Perea, Jose A.},
title = {Multiscale Projective Coordinates via Persistent Cohomology of Sparse Filtrations},
year = {2018},
issue_date = {January   2018},
publisher = {Springer-Verlag},
address = {Berlin, Heidelberg},
volume = {59},
number = {1},
issn = {0179-5376},
url = {https://doi.org/10.1007/s00454-017-9927-2},
doi = {10.1007/s00454-017-9927-2},
journal = {Discrete Comput. Geom.},
month = jan,
pages = {175–225},
numpages = {51}
}

@inproceedings{
scoccola2025coverlearninglargescaletopology,
title={Cover learning for large-scale topology representation},
author={Luis Scoccola and Uzu Lim and Heather A. Harrington},
booktitle={Forty-second International Conference on Machine Learning},
year={2025},
url={https://openreview.net/forum?id=fgxobShivL}
}

@article{doi:10.1137/22M1522711,
author = {Lim, Uzu and Oberhauser, Harald and Nanda, Vidit},
title = {Tangent Space and Dimension Estimation with the Wasserstein Distance},
journal = {SIAM Journal on Applied Algebra and Geometry},
volume = {8},
number = {3},
pages = {650-685},
year = {2024},
doi = {10.1137/22M1522711},

URL = { 
    
        https://doi.org/10.1137/22M1522711
    
    

},
eprint = { 
    
        https://doi.org/10.1137/22M1522711
    
    

}
}

@ARTICLE{11301037,
  author={Clémot, Mattéo and Digne, Julie and Tierny, Julien},
  journal={IEEE Transactions on Visualization and Computer Graphics}, 
  title={Topological Autoencoders++: Fast and Accurate Cycle-Aware Dimensionality Reduction}, 
  year={2025},
  volume={},
  number={},
  pages={1-17},
  doi={10.1109/TVCG.2025.3644671}}

@article{Tinarrage_2021,
   title={Computing persistent Stiefel–Whitney classes of line bundles},
   volume={6},
   ISSN={2367-1734},
   url={http://dx.doi.org/10.1007/s41468-021-00080-4},
   DOI={10.1007/s41468-021-00080-4},
   number={1},
   journal={Journal of Applied and Computational Topology},
   publisher={Springer Science and Business Media LLC},
   author={Tinarrage, Raphaël},
   year={2021},
   month=nov, pages={65–125} }

@misc{scoccola2023fiberedfiberwisedimensionalityreduction,
      title={FibeRed: Fiberwise Dimensionality Reduction of Topologically Complex Data with Vector Bundles}, 
      author={Luis Scoccola and Jose A. Perea},
      year={2023},
      eprint={2206.06513},
      archivePrefix={arXiv},
      primaryClass={cs.CG},
      url={https://arxiv.org/abs/2206.06513}, 
}

@phdthesis{yiding_2009,
    author = {Yi Ding},
    title = {Topological algorithms mapping point cloud data},
    school = {
Stanford University. Institute for Computational and Mathematical Engineering},
    year =  {2009}
}

@inproceedings{10.1109/SMI.2007.25,
author = {Zomorodian, Afra and Carlsson, Gunnar},
title = {Localized Homology},
year = {2007},
isbn = {0769528155},
publisher = {IEEE Computer Society},
address = {USA},
url = {https://doi.org/10.1109/SMI.2007.25},
doi = {10.1109/SMI.2007.25},
booktitle = {Proceedings of the IEEE International Conference on Shape Modeling and Applications 2007},
pages = {189–198},
numpages = {10},
series = {SMI '07}
}

@book{hatcher2002algebraic,
  address = {Cambridge},
  author = {Hatcher, Allen},
  isbn = {0-521-79160-X; 0-521-79540-0},
  mrclass = {55-01 (55-00)},
  mrnumber = {1867354 (2002k:55001)},
  mrreviewer = {Donald W. Kahn},
  pages = {xii+544},
  publisher = {Cambridge University Press},
  title = {Algebraic topology},
  username = {mwpb479},
  year = 2002
}

@inproceedings{singhMC07,
  author = {Singh, Gurjeet and Mémoli, Facundo and Carlsson, Gunnar E.},
  booktitle = {PBG@Eurographics},
  editor = {Botsch, Mario and Pajarola, Renato and Chen, Baoquan and Zwicker, Matthias},
  ee = {https://doi.org/10.2312/SPBG/SPBG07/091-100},
  isbn = {978-3-905673-51-7},
  pages = {91-100},
  publisher = {Eurographics Association},
  title = {Topological Methods for the Analysis of High Dimensional Data Sets and 3D Object Recognition.},
  year = {2007}
}

@article{Chen2011,
  title = {Hardness Results for Homology Localization},
  volume = {45},
  ISSN = {1432-0444},
  url = {http://dx.doi.org/10.1007/s00454-010-9322-8},
  DOI = {10.1007/s00454-010-9322-8},
  number = {3},
  journal = {Discrete \& Computational Geometry},
  publisher = {Springer Science and Business Media LLC},
  author = {Chen,  Chao and Freedman,  Daniel},
  year = {2011},
  month = jan,
  pages = {425–448}
}

@book{horn2012matrix,
  title = {Matrix Analysis},
  author = {Horn, Roger A. and Johnson, Charles R.},
  edition = {2nd},
  year = {2012},
  publisher = {Cambridge University Press},
  isbn = {978-0-521-54823-6}
}

@book{BottTu,
  title = {Differential Forms in Algebraic Topology},
  author = {Bott, Raoul and Tu, Loring W.},
  year = {1982},
  publisher = {Springer-Verlag},
  series = {Graduate Texts in Mathematics},
  number = {82},
  isbn = {978-0-387-90613-3},
  doi = {10.1007/978-1-4757-3951-0}
}

@article{doi:10.1137/24M1641312,
author = {Alvarado, Enrique and Belton, Robin and Fischer, Emily and Lee, Kang-Ju and Palande, Sourabh and Percival, Sarah and Purvine, Emilie},
title = {G-Mapper: Learning a Cover in the Mapper Construction},
journal = {SIAM Journal on Mathematics of Data Science},
volume = {7},
number = {2},
pages = {572-596},
year = {2025},
doi = {10.1137/24M1641312},

URL = { 
    
        https://doi.org/10.1137/24M1641312
    
    

},
eprint = { 
    
        https://doi.org/10.1137/24M1641312
    
    

}
}

@article{Scoccola_Perea_2023, title={Approximate and discrete Euclidean vector bundles}, volume={11}, DOI={10.1017/fms.2023.16}, journal={Forum of Mathematics, Sigma}, author={Scoccola, Luis and Perea, Jose A.}, year={2023}, pages={e20}}

@article{Karoubi2017,
  title = {On the covering type of a space},
  volume = {62},
  ISSN = {2309-4672},
  url = {http://dx.doi.org/10.4171/LEM/62-3/4-4},
  DOI = {10.4171/lem/62-3/4-4},
  number = {3},
  journal = {L’Enseignement Mathématique},
  publisher = {European Mathematical Society - EMS - Publishing House GmbH},
  author = {Karoubi,  Max and Weibel,  Charles A.},
  year = {2017},
  month = aug,
  pages = {457–474}
}

@inproceedings{
robinett2026atlasbased,
title={Atlas-based Manifold Representations for Interpretable Riemannian Machine Learning},
author={Ryan Allen Robinett and Sophia Madejski and Kyle Ruark and Samantha J. Riesenfeld and Lorenzo Orecchia},
booktitle={The 29th International Conference on Artificial Intelligence and Statistics},
year={2026},
url={https://openreview.net/forum?id=yTGh29HLlp}
}

@misc{yu2026learninggeometrytopologymultichart,
      title={Learning Geometry and Topology via Multi-Chart Flows}, 
      author={Hanlin Yu and Søren Hauberg and Marcelo Hartmann and Arto Klami and Georgios Arvanitidis},
      year={2026},
      eprint={2505.24665},
      archivePrefix={arXiv},
      primaryClass={cs.LG},
      url={https://arxiv.org/abs/2505.24665}, 
}

@article{Niyogi2008,
  author    = {Niyogi, Partha and Smale, Stephen and Weinberger, Shmuel},
  title     = {Finding the Homology of Submanifolds with High Confidence from Random Samples},
  journal   = {Discrete \& Computational Geometry},
  volume    = {39},
  number    = {1--3},
  pages     = {419--441},
  year      = {2008},
  doi       = {10.1007/s00454-008-9053-2},
  publisher = {Springer}
}

\begin{appendices}
\section{Proofs of Section~\ref{sec:stability}}\label{app:proofs}

\subsection{Auxiliary lemmas}

We first record two technical lemmas extracted from the proof of Theorem~\ref{thm:sign-cocycle-stability}, used inside and around it. The first is a sharp determinant perturbation bound; the second is the Lipschitz estimate $K_{\det}$ for $p \mapsto \det g_{ji}(p)$ used in Step~4.

\begin{lemma}[Determinant perturbation]\label{lem:det-perturbation}
For any $A, B \in \R^{d \times d}$,
\[
|\det(A + B) - \det(A)| \leq d \cdot \|B\|_{\mathrm{op}} \cdot (\|A\|_{\mathrm{op}} + \|B\|_{\mathrm{op}})^{d-1}.
\]
\end{lemma}

\begin{proof}
By multilinearity of the determinant in the columns,
\[
\det(A + B) - \det(A) = \sum_{l=1}^{d} \det\bigl[a_1 \mid \cdots \mid a_{l-1} \mid b_l \mid (a_{l+1} + b_{l+1}) \mid \cdots \mid (a_d + b_d)\bigr],
\]
where $a_l$, $b_l$ denote the $l$-th columns of $A$ and $B$ respectively. Each of the $d$ terms is the determinant of a matrix whose $l$-th column has norm at most $\|B\|_{\mathrm{op}}$ and whose other columns have norm at most $\|A\|_{\mathrm{op}} + \|B\|_{\mathrm{op}}$. By Hadamard's inequality~\cite[Theorem~3.3.16]{horn2012matrix}, the absolute value of each such determinant is at most $\|B\|_{\mathrm{op}} \cdot (\|A\|_{\mathrm{op}} + \|B\|_{\mathrm{op}})^{d-1}$. Summing yields the claim.
\end{proof}

\begin{lemma}[Lipschitz bound for $g_{ji}$ and $\det g_{ji}$]\label{lem:Ldet}
Let $\mathcal{A}$ be a $C^1$ approximate autoencoder atlas satisfying conditions~(i)--(iv) of Theorem~\ref{thm:sign-cocycle-stability}, and set
\[
C_g \coloneqq L_E\,(L_E L_D' + L_E' L_D^2).
\]
Then for all overlapping pairs $(i, j)$ and all $p, q \in O_i \cap O_j$,
\[
\|g_{ji}(p) - g_{ji}(q)\|_{\mathrm{op}} \leq C_g\,\|p - q\|,
\]
and consequently
\[
|\det g_{ji}(p) - \det g_{ji}(q)| \leq d\,C_g\,\|p - q\|\,\bigl(L_E L_D + C_g\,\|p - q\|\bigr)^{d-1}.
\]
\end{lemma}

\begin{proof}
Write $g_{ji}(p) = A(p) \cdot B(p)$ with $A(p) \coloneqq d(E_j)_{D_i(E_i(p))}$ and $B(p) \coloneqq d(D_i)_{E_i(p)}$. We first bound the Lipschitz constants of $A$ and $B$ on $O_i \cap O_j$.

For $A$: the map $p \mapsto D_i(E_i(p))$ is Lipschitz with constant at most $L_D L_E$, and $E_j$ has derivative Lipschitz constant $L_E'$, so
\[
\|A(p) - A(q)\|_{\mathrm{op}} \leq L_E' \cdot L_D L_E \|p - q\| = L_E' L_D L_E \|p - q\|.
\]
For $B$: $E_i$ is Lipschitz with constant $L_E$, and $D_i$ has derivative Lipschitz constant $L_D'$, so
\[
\|B(p) - B(q)\|_{\mathrm{op}} \leq L_D' L_E \|p - q\|.
\]
Hence the product $g_{ji} = A \cdot B$ satisfies, by the product rule for Lipschitz functions,
\begin{align*}
\|g_{ji}(p) - g_{ji}(q)\|_{\mathrm{op}}
&\leq \|A(p) - A(q)\|_{\mathrm{op}}\, \|B(q)\|_{\mathrm{op}} + \|A(p)\|_{\mathrm{op}}\, \|B(p) - B(q)\|_{\mathrm{op}} \\
&\leq (L_E' L_D L_E)(L_D)\,\|p - q\| + (L_E)(L_D' L_E)\,\|p - q\| \\
&= L_E\,(L_E L_D' + L_E' L_D^2)\,\|p - q\|.
\end{align*}
Finally, by Lemma~\ref{lem:det-perturbation} applied with $A = g_{ji}(q)$ (so $\|A\|_{\mathrm{op}} \leq L_E L_D$) and $A + B = g_{ji}(p)$,
\[
|\det g_{ji}(p) - \det g_{ji}(q)|
\leq d\,\|g_{ji}(p) - g_{ji}(q)\|_{\mathrm{op}}\,
\bigl(L_E L_D + \|g_{ji}(p) - g_{ji}(q)\|_{\mathrm{op}}\bigr)^{d-1}.
\]
Substituting the operator-norm bound $\|g_{ji}(p) - g_{ji}(q)\|_{\mathrm{op}} \leq C_g\,\|p-q\|$ gives the claim. \qedhere
\end{proof}

\subsection{Proof of Theorem~\ref{thm:sign-cocycle-stability}}\label{proof:sign-cocycle-stability}

\begin{proof}[Proof of Theorem~\ref{thm:sign-cocycle-stability}]
Fix $x \in U_i \cap U_j \cap U_k$ and set $y \coloneqq D_i(E_i(x))$. By Remark~\ref{rem:sign-constancy}, the sign $\omega_{ji}$ is constant on each connected component of each overlap. It suffices to show
\[
\sign(\det g_{ki}(x)) = \sign\bigl(\det(g_{kj}(y) \cdot g_{ji}(x))\bigr) = \omega_{kj}(x) \cdot \omega_{ji}(x).
\]
We do this by showing $\bigl|\det g_{ki}(x) - \det(g_{kj}(y) \cdot g_{ji}(x))\bigr| < |\det g_{ki}(x)|$, then correcting the evaluation point.

\paragraph{Step 0: Domain verification.}
We take $R=L_E L_D + 3$ in Definition~\ref{def:approximate-atlas}(5). Since $x \in U_i \cap U_j \cap U_k$ and $\|y - x\| \leq \varepsilon$, the point $y$ lies in $O_j \cap O_k$ by Definition~\ref{def:approximate-atlas}(5) (the closed $R\varepsilon$-neighborhood condition with $R = L_E L_D + 3 \geq 1$). In particular, the encoders $E_j, E_k$ are $C^1$ at $y$, so the Jacobians $g_{kj}(y) = d(T_{kj})_{E_j(y)}$ and $g_{ji}(x) = d(T_{ji})_{E_i(x)}$ are well defined. Moreover, $\Phi_j$ is $C^1$ at $y$ since $y \in O_j$. By Lemma~\ref{lem:off-manifold-reconstruction}, $\|\Phi_j(y) - y\| \leq (L_E L_D + 2)\varepsilon$, hence $\|\Phi_j(y) - x\| \leq (L_E L_D + 3)\varepsilon = R\varepsilon$, so $\Phi_j(y) \in O_k$ by Definition~\ref{def:approximate-atlas}(5). The encoder $E_k$ is therefore $C^1$ at $\Phi_j(y)$, justifying the appearance of $d(E_k)_{\Phi_j(y)}$ in Lemma~\ref{lemma:jacobian-factorisation}(2).

\paragraph{Step 1: Bounding the Jacobian perturbation.}
By Lemma~\ref{lemma:jacobian-factorisation},
\[
g_{ki}(x) = Q \cdot P, \qquad g_{kj}(y) \cdot g_{ji}(x) = Q' R P,
\]
where
\[
P \coloneqq d(D_i)_{E_i(x)} \in \R^{N \times d}, \quad
Q \coloneqq d(E_k)_y \in \R^{d \times N}, \quad
Q' \coloneqq d(E_k)_{\Phi_j(y)} \in \R^{d \times N}, \quad
R \coloneqq d(\Phi_j)_y \in \R^{N \times N}.
\]
Writing $Q' = Q + \Delta_Q$ and $\Delta_R \coloneqq R - I_N$,
\begin{equation}\label{eq:Delta-decomposition}
Q'RP = QP + \underbrace{Q\,\Delta_R\, P + \Delta_Q\,(I_N + \Delta_R)\,P}_{=:\,\Delta},
\end{equation}
so $g_{kj}(y) \cdot g_{ji}(x) = g_{ki}(x) + \Delta$. It remains to bound $\|\Delta\|_{\mathrm{op}}$.

\medskip

\textbf{Step 1a: The restricted action of $\Delta_R$ on $\mathrm{range}(P)$.}
A naive bound on $\|Q\,\Delta_R\,P\|_{\mathrm{op}}$ using $\|\Delta_R\|_{\mathrm{op}}$ fails: since $\Phi_j = D_j \circ E_j$ factors through $\R^d$, the differential $R$ has rank at most $d < N$ and annihilates a subspace of dimension at least $N - d$, forcing $\|\Delta_R\|_{\mathrm{op}} \geq 1$ regardless of reconstruction quality. The key observation is that only the restricted action of $\Delta_R$ on $\mathrm{range}(P)$ enters the product $Q\,\Delta_R\,P$, and vectors in $\mathrm{range}(P)$ are nearly tangential to~$M$.

Fix a unit $v \in \R^d$ and set $w \coloneqq Pv = d(D_i)_{E_i(x)}\,v$. The differential reconstruction condition (Definition~\ref{def:approximate-atlas}(4)) for chart~$i$ states $\|d(\Phi_i)_x|_{T_xM} - \Id_{T_xM}\|_{\mathrm{op}} \leq \eta$. Since $\eta < 1$, $d(\Phi_i)_x|_{T_xM} = P \circ d(E_i)_x|_{T_xM}$ is invertible. In particular $d(E_i)_x|_{T_xM} \colon T_xM \to \R^d$ is a bijection, so there exists a unique $u \in T_xM$ with $d(E_i)_x\,u = v$. Then
\[
w = P\,v = d(D_i)_{E_i(x)}\,d(E_i)_x\,u = d(\Phi_i)_x\,u,
\]
and condition~(4) gives $\|w - u\| \leq \eta\,\|u\|$.

Decompose $w = w_T + w_\perp$ with $w_T \in T_xM$ and $w_\perp \perp T_xM$. The normal component obeys
\begin{equation}\label{eq:normal-bound}
\|w_\perp\| = \|\Pi_{T_xM}^\perp(w - u)\| \leq \|w - u\| \leq \eta\,\|u\| \leq \frac{\eta}{1 - \eta}\,\|w\|,
\end{equation}
using $\|w\| \geq \|u\| - \|w - u\| \geq (1 - \eta)\|u\|$ in the last step.

\medskip

\textbf{Step 1b: Bounding $\|\Delta_R\,w\|$ via the tangent--normal decomposition.}
We bound $\|\Delta_R\,w\| = \|(d(\Phi_j)_y - I_N)\,w\|$ by treating $w_T$ and $w_\perp$ separately.

\emph{Tangential component.} Since $x \in U_j$ and $w_T \in T_xM$, condition~(4) for chart~$j$ gives $\|(d(\Phi_j)_x - I_N)\,w_T\| \leq \eta\,\|w_T\|$. To pass from $x$ to $y$, use the derivative Lipschitz constant $L_\Phi' = L_D' L_E^2 + L_D L_E'$ of $p \mapsto d(\Phi_j)_p$ on $O_j$ (Remark~\ref{rem:Lphi-constant}): since $\|y - x\| \leq \varepsilon$,
\[
\|(d(\Phi_j)_y - I_N)\,w_T\|
\leq \|(d(\Phi_j)_x - I_N)\,w_T\| + \|d(\Phi_j)_y - d(\Phi_j)_x\|_{\mathrm{op}}\,\|w_T\|
\leq (\eta + L_\Phi'\,\varepsilon)\,\|w_T\|
\leq (\eta + L_\Phi'\,\varepsilon)\,\|w\|,
\]
where the last step uses $\|w_T\| \leq \|w\|$, since $w_T$ is the orthogonal projection of $w$ onto $T_xM$.

\emph{Normal component.} For directions normal to~$M$ no reconstruction guarantee is available; we use the crude bound $\|d(\Phi_j)_y\|_{\mathrm{op}} \leq L_E L_D$:
\[
\|(d(\Phi_j)_y - I_N)\,w_\perp\|
\leq (L_E L_D + 1)\,\|w_\perp\|
\leq \frac{(L_E L_D + 1)\,\eta}{1 - \eta}\,\|w\|,
\]
where the last step uses~\eqref{eq:normal-bound}.

\emph{Combining.} Adding the two contributions,
\begin{equation}\label{eq:restricted-DeltaR}
\|\Delta_R\,w\|
\leq \biggl(\eta + L_\Phi'\varepsilon + \frac{(L_E L_D + 1)\eta}{1 - \eta}\biggr)\|w\|
\leq \biggl(\frac{(L_E L_D + 2)\eta}{1 - \eta} + L_\Phi'\varepsilon\biggr)\|w\|
= \eta_{\mathrm{eff}}\,\|w\|.
\end{equation}
The simplification step uses
\[
\eta + \frac{(L_E L_D + 1)\eta}{1 - \eta} = \frac{\eta(1 - \eta) + (L_E L_D + 1)\eta}{1 - \eta} = \frac{(L_E L_D + 2 - \eta)\eta}{1 - \eta} \leq \frac{(L_E L_D + 2)\eta}{1 - \eta},
\]
valid since $\eta \geq 0$. Since $\|w\| = \|Pv\| \leq L_D\|v\|$,
\begin{equation}\label{eq:DeltaR-P-bound}
\|\Delta_R\,P\,v\| \leq \eta_{\mathrm{eff}}\,L_D\,\|v\| \qquad \text{for all } v \in \R^d.
\end{equation}

\medskip

\textbf{Step 1c: Bounding $\|\Delta_Q\|_{\mathrm{op}}$.}
The matrices $Q = d(E_k)_y$ and $Q' = d(E_k)_{\Phi_j(y)}$ differ because $E_k$ is evaluated at the shifted point $\Phi_j(y)$. By condition~(ii) and Lemma~\ref{lem:off-manifold-reconstruction},
\begin{equation}\label{eq:DeltaQ-bound}
\|\Delta_Q\|_{\mathrm{op}} = \|d(E_k)_{\Phi_j(y)} - d(E_k)_y\|_{\mathrm{op}}
\leq L_E'\,\|\Phi_j(y) - y\| \leq L_E'\,\widetilde{\varepsilon},
\end{equation}
where $\widetilde{\varepsilon} = (L_E L_D + 2)\varepsilon$.

\medskip

\textbf{Step 1d: Assembling the bound on $\|\Delta\|_{\mathrm{op}}$.}
Returning to~\eqref{eq:Delta-decomposition},
\[
\|Q\,\Delta_R\,P\|_{\mathrm{op}} \leq L_E\,\eta_{\mathrm{eff}}\,L_D,
\]
by \eqref{eq:DeltaR-P-bound} and $\|Q\|_{\mathrm{op}} \leq L_E$. For the second term, $\|(I_N + \Delta_R)\,P\,v\| \leq (1 + \eta_{\mathrm{eff}})\,L_D\|v\|$ by~\eqref{eq:DeltaR-P-bound}, so
\[
\|\Delta_Q\,(I_N + \Delta_R)\,P\|_{\mathrm{op}}
\leq L_E'\,\widetilde{\varepsilon}\,(1 + \eta_{\mathrm{eff}})\,L_D.
\]
Therefore
\begin{equation}\label{eq:Delta-bound}
\|\Delta\|_{\mathrm{op}}
\leq L_E\,\eta_{\mathrm{eff}}\,L_D + L_E'\,\widetilde{\varepsilon}\,(1 + \eta_{\mathrm{eff}})\,L_D
= \Gamma.
\end{equation}

\paragraph{Step 2: Determinant perturbation.}
We have $g_{kj}(y) \cdot g_{ji}(x) = g_{ki}(x) + \Delta$ with $\|\Delta\|_{\mathrm{op}} \leq \Gamma$. By Lemma~\ref{lem:det-perturbation} with $A = g_{ki}(x)$, $B = \Delta$, and $\|A\|_{\mathrm{op}} \leq L_E L_D$,
\begin{equation}\label{eq:det-perturbation-bound}
\bigl|\det\bigl(g_{kj}(y) \cdot g_{ji}(x)\bigr) - \det g_{ki}(x)\bigr|
\leq d \cdot \Gamma \cdot (L_E L_D + \Gamma)^{d-1}.
\end{equation}

\paragraph{Step 3: Sign agreement.}
By non-degeneracy, $|\det g_{ki}(x)| \geq \delta$, and \eqref{eq:stability-condition} gives $d\,\Gamma\,(L_E L_D + \Gamma)^{d-1} < \delta$. Hence the right-hand side of~\eqref{eq:det-perturbation-bound} is strictly less than $|\det g_{ki}(x)|$, forcing
\[
\sign\bigl(\det(g_{kj}(y) \cdot g_{ji}(x))\bigr) = \sign(\det g_{ki}(x)),
\]
and by multiplicativity of $\sign \circ \det$,
\[
\sign(\det g_{ki}(x)) = \sign(\det g_{kj}(y)) \cdot \sign(\det g_{ji}(x)).
\]

\paragraph{Step 4: Evaluation point correction.}
It remains to show $\sign(\det g_{kj}(y)) = \omega_{kj}(x)$. By Step~0, both $x$ and $y$ lie in $O_j \cap O_k$. Consider the segment $\gamma(t) \coloneqq (1-t)x + ty$ for $t \in [0, 1]$. Since $\|\gamma(t) - x\| \leq \|y - x\| \leq \varepsilon$ and $x \in U_j \cap U_k$, every $\gamma(t)$ satisfies $\mathrm{dist}(\gamma(t), U_j \cap U_k) \leq \varepsilon$, so $\gamma(t) \in O_j \cap O_k$ by Definition~\ref{def:approximate-atlas}(5). On this set $g_{kj}$ is well defined and continuous in its base point by Lemma~\ref{lem:Ldet}. Condition~(v) of the theorem ($\varepsilon < \tau(M)$) further ensures that the segment lies within the open $\tau(M)$-tube of $M$, where the nearest-point projection is single-valued; this is what justifies the use of the segment $[x, y]$ as a continuous path inside a regular tubular neighbourhood.

The function $t \mapsto \det g_{kj}(\gamma(t))$ is continuous on $[0,1]$. By Lemma~\ref{lem:Ldet}, using $\|\gamma(t) - x\| \leq \varepsilon$ and the monotonicity of $\rho \mapsto d\,C_g\,\rho\,(L_E L_D + C_g\,\rho)^{d-1}$,
\[
|\det g_{kj}(\gamma(t)) - \det g_{kj}(x)|
\leq d\,C_g\,\|\gamma(t) - x\|\,\bigl(L_E L_D + C_g\,\|\gamma(t) - x\|\bigr)^{d-1}
\leq K_{\det} < \delta
\]
by~\eqref{eq:stability-condition}, so $|\det g_{kj}(\gamma(t))| \geq \delta - K_{\det} > 0$ along the entire segment. Since the determinant is continuous and never vanishes on $\gamma$, its sign is constant:
\[
\sign(\det g_{kj}(y)) = \sign(\det g_{kj}(x)) = \omega_{kj}(x).
\]
Combining with Step~3, $\omega_{ki}(x) = \omega_{kj}(x) \cdot \omega_{ji}(x)$.
\end{proof}

\begin{remark}[Interpretation of the two branches in~\eqref{eq:stability-condition}]
The first branch $d\,\Gamma\,(L_E L_D + \Gamma)^{d-1} < \delta$ controls the cocycle defect by mixing $\eta$ and $\varepsilon$ through $\Gamma$; it is the binding condition when reconstruction is the primary source of error. The second branch $K_{\det} < \delta$ controls the evaluation-point shift in Step~4 and depends only on $\varepsilon$ (which is already absorbed into the definition of $K_{\det}$). Neither branch implies the other.
\end{remark}

\begin{remark}[Interpretation of $\eta_{\mathrm{eff}}$]
In the exact setting ($\varepsilon = \eta = 0$) the reconstruction map $\Phi_j$ is the identity on $M$, and $d(\Phi_j)_x$ acts as the identity on $T_xM$ while annihilating normal directions. The cocycle condition holds exactly because all relevant vectors lie in tangent spaces. In the approximate setting, the decoder differential $P = d(D_i)_{E_i(x)}$ produces vectors $w = Pv$ that are \emph{nearly} tangential but have a small normal component of relative size $\eta/(1 - \eta)$. The reconstruction differential $d(\Phi_j)_y - I_N$ controls the tangential part by $\eta + O(\varepsilon)$ but amplifies the normal part by up to $L_E L_D + 1$. The effective error $\eta_{\mathrm{eff}}$ accounts for this amplification: it equals the tangent-restricted $\eta$ plus a correction of order $(L_E L_D)\eta/(1-\eta)$ from the normal component, plus an off-manifold evaluation correction of order $L_\Phi'\varepsilon$.
\end{remark}

\subsection{Proof of Theorem~\ref{thm:local-stability}}\label{proof:local-stability}

The proof is a direct adaptation of the proof of Theorem~\ref{thm:sign-cocycle-stability}, with each uniform constant replaced by its per-triple version. We make this explicit through the following dictionary.

\begin{proof}[Proof of Theorem~\ref{thm:local-stability}]
Fix $x \in U_i \cap U_j \cap U_k$ and set $y \coloneqq D_i(E_i(x))$, as in the global proof. Working through Steps~0--4 of the proof of Theorem~\ref{thm:sign-cocycle-stability} verbatim, we record at each invocation which hypothesis is used.

\medskip

\noindent
\emph{Dictionary of substitutions.}
\begin{center}
\centering
\small
\begin{tabularx}{\textwidth}{lXX}
\toprule
Step & Uniform invocation & Per-triple replacement \\
\midrule
Step~0 & $y \in O_j \cap O_k$ via Def.~\ref{def:approximate-atlas}(5) 
& same, since condition~(v)$_{ijk}$ is global \\

Step~0 & $\Phi_j(y) \in O_k$ 
& Lemma~\ref{lem:off-manifold-reconstruction} with $L_E^{(ijk)}, L_D^{(ijk)}$ \\

Step~1 & factorisation Lemma~\ref{lemma:jacobian-factorisation} 
& unchanged: only uses charts $i, j, k$ \\

Step~1a & $\|d(D_i)\|_{\mathrm{op}} \leq L_D$ 
& (iii)$_{ijk}$ for $\ell = i$ \\

Step~1a & condition (4) for chart $i$, with bound $\eta$ 
& $\eta_i \leq \eta^{(ijk)}$ \\

Step~1b & condition (4) for chart $j$, with bound $\eta$ 
& $\eta_j \leq \eta^{(ijk)}$ \\

Step~1b & $L_\Phi'$ for $\Phi_j$ 
& $L_\Phi'^{(ijk)} \coloneqq (L_D')^{(ijk)}(L_E^{(ijk)})^2 + L_D^{(ijk)}(L_E')^{(ijk)}$ \\

Step~1b & $\|d(\Phi_j)\|_{\mathrm{op}} \leq L_E L_D$ 
& $L_E^{(ijk)} L_D^{(ijk)}$ via (i)$_{ijk}$, (iii)$_{ijk}$ for $\ell = j$ \\

Step~1c & $\|\Delta_Q\|_{\mathrm{op}} \leq L_E'\,\widetilde\varepsilon$ 
& (ii)$_{ijk}$ for $\ell = k$ \\

Step~1c & $\|Q\|_{\mathrm{op}} \leq L_E$ 
& (i)$_{ijk}$ for $\ell = k$ \\

Step~2 & $\|g_{ki}\|_{\mathrm{op}} \leq L_E L_D$ 
& (i)$_{ijk}$ for $\ell = k$, (iii)$_{ijk}$ for $\ell = i$ \\

Step~3 & $|\det g_{ki}(x)| \geq \delta$ 
& $|\det g_{ki}(x)| \geq \delta^{(ijk)}$ \\

Step~4 & Lemma~\ref{lem:Ldet} for $\det g_{kj}$ 
& same lemma with constants $L_E^{(ijk)}, L_E'^{(ijk)}, L_D^{(ijk)}, L_D'^{(ijk)}$ \\

Step~4 & $|\det g_{kj}(x)| \geq \delta$ 
& $|\det g_{kj}(x)| \geq \delta^{(ijk)}$ \\
\bottomrule
\end{tabularx}
\end{center}

Inspecting the table, the chart~$k$ enters only through its encoder $E_k$ (via (i)$_{ijk}$ and (ii)$_{ijk}$ for $\ell = k$), never through its decoder $D_k$ or its reconstruction map $\Phi_k$. Likewise, the differential reconstruction errors that enter the proof are $\eta_i$ (Step~1a) and $\eta_j$ (Step~1b), and these are absorbed into the single per-triple quantity $\eta^{(ijk)} = \max(\eta_i, \eta_j)$.

Performing all substitutions, Steps~1--4 yield
\[
|\det(g_{kj}(y) \cdot g_{ji}(x)) - \det g_{ki}(x)| \leq d\,\Gamma^{(ijk)}\,\bigl(L_E^{(ijk)} L_D^{(ijk)} + \Gamma^{(ijk)}\bigr)^{d-1}
\]
and
\[
|\det g_{kj}(\gamma(t)) - \det g_{kj}(x)| \leq K_{\det}^{(ijk)}.
\]
Both right-hand sides are strictly less than $\delta^{(ijk)}$ by~\eqref{eq:local-stability-condition}. The sign-agreement argument of Steps~3--4 then gives $\omega_{ki}(x) = \omega_{kj}(x) \cdot \omega_{ji}(x)$ on the triangle $(i, j, k)$.

The final statement (global cocycle when the per-triple condition holds for every triangle) follows because the cocycle identity is required separately on each triple intersection.
\end{proof}

\subsection{Proof of Theorem~\ref{thm:w1-agreement}}\label{proof:agreement}

\begin{proof}[Proof of Theorem~\ref{thm:w1-agreement}]
The strategy is a one-parameter interpolation between the exact compatible atlas $\mathcal{A}_0$ and the learned atlas $\mathcal{A}$, with uniform non-degeneracy maintained along the path.

\paragraph{Step 0: Common domain.}
By hypothesis~(i), $\widetilde{Z}_i \subset \R^d$ is an open set containing the $\mu$-neighborhood of $Z_i^0$ on which both $D_i^0$ and $D_i$ are $C^1$. Since $\|E_i - E_i^0\|_{C^0(U_i)} \leq \mu$, we have $Z_i \subset \widetilde{Z}_i$, and for any $t \in [0,1]$ and $x \in U_i$ the point $E_i^t(x) = (1-t)E_i^0(x) + tE_i(x)$ lies on a segment of length $\leq \mu$ originating in $Z_i^0$, hence in $\widetilde{Z}_i$. The interpolated decoders $D_i^t \coloneqq (1-t)\,D_i^0 + t\,D_i$ are $C^1$ on $\widetilde{Z}_i$ for all $t \in [0,1]$. Since $\|E_i - E_i^0\|_{C^0(U_i)} \leq \mu$, the interpolated encoders $E_i^t \coloneqq (1-t)E_i^0 + t E_i$ satisfy $\|E_i^t - E_i^0\|_{C^0(U_i)} \leq t\mu \leq \mu$, so $E_i^t(U_i) \subset \widetilde{Z}_i$ provided $\widetilde{Z}_i$ contains the $\mu$-neighborhood of $Z_i^0$. (This last inclusion is automatic in practice: neural network decoders are $C^1$ on all of $\R^d$, and the exact decoder $D_i^0 = (E_i^0)^{-1}$ extends $C^1$-smoothly to a neighborhood of $\overline{Z_i^0}$ by the inverse function theorem applied to $E_i^0$ on a slight enlargement of $U_i$, where $E_i^0$ has non-degenerate Jacobian by smoothness and compactness of $\overline{U_i}$.)
\paragraph{Step 1: The interpolated family of atlases.}
With the encoders $E_i^t$ and decoders $D_i^t$ defined as in Step~0, set $\mathcal{A}_t \coloneqq \{(U_i, E_i^t, D_i^t)\}$. The transition maps $T_{ji}^t = E_j^t \circ D_i^t$ require evaluating $E_j^t$ at points $D_i^t(z) \in \R^N$ for $z \in E_i^t(U_i)$. In the exact case ($t = 0$), $D_i^0(E_i^0(x)) = x \in M$. For $t > 0$, $D_i^t(E_i^t(x))$ lies within distance $O(\mu)$ of $x \in M$. By Definition~\ref{def:approximate-atlas}(5) applied to $\mathcal{A}$, the encoders $E_j$ are defined on the closed $\varepsilon$-neighborhood of $U_j$; for $\mu$ smaller than this $\varepsilon$, all relevant evaluations lie in the domain of $E_j^t$, and the Jacobians $g_{ji}^t(x) = d(T_{ji}^t)_{E_i^t(x)}$ depend continuously on $t$.

\paragraph{Step 2: Uniform non-degeneracy along the homotopy.}
At $t = 0$, $|\det g_{ji}^0(x)| \geq \delta_0$ for all relevant $(i, j, x)$. The function $(t, x) \mapsto \det g_{ji}^t(x)$ is continuous on the compact set $[0, 1] \times \overline{U_i \cap U_j}$ for each pair $(i, j)$, hence uniformly continuous. Writing $C_0$ for the uniform Lipschitz constant of $(t, x) \mapsto \det g_{ji}^t(x)$ in $\mu$ -- which is finite and bounded by an expression in the $C^2$-norms of $\mathcal{A}_0$, the geometry of $M$, and $|I|^2$ (the number of overlapping pairs) -- the bound \eqref{eq:mu-condition} gives
\[
|\det g_{ji}^t(x) - \det g_{ji}^0(x)| \leq C_0 \cdot \mu < \frac{\delta_0}{2}
\]
for all $t \in [0, 1]$ and $x \in U_i \cap U_j$. Hence
\[
|\det g_{ji}^t(x)| \geq \frac{\delta_0}{2} > 0 \qquad \text{for all } t \in [0, 1],\; x \in U_i \cap U_j.
\]

\paragraph{Step 3: Sign constancy along the path.}
Since $\det g_{ji}^t(x)$ is continuous in $t$ and never vanishes, $\sign(\det g_{ji}^t(x))$ is constant in $t$ for each fixed $(i, j)$ and each connected component of $U_i \cap U_j$. Therefore
\[
\omega_{ji}^1(x) = \omega_{ji}^0(x) \qquad \text{on each connected component of } U_i \cap U_j.
\]

\paragraph{Step 4: Conclusion.}
The sign cocycles of $\mathcal{A}$ and $\mathcal{A}_0$ coincide pointwise, hence as elements of $\check{Z}^1(\mathcal{U}; \Z/2)$. By Proposition~\ref{prop:tangent-bundle-isomorphism}, $\mathcal{T}_{\mathcal{A}_0} \cong TM$, so $[\omega^{\mathcal{A}_0}] = w_1(\mathcal{T}_{\mathcal{A}_0}) = w_1(TM) \in \check{H}^1(\mathcal{U}; \Z/2)$ by Theorem~\ref{thm:w1-sign}. Therefore $[\omega^{\mathcal{A}}] = [\omega^{\mathcal{A}_0}] = w_1(TM)$.
\end{proof}

\begin{figure}[p]
    \centering
    \includegraphics[width=\textwidth, height=0.8\textheight, keepaspectratio]{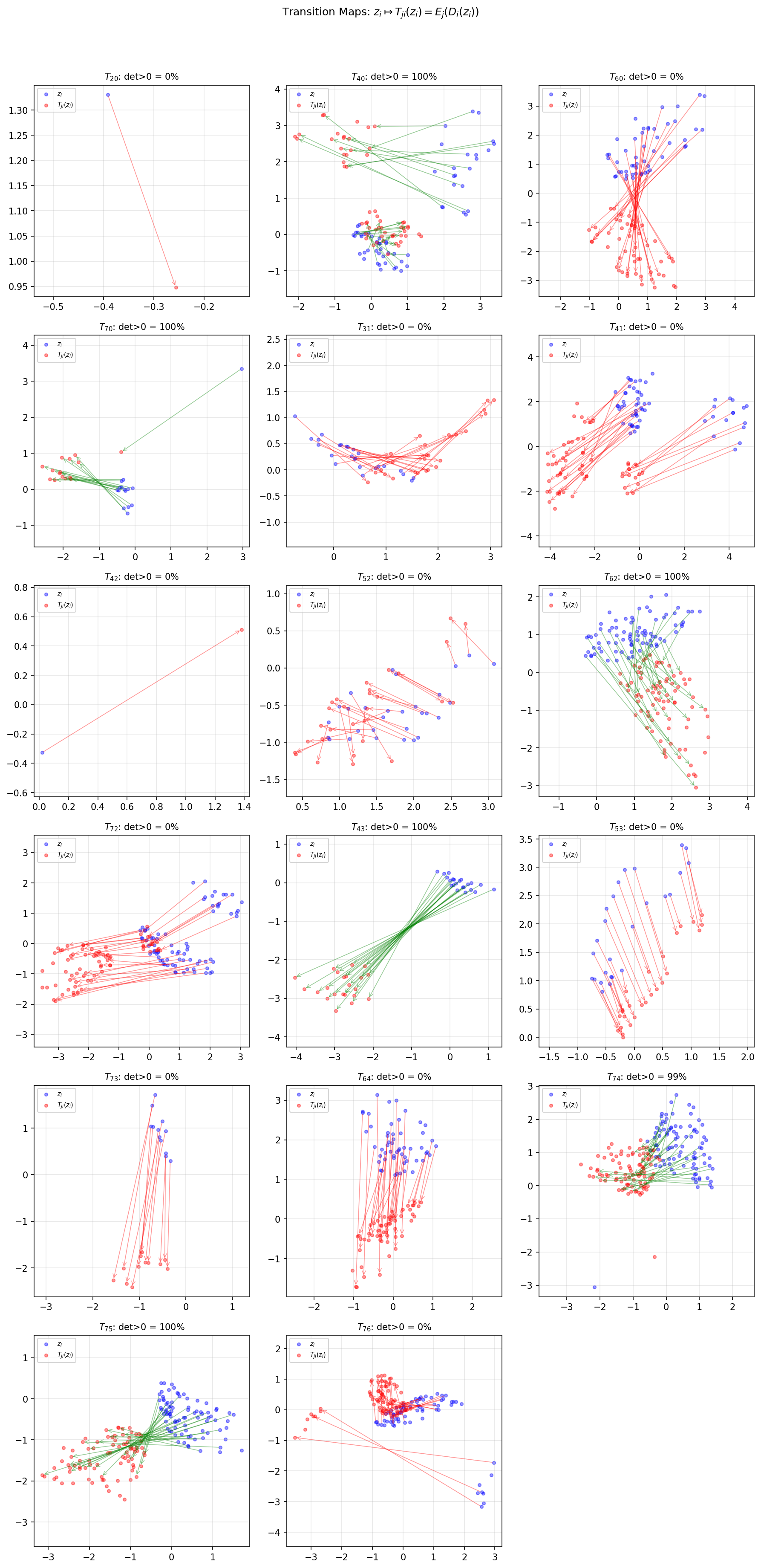}
    \caption{Transition maps $z \mapsto T_{ji}(z) = E_j(D_i(z))$ for the Klein bottle experiment (Section~\ref{sec:exp-klein}). Each panel shows the learned transition map between a pair of charts $(i, j)$, with source points in blue and their images under $T_{ji}$ in red (or green for triple-overlap regions). The title of each panel reports the percentage of overlap points with $\det g_{ji} > 0$; panels showing $0\%$ or $100\%$ indicate sign-consistent overlaps. On of the overlap display 99\% of consistency as it displays two points from another connected component.}
    \label{fig:klein-transitions_supplement}
\end{figure}
\begin{figure}[p]
    \centering
    \includegraphics[width=\textwidth, height=0.8\textheight, keepaspectratio]{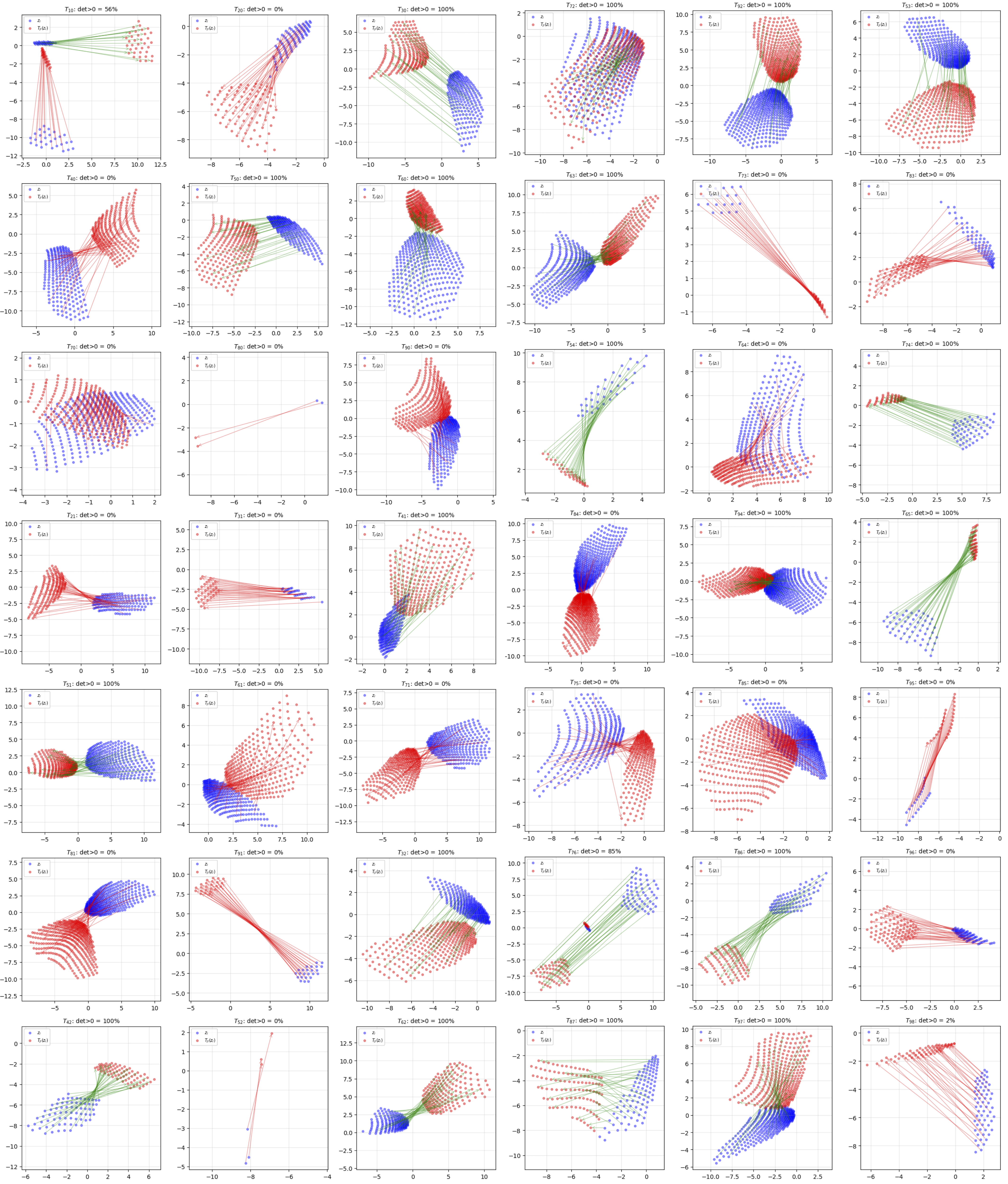}
    \caption{Transition maps $z \mapsto T_{ji}(z) = E_j(D_i(z))$ for the $\mathbb{RP}^2$ line patches experiment (Section~\ref{sec:exp-rp2}). Each panel shows the learned transition map between a pair of charts $(i, j)$, with source points in blue and their images under $T_{ji}$ in red (or green for triple-overlap regions). The title of each panel reports the percentage of overlap points with $\det g_{ji} > 0$; panels showing $0\%$ or $100\%$ indicate sign-consistent overlaps. The balanced distribution of positive and negative signs across the 25 overlap components confirms that the sign cocycle $[\omega] \in \check{H}^1(\mathcal{U}; \mathbb{Z}/2)$ is non-trivial, detecting the non-orientability of $\mathbb{RP}^2$.}
    \label{fig:rp2-transitions_supplement}
\end{figure}
\end{appendices}

\end{document}